\pdfoutput=1
\documentclass[leqno,10pt]{amsart}
\usepackage[normal]{phaine_style}
\usepackage[margin=1.175in]{geometry}
\newif\ifpersonal

\ifpersonal
\newcommand{\personal}[1]{\textcolor[rgb]{0,0,1}{(Personal: #1)}}
\newcommand{\discussion}[1]{\textcolor{violet}{(Discussion: #1)}}
\else
\newcommand{\personal}[1]{\ignorespaces}
\newcommand{\discussion}[1]{\ignorespaces}
\fi

\newcommand{\cB}{\mathcal B}

\newcommand{\cE}{\mathcal E}

\DeclareFontFamily{U}{BOONDOX-calo}{\skewchar\font=45 }
\DeclareFontShape{U}{BOONDOX-calo}{m}{n}{<-> s*[1.05] BOONDOX-r-calo}{}
\DeclareFontShape{U}{BOONDOX-calo}{b}{n}{<-> s*[1.05] BOONDOX-b-calo}{}
\DeclareMathAlphabet{\mathcalboondox}{U}{BOONDOX-calo}{m}{n}

\makeatletter
\let\save@mathaccent\mathaccent
\newcommand*\if@single[3]{%
	\setbox0\hbox{${\mathaccent"0362{#1}}^H$}%
	\setbox2\hbox{${\mathaccent"0362{\kern0pt#1}}^H$}%
	\ifdim\ht0=\ht2 #3\else #2\fi
}
\newcommand*\rel@kern[1]{\kern#1\dimexpr\macc@kerna}
\newcommand*\widebar[1]{\@ifnextchar^{{\wide@bar{#1}{0}}}{\wide@bar{#1}{1}}}
\newcommand*\wide@bar[2]{\if@single{#1}{\wide@bar@{#1}{#2}{1}}{\wide@bar@{#1}{#2}{2}}}
\newcommand*\wide@bar@[3]{%
	\begingroup
	\def\mathaccent##1##2{%
		\let\mathaccent\save@mathaccent
		\if#32 \let\macc@nucleus\first@char \fi
		\setbox\z@\hbox{$\macc@style{\macc@nucleus}_{}$}%
		\setbox\tw@\hbox{$\macc@style{\macc@nucleus}{}_{}$}%
		\dimen@\wd\tw@
		\advance\dimen@-\wd\z@
		\divide\dimen@ 3
		\@tempdima\wd\tw@
		\advance\@tempdima-\scriptspace
		\divide\@tempdima 10
		\advance\dimen@-\@tempdima
		\ifdim\dimen@>\z@ \dimen@0pt\fi
		\rel@kern{0.6}\kern-\dimen@
		\if#31
		\overline{\rel@kern{-0.6}\kern\dimen@\macc@nucleus\rel@kern{0.4}\kern\dimen@}%
		\advance\dimen@0.4\dimexpr\macc@kerna
		\let\final@kern#2%
		\ifdim\dimen@<\z@ \let\final@kern1\fi
		\if\final@kern1 \kern-\dimen@\fi
		\else
		\overline{\rel@kern{-0.6}\kern\dimen@#1}%
		\fi
	}%
	\macc@depth\@ne
	\let\math@bgroup\@empty \let\math@egroup\macc@set@skewchar
	\mathsurround\z@ \frozen@everymath{\mathgroup\macc@group\relax}%
	\macc@set@skewchar\relax
	\let\mathaccentV\macc@nested@a
	\if#31
	\macc@nested@a\relax111{#1}%
	\else
	\def\gobble@till@marker##1\endmarker{}%
	\futurelet\first@char\gobble@till@marker#1\endmarker
	\ifcat\noexpand\first@char A\else
	\def\first@char{}%
	\fi
	\macc@nested@a\relax111{\first@char}%
	\fi
	\endgroup
}
\makeatother

\DeclareSourcemap{
  \maps[datatype=bibtex]{
    \map{
      \step[fieldsource=note, final]
      \step[fieldset=addendum, origfieldval, final]
      \step[fieldset=note, null]
    }
  }
}

\tikzcdset{
  cells={font=\everymath\expandafter{\the\everymath\displaystyle}},
}

\renewcommand{\longrightarrow}{\to}

\newcommand{\LTop}{\categ{LTop}_{\infty}}

\renewcommand{\Open}{\mathrm{Open}}

\renewcommand{\enumref}[2]{\Cref{#1}-(\ref*{#1.#2})}

\renewcommand{\Bbar}{\overline{B}}

\renewcommand{\Ubar}{\overline{U}}
\renewcommand{\Xbar}{\overline{X}}

\renewcommand{\ibar}{\bar{\imath}}

\newcommand{\source}{\mathrm{s}}
\newcommand{\target}{\mathrm{t}}

\newcommand{\RBS}{\mathrm{RBS}}
\newcommand{\Mgnbar}{\overline{\Mcal}_{g,n}}
\newcommand{\Mgnbartop}{\overline{\Mcal} {}_{g,n}^{\mathrm{top}}}
\newcommand{\Broken}{\mathrm{Broken}}

\newcommand{\surj}{\mathrm{surj}}
\newcommand{\Deltasurj}{\DDelta_{\surj}}

\definecolor{goodred}{RGB}{216,27,96}
\definecolor{goodblue}{RGB}{30,136,229}
\definecolor{goodyellow}{RGB}{255,193,7}

\newcommand{\U}{\Ucal}
\renewcommand{\W}{\Wcal}
\renewcommand{\X}{\Xcal}
\renewcommand{\Y}{\Ycal}
\renewcommand{\Z}{\Zcal}

\newcommand{\Xhyp}{\X^{\hyp}}
\newcommand{\Yhyp}{\Y^{\hyp}}

\newcommand{\XS}{\X_{S}}
\newcommand{\XU}{\X_{U}}
\newcommand{\XZ}{\X_{Z}}

\DeclareMathOperator{\inj}{inj}

\DeclareMathOperator{\at}{at}

\DeclareMathOperator{\idem}{idem}

\newcommand{\PrLat}{\categ{Pr}^{\kern0.05em\operatorname{L},\at}}
\newcommand{\PrRat}{\categ{Pr}^{\kern0.05em\operatorname{R},\at}}
\newcommand{\PrLomega}{\categ{Pr}^{\kern0.05em\operatorname{L},\upomega}}

\DeclareMathOperator{\cosp}{cosp}

\newcommand{\CATinfty}{\categ{\textsc{Cat}}_{\infty} }

\newcommand{\Catidem}{\Cat_{\infty}^{\idem}}
\newcommand{\Catomega}{\Cat_{\infty}^{\upomega}}
\newcommand{\Catfin}{\Cat_{\infty}^{\fin}}

\newcommand{\CatomegaD}{\Cat_{\infty,/\Dcal}^{\upomega}}
\newcommand{\CatfinD}{\Cat_{\infty,/\Dcal}^{\fin}}

\newcommand{\Piinfty}{\Pi_{\infty}}

\newcommand{\Catinfty}{\Cat_{\infty}}

\newcommand{\CatP}{\Cat_{\infty,/P}}
\newcommand{\CatPcons}{\Cat_{\infty,/P}^{\cons}}

\DeclareMathOperator{\Cons}{Cons}
\newcommand{\Conshyp}{\Cons^{\hyp}}
\newcommand{\ConsP}{\Cons_{P}}
\newcommand{\ConsQ}{\Cons_{Q}}
\newcommand{\ConsR}{\Cons_{R}}
\newcommand{\ConsS}{\Cons_{S}}
\newcommand{\ConsU}{\Cons_{U}}
\newcommand{\ConsZ}{\Cons_{Z}}
\newcommand{\ConsPhyp}{\ConsP^{\hyp}}
\newcommand{\ConsQhyp}{\ConsQ^{\hyp}}

\newcommand{\Deltainj}{\DDelta_{\inj}}
\newcommand{\Deltainjop}{\DDelta_{\inj}^{\op}}

\DeclareMathOperator{\Env}{Env}

\DeclareMathOperator{\ex}{ex}
\DeclareMathOperator{\mon}{mon}

\newcommand{\RTop}{\categ{RTop}_{\infty}}
\newcommand{\RTopmon}{\categ{RTop}_{\infty}^{\mon}}
\newcommand{\StrTop}{\categ{StrTop}_{\infty}}
\newcommand{\StrTopex}{\StrTop^{\ex}}

\newcommand{\Poset}{\categ{Poset}}

\newcommand{\PiSigma}{\Piinfty}

\newcommand{\iotaup}{\mathrm{\iota}}

\newcommand{\upc}{\mathrm{c}}

\newcommand{\fcons}{f^{\upc}}
\newcommand{\fconslowerstar}{\fcons_\ast}
\newcommand{\flowerstarcons}{\fcons_\ast}

\newcommand{\iUlowerstar}{i_{U,\ast}}

\newcommand{\iSlowerstar}{i_{S,\ast}}

\newcommand{\icons}{i^{\upc}}
\newcommand{\iconsSlowerstar}{\icons_{S,\ast}}

\newcommand{\iUlowerstarcons}{\icons_{U,\ast}}

\newcommand{\flowersharpcons}{\flowersharp^{\upc}}
\newcommand{\iSlowersharpcons}{\icons_{S,\sharp}}

\newcommand{\iZlowersharpcons}{\icons_{Z,\sharp}}
\newcommand{\iUlowershriek}{i_{U,!}}

\newcommand{\scons}{s^{\upc}}
\newcommand{\slowersharpcons}{\scons_{\sharp}}

\newcommand{\supperstarhyp}{s^{\ast,\hyp}}
\newcommand{\slowerstarhyp}{\slowerstar^{\hyp}}

\newcommand{\ilowerstarhyp}{\ilowerstar^{\hyp}}

\newcommand{\iUupperstar}{\iupperstar_{U}}
\newcommand{\iZupperstar}{\iupperstar_{Z}}

\newcommand{\upperstarhyp}{^{\ast,\hyp}}
\newcommand{\fupperstarhyp}{f\upperstarhyp}
\newcommand{\iupperstarhyp}{i\upperstarhyp}

\newcommand{\jupperstarhyp}{j\upperstarhyp}

\newcommand{\inv}{^{-1}}

\theoremstyle{definition}

\newtheorem{existingresults}[equation]{Existing Results}
\newtheorem{methods}[equation]{Methods}

\makeatletter
\def\swappedhead#1#2#3{%
  \thmnumber{\@upn{\the\thm@headfont#2\@ifnotempty{#1}{\,~}}}%
  \thmname{#1}%
  \thmnote{ {\the\thm@notefont{#3}}}}
\makeatother

\makeatletter
\def\@seccntformat#1{%
  \protect\textup{\protect\@secnumfont
    \ifnum\pdfstrcmp{subsection}{#1}=0 \bfseries\fi
    \csname the#1\endcsname
    \protect{\,}
  }%
}  
\makeatother

\makeatletter
\def\@tocline#1#2#3#4#5#6#7{\relax
  \ifnum #1>\c@tocdepth 
  \else
    \par \addpenalty\@secpenalty\addvspace{#2}%
    \begingroup \hyphenpenalty\@M
    \@ifempty{#4}{%
      \@tempdima\csname r@tocindent\number#1\endcsname\relax
    }{%
      \@tempdima#4\relax
    }%
    \parindent\z@ \leftskip#3\relax \advance\leftskip\@tempdima\relax
    \rightskip\@pnumwidth plus4em \parfillskip-\@pnumwidth
    #5\leavevmode\hskip-\@tempdima
      \ifcase #1
       \or\or \hskip 1em \or \hskip 2em \else \hskip 3em \fi%
      #6\nobreak\relax
    \hfill\hbox to\@pnumwidth{\@tocpagenum{#7}}\par
    \nobreak
    \endgroup
  \fi}
\makeatother

\usepackage{caption}
\numberwithin{equation}{subsection} 

\title{Exodromy beyond conicality}

\author{Peter J. Haine}
\address{Peter J. Haine, Department of Mathematics, University of Southern California, Kaprielian Hall, 3620 S Vermont Ave, Los Angeles, CA 90089, USA}
\email{phaine@usc.edu}

\author{Mauro Porta}
\address{Mauro Porta, Institut de Recherche Mathématique Avancée, 7 Rue René Descartes, 67000 Strasbourg, France}
\email{porta@math.unistra.fr}

\author{Jean-Baptiste Teyssier}
\address{Jean-Baptiste Teyssier, Institut de Mathématiques de Jussieu, 4 place Jussieu, 75005 Paris, France}
\email{jean-baptiste.teyssier@imj-prg.fr}

\date{\today}

\begin{document}


\begin{abstract}
	We show that compact subanalytic stratified spaces and algebraic stratifications of real varieties have finite exit-path \categories, refining classical theorems of Lefschetz--Whitehead, Łojasiewicz, and Hironaka on the finiteness of the underlying homotopy types of these spaces.
	These stratifications are typically not conical; hence we cannot rely on the currently available exodromy equivalence between constructible sheaves on a stratified space, which requires conicality as a fundamental hypothesis.
	Building on ideas of Clausen and Jansen, we study the class of \textit{exodromic} stratified spaces, for which the conclusion of the exodromy theorem holds.
	We prove two new fundamental properties of this class of stratified spaces: coarsenings of exodromic stratifications are exodromic, and every morphism between exodromic stratified spaces induces a functor between the associated exit path \categories.
	As a consequence, we produce many new examples of exodromic stratified spaces, including: coarsenings of conical stratifications, locally finite subanalytic stratifications of real analytic spaces, and algebraic stratifications of real varieties.
	Our proofs are at the generality of stratified \topoi, hence apply to even more general situations such as stratified topological stacks.
	In a subsequent paper, we use the previously mentioned finiteness results to construct derived moduli stacks of constructible and perverse sheaves.
\end{abstract}

\maketitle

\setcounter{tocdepth}{1}
\tableofcontents


\setcounter{section}{-1}

\section{Introduction}


\subsection{Motivation}\label{intro-subsec:motivation}

Let $ (X,P) $ be a stratified space. 
MacPherson observed the following generalization of the monodromy equivalence: provided the stratification of $ X $ is sufficiently nice, the category of constructible sheaves of sets on $ (X,P) $ is equivalent to the category of functors from the \textit{exit-path category} of $ (X,P) $ to $ \Set $.
Treumann \cite{MR2575092} provided the first general account of this phenomenon, and Treumann's result has since been generalized by Lurie \HAa{Theorem}{A.9.3}, Lejay \cite{arXiv:2102.12325}, and Porta--Teyssier \cite{arXiv:2211.05004}.
To contextualize the results of this paper, let us first recall the most general theorem of this form currently available.
Write \smash{$ \ConsPhyp(X) $} for the \category of hyperconstructible hypersheaves%
\footnote{In this introduction, the reader can safely disregard the adjective ``hyper''. 
Hypersheaves are used in \cites{HPT}{Lejay}{PortaTeyssier} to relax the geometric assumptions needed for the theorem.}
of spaces on $ X $.
If the stratification of $ (X,P) $ is \textit{conical} (see \cite[Definition 2.1.9]{arXiv:2211.05004}) and the strata are locally weakly contractible, the \textit{exodromy theorem}%
\footnote{The term `exodromy' was first introduced in \cite{arXiv:1807.03281} as a combination of `monodromy' and `exit-paths'.}
\cite[Theorem 5.4.1]{PortaTeyssier} provides an equivalence of \categories
\begin{equation}\label{eq:topological_exodromy}
	\Phi_{X,P} \colon \equivto{\ConsPhyp(X)}{\Fun(\Piinfty(X,P), \Spc)} \period
\end{equation}
Here $ \Piinfty(X,P) $ is Lurie's exit-path \category of $ (X,P) $, introduced in \HAa{Definition}{A.6.2}.
The objects of $ \Piinfty(X,P) $ are the points of $ X $.
Roughly speaking, the $ 1 $-morphisms are \textit{exit-paths} flowing from lower to higher strata (and once they exit a stratum are not allowed to return), the $ 2 $-morphisms are homotopies of exit-paths respecting stratifications, etc.
The functor $ \Phi_{X,P} $ carries a sheaf $ F $ to the functor informally described by sending a point $ x \in X $ to the stalk $ F_{x} $, and each exit-path $ \fromto{x}{y} $ to a \textit{specialization map} $ \fromto{F_x}{F_y} $, together with higher coherences relating these data.
Conicality has played an essential role in almost all exodromy theorems available in the literature.
First, it is crucial in proving that the geometrically defined object $ \Piinfty(X,P) $ is indeed \acategory \HAa{Theorem}{A.6.4}.
Second, it is used at various points in the proof of the equivalence \eqref{eq:topological_exodromy}.
Many stratifications naturally arising in geometry \textit{fail} to be conical: typical examples are general subanalytic stratifications of real analytic manifolds, such as those arising from the study of the Stokes phenomenon for algebraic differential equations (see \cites{arXiv:2401.12335}{arXiv:2504.05360}).
Deep work of Thom, Mather, and Verdier among others on analytic stratified spaces has shown that conical (in fact, Whitney) refinements are always available \cite{Goresky_MacPherson_Stratified_Morse_theory}; however, it is sometimes essential to work with a fixed stratification.
The purpose of this article is to generalize the exodromy equivalence to many naturally occurring \textit{non-conically} stratified spaces, paying particular attention to the \textit{conically refineable} situation.


\subsection{Exodromic stratified spaces}\label{intro-subsec:exodromic_stratified_spaces}

To state the results of this paper, we need to briefly introduce the concept of an exit-path \category without reference to any particular simplicial model.
As highlighted by Ayala--Francis--Rozenblyum \cites[Problem 0.0.9]{MR3941460} and explained by Clausen--Jansen \cites{arXiv:2108.01924}{arXiv:2308.09551}{arXiv:2308.09550}, one should be able to trade off the conicality of a stratified space $ (X,P) $ and Lurie's simplicial model for the exit-path \category for the following three requirements of the \category \smash{$ \ConsPhyp(X) $}:

\begin{definition}[{(cf. \cite[Definition 3.5]{arXiv:2108.01924})}]\label{intro-def:exodromic_stratified_space}
	A stratified space $ s \colon \fromto{X}{P} $ is \defn{exodromic} if the following conditions are satisfied:
	\begin{enumerate}
		\item The \category $ \ConsPhyp(X) $ is atomically generated.
		
		\item The subcategory $\ConsPhyp(X) \subset \Shhyp(X) $ is closed under both limits and colimits.
		
		\item The pullback functor $ \supperstarhyp \colon \Fun(P,\Spc) \equivalent \Shhyp(P) \to \ConsPhyp(X)$ preserves limits.
	\end{enumerate}
\end{definition}

Let us comment these requirements.
Concerning (1), note that the exodromy theorem guarantees that the \category \smash{$ \ConsPhyp(X) $} can be written as \acategory of presheaves.
\textit{Atomic generation} is an intrinsic way to formulate this property: given a presentable \category $\Ccal$, an object $ c \in \Ccal $ is \textit{atomic} if the functor
\begin{equation*} 
	\Map_{\Ccal}(c,-) \colon \Ccal \to \Spc
\end{equation*}
preserves \textit{all} colimits.
Write $ \Ccal^{\at} \subset \Ccal $ for the full subcategory spanned by the atomic objects.
Then $\Ccal$ is said to be \textit{atomically generated} if the unique colimit-preserving extension
\begin{equation*}
	\fromto{\PSh(\Ccal^{\at})}{\Ccal}
\end{equation*} 
of $ \Ccal^{\at} \subset \Ccal $ along the Yoneda embedding is an equivalence (see \cref{subsec:recollections_on_atomic_generation} for more background on this notion).
In the setting of \cref{intro-def:exodromic_stratified_space}, we write \smash{$ \Piinfty(X,P) $} for the opposite of the full subcategory of \smash{$ \ConsPhyp(X) $} spanned by atomic objects.
We refer to $ \Piinfty(X,P) $ as the \defn{exit-path \category} of $ (X,P) $.
The second feature is that \smash{$ \ConsPhyp(X) \subset \Shhyp(X) $} is closed under both limits and colimits.
This is in some sense a \textit{categorical regularity condition}, which is akin to but weaker than conicality: see \cite[Corollary 5.4.4]{PortaTeyssier} for a proof in the conical setting, and see \cref{defin:weakly_conical} and \cref{eg:weakly_conical} for other examples of regularity properties enjoyed by the conical situation.
The third feature is that, by construction, the exit-path \category of $ (X,P) $ is equipped with a functor to the stratifying poset $ P $.
Given conditions (1) and (2), condition (3) guarantees that the stratification of $ X $ equips $ \Piinfty(X,P) $ with a functor $\Piinfty(X,P) \to P $; see \Cref{recollection:atomically_generated}.


\subsection{The stability theorem}\label{intro-subsec:the_stability_theorem}

The analysis of the conical situation carried out in \cite{PortaTeyssier} shows that conically stratified spaces with locally weakly contractible strata are exodromic in the sense of \cref{intro-def:exodromic_stratified_space}.
However, the class of such stratified spaces does not have many stability properties; for example, a coarsening%
\footnote{Let $ X $ be a topological space and let $ s \colon X \to P $ and $ t \colon X \to Q $ be stratifications.
If there is a map of posets $ \phi \colon P \to Q $ such that $ s\phi = t $, we say that the stratification $ t $ is a \defn{coarsening} of the stratification $ s $ and the stratification $ s $ is a \defn{refinement} of the stratification $ t $.}
of a conical stratification need not be conical.
As previously mentioned, in subanalytic geometry and real algebraic geometry conical refinements always exist, at least locally.
The following is the main result of this paper, and in particular it implies that \textit{every} subanalytic or real analytic stratified space is exodromic:

\begin{theorem}[(stability properties of exodromic stratified spaces; \Cref{thm:stability_properties_of_exodomic_stratified_spaces})]\label{intro_thm:properties_of_exodromic_stratified_spaces}
	\noindent 
	\begin{enumerate}
		\item\label{intro_thm:properties_of_exodromic_stratified_spaces.1} \emph{Stability under pulling back to locally closed subposets:} If $ (X,P) $ is an exodromic stratified space, then for each locally closed subposet $ S \subset P $, the stratified space $ (X \cross_P S, S) $ is exodromic and the induced functor
		\begin{equation*}
			\fromto{\Piinfty(X \cross_P S,S)}{\Piinfty(X,P) \cross_P S}
		\end{equation*}
		is an equivalence.
		As a consequence, the induced functor $ \fromto{\Piinfty(X,P)}{P} $ is conservative.

		\item\label{intro_thm:properties_of_exodromic_stratified_spaces.2} \emph{Functoriality:} The exodromy equivalence is functorial in all stratified maps between exodromic stratified spaces.
		That is, for every stratified map $ f \colon \fromto{(X,P)}{(Y,Q)} $ between exodromic stratified spaces, under the exodromy equivalence the pullback functor
		\begin{equation*}
			\fupperstarhyp \colon \fromto{\ConsQhyp(Y)}{\ConsPhyp(X)}
		\end{equation*}
		is induced by a functor of exit-path \categories
		\begin{equation*}
			\fromto{\Piinfty(X,P)}{\Piinfty(Y,Q)} \period
		\end{equation*}

		\item\label{intro_thm:properties_of_exodromic_stratified_spaces.3} \emph{Stability under coarsening and localization formula:} 
		Let $ (X,R) $ be an exodromic stratified space and let $ \phi \colon \fromto{R}{P} $ be a map of posets. 
		Write $ W_P $ for the collection of morphisms in $ \Piinfty(X,R) $ that the composite $ \Piinfty(X,R) \to R \to P $ sends to equivalences.
		Then the stratified space $ (X,P) $ is exodromic and the natural functor $ \fromto{\Piinfty(X,R)}{\Piinfty(X,P)} $ induces an equivalence
		\begin{equation*}
			\equivto{\Piinfty(X,R)[W_P\inv]}{\Piinfty(X,P)} \period
		\end{equation*}

		\item\label{intro_thm:properties_of_exodromic_stratified_spaces.4} \emph{van Kampen:} The property of a stratified space being exodromic can be checked locally.

		\item\label{intro_thm:properties_of_exodromic_stratified_spaces.5} \emph{Stability of finiteness/compactness:} The property of an exit-path \category being finite (resp., compact) is stable under pulling back to a locally closed subposet, is stable under coarsening, and can be checked on a finite open cover.
	\end{enumerate}
\end{theorem}

Together, the items in \Cref{intro_thm:properties_of_exodromic_stratified_spaces} provide robust techniques to produce new examples of exodromic stratified spaces starting from conically stratified spaces.
We will explain many new examples of stratified spaces momentarily.
Before proceeding further, we comment on how \Cref{intro_thm:properties_of_exodromic_stratified_spaces} relates to existing results, and our proof methods. 

\begin{existingresults}
	Item (1) was proven by Clausen--Jansen in a slightly different topological setting \cite[Proposition 3.6-(1)]{arXiv:2108.01924}, and by Jansen for topological stacks \cite[Proposition 3.13-(1)]{arXiv:2308.09550}.
	Item (4) is an easy consequence of the theory, and, in the same settings, was previously observed by Clausen--Jansen \cite[Proposition 3.6-(2)]{arXiv:2108.01924} and Jansen \cite[Proposition 3.13-(2)]{arXiv:2308.09550}.
	Two early instances of (2) were proven in the conically stratified setting by Lurie \HAa{Corollary}{A.9.4} in the case where $ P = \pt $, and Ayala--Francis--Rozenblyum \cite[Theorem 3.3.12]{MR3941460} under some additional hypotheses on the stratifying posets.
	Recently, Jansen \cite[Proposition 3.20]{arXiv:2308.09550} generalized the argument given by Ayala--Francis--Rozenblyum; however the hypotheses are still somewhat restrictive.
		
	The first main contribution of \Cref{intro_thm:properties_of_exodromic_stratified_spaces} is that our results have no restrictions on the stratifiying posets.
	The second is that we prove functoriality of the exodromy equivalence in all maps of stratified spaces. This is a new result and may be somewhat surprising; with previous methods, even functoriality in the conical case was a nontrivial result, first proven in \cite[Proposition 6.2.3]{PortaTeyssier}.
	The third is item (5) on the stability of finiteness/compactness; its proof requires a careful understanding of the localization formula from (3) and it generalizes classical finiteness results for homotopy types of real analytic manifolds.
	See \cref{intro_rem:classical_finiteness}.
\end{existingresults}

\begin{methods}
	The final main contribution is that our result is actually even more general than \Cref{intro_thm:properties_of_exodromic_stratified_spaces}.
	The point is that the conditions in \Cref{intro-def:exodromic_stratified_space} only depend on the datum of the geometric morphism of \topoi
	\begin{equation*}
		\fromto{\Shhyp(X)}{\Fun(P,\Spc)} \period
	\end{equation*}
	This is an example of a \textit{stratified \topos}, as introduced in the work of Barwick--Glasman--Haine \cite{arXiv:1807.03281}.
	Consequently, \Cref{intro-def:exodromic_stratified_space} makes sense at the generality of stratified \topoi.
	See \cref{sec:exit-path_categories}, in particular \Cref{def:exit_path}.
		
	We prove \Cref{intro_thm:properties_of_exodromic_stratified_spaces} by proving its natural generalization to stratified \topoi.
	See \Cref{thm:stability_properties_of_stratified_topoi} for a precise statement.
	This generalization gives added flexibility; for example, it immediately applies to stratified topological stacks.
	It also subsumes all results of this form that we are aware of, for example, the stability results proven by Clausen--Jansen \cite{arXiv:2108.01924} and Jansen \cite{arXiv:2308.09550}. 
	The topos-theoretic result has the added benefit of providing a common framework for the various contexts where exodromy was previously considered (e.g., sheaves vs. hypersheaves). 
\end{methods}


\subsection{Applications of the stability theorem}\label{intro-subsec:applications_of_the_stability_theorem}

We now state our main applications of \Cref{intro_thm:properties_of_exodromic_stratified_spaces}.
Since every conically stratified space with locally weakly contractible strata is exodromic and the class of conically stratified spaces with locally weakly contractible strata is stable under passing to open subsets, we deduce:

\begin{corollary}[{(\Cref{prop:properties_of_locally_conically_refineable_stratified_spaces})}]
	If a stratified space $ (X,P) $ locally admits a refinement by a conical stratification with locally weakly contractible strata, then $ (X,P) $ is exodromic.
\end{corollary}

\noindent A theorem of Verdier guarantees that a locally finite subanalytic stratification of a real analytic space admits a refinement that is Whitney stratified \cite[Théorème 2.2]{MR481096}.
Since Whitney stratifications are conical \cites{MR2958928}{MR239613}, a little more work on top of \Cref{intro_thm:properties_of_exodromic_stratified_spaces} shows:

\begin{theorem}[{(\Cref{thm:subanalytic_stratified_spaces_are_conically_refineable})}]\label{intro_thm:subanalytic_stratifications}
	Let $ (X,P) $ be a real analytic manifold equipped with a locally finite stratification by subanalytic subsets.
	Then:
	\begin{enumerate}
		\item\label{intro_thm:subanalytic_stratifications.1} The stratified space $ (X,P) $ is exodromic.

		\item\label{intro_thm:subanalytic_stratifications.2} If $ X $ is compact, then the exit-path \category $ \Piinfty(X,P) $ is finite.
	\end{enumerate}
\end{theorem}

\begin{theorem}[{(\Cref{thm:algebraic_stratified_spaces_are_conically_refineable_and_categorically_finite})}]\label{intro_thm:algebraic_stratified_spaces_are_conically_refineable_and_categorically_finite}
	Let $ X $ be an algebraic variety over $ \RR $ and let $ (X,P) $ be a finite stratification of $ X $ by Zariski locally closed subsets.
	Then:
	\begin{enumerate}
		\item\label{intro_thm:algebraic_stratified_spaces_are_conically_refineable_and_categorically_finite.1} The stratified space $ (X,P) $ is is exodromic.

		\item\label{intro_thm:algebraic_stratified_spaces_are_conically_refineable_and_categorically_finite.2} The exit-path \category $ \Piinfty(X,P) $ is finite.
	\end{enumerate}
\end{theorem}

\begin{remark}\label{intro_rem:classical_finiteness}
	\enumref{intro_thm:subanalytic_stratifications}{2} and \enumref{intro_thm:algebraic_stratified_spaces_are_conically_refineable_and_categorically_finite}{2} extend results of Lefschetz--Whitehead \cite{MR1501698}, Łojasiewicz \cite{MR0173265}, and Hironaka \cite{MR0374131} on the finiteness of the underlying homotopy types of compact subanalytic spaces and real algebraic varieties.
\end{remark}

As an application, in a follow-up paper \cite{HPT-derived_moduli_of_perverse_sheaves}, we use \Cref{intro_thm:subanalytic_stratifications,intro_thm:algebraic_stratified_spaces_are_conically_refineable_and_categorically_finite} combined with Toën and Vaquié's theory of \textit{moduli of objects} \cite{MR2493386} to prove representability results for derived moduli stacks of constructible and perverse sheaves.


\subsection{Examples}\label{intro-subsec:examples}

We conclude the introduction with some examples of non-conical stratifications to which our results apply.
First we demonstrate how to compute the exit-path \category of a coarsening in a simple situation.

\begin{example}\label{ex:circle}
	Consider a circle stratified by a point, half-open interval, and open interval, as depicted on the right-hand side of \Cref{fig:circle}.
	This stratification of $ \Sph{1} $ is not conical.
	However, the stratification of $ \Sph{1} $ by two points and two half-open intervals appearing on the left-hand side of \Cref{fig:circle} is a conical stratification that refines the stratification on the right-hand side.
	\begin{figure}[!h]
	  \centering
	  \includegraphics[width=0.5\linewidth]{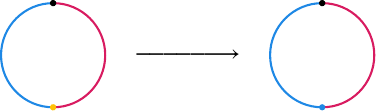}
	  \caption{A non-conical stratification of $ \Sph{1} $ is pictured on the right. 
	  On the left is a conical refinement of the right-hand stratification.}
	  \label{fig:circle}
	\end{figure}
	The exodromy theorem in the concical case shows that the exit-path \category of the left-hand stratification of $ \Sph{1} $ is equivalent to the poset 
	\begin{equation}\label{diag:noncommutative_square}
		\begin{tikzcd}[sep=1.75em]
			& \mdsmblkcircle  \arrow[dr] \arrow[dl] &  \\ 
			\textcolor{goodblue}{\mdsmblkcircle} & & \textcolor{goodred}{\mdsmblkcircle} \\
			& \textcolor{goodyellow}{\mdsmblkcircle}  \arrow[ur] \arrow[ul] & \phantom{\mdsmblkcircle} \period
		\end{tikzcd}
	\end{equation}
	Thus \enumref{intro_thm:properties_of_exodromic_stratified_spaces}{3} implies that the exit-path \category of the right-hand stratification is equivalent to the localization of the poset \eqref{diag:noncommutative_square} at the morphism $ \textcolor{goodyellow}{\bullet} \to \textcolor{goodblue}{\bullet} $.
	This localization is simply the category given by a \textit{noncommutative} triangle
	\begin{equation*}
		\begin{tikzcd}[sep=1.75em]
			& \mdsmblkcircle  \arrow[dr] \arrow[dl] &  \\ 
			\textcolor{goodblue}{\mdsmblkcircle} \arrow[rr] & & \textcolor{goodred}{\mdsmblkcircle} \period
		\end{tikzcd}
	\end{equation*}
\end{example}

\begin{example}[(see \Cref{ex:Favero-Huang_tree_stratification_is_conically_refineable,ex:Favero-Huang-Bondal-Ruan_stratification_is_locally_conically_refineable})]
	Favero and Huang \cite{arXiv:2205.03730} recently proved an exodromy result for certain non-conical stratifications naturally arising in mirror symmetry.
	Of particular interest are the \textit{tree stratification} on a finite simplicial complex \cite[\S4.4]{arXiv:2205.03730} and the \textit{Bondal--Ruan stratification} of the $ n $-torus \cites{MR2278891}[\S5.2]{arXiv:2205.03730}.
	The tree stratification is a coarsening of the natural stratification on a finite simplicial complex, which is conical.
	Moreover, the Bondal--Ruan stratification is subanalytic.
	Thus \Cref{intro_thm:properties_of_exodromic_stratified_spaces,intro_thm:subanalytic_stratifications} give an alternative perspective on Favero and Huang's exit-path description of constructible sheaves on these stratified spaces.
\end{example}

More examples arise naturally from the study of the Stokes phenomenon for algebraic differential equations.
See \cites{arXiv:2401.12335}{arXiv:2504.05360} for more on this topic, as well as a systematic use of the results of this paper.


\subsection{Linear overview}\label{intro-subsec:linear_overview}

In \cref{sec:background_on_atomic_generation_locally_constant_objects_and_monodromy}, we provide background on atomically generated \categories, locally constant objects of \topoi, and monodromy that we need for the rest of the paper.
In \cref{sec:exit-path_categories}, be begin by recalling the theory of stratifications of \topoi introduced in \cite[\S8.2]{arXiv:1807.03281} as well as constructible objects.
We then explain what it means for a stratified \topos to be \textit{exodromic}, see \Cref{def:exit_path}.
We also prove a few basic results about the class of exodromic stratified \topoi.
In \cref{sec:stability_properties_of_exodromic_stratified_topoi}, we prove a stability theorem for the class of exodromic stratified \topoi, see \Cref{thm:stability_properties_of_stratified_topoi}.
This is the main technical result of the paper, and implies the analogous result for stratified spaces stated in this introduction (\Cref{intro_thm:properties_of_exodromic_stratified_spaces}).
\Cref{sec:exodromy_with_coefficients} explains when exodromy (with coefficients in the \category of spaces) implies exodromy with coefficients in other presentable \categories. 
The key takeaway is that exodromy with coefficients in a compactly assembled \category is automatic (see \Cref{lem:exit_path_category_for_compactly_assembled_coefficients_is_automatic}).
We need this result in order to prove our representability result for the derived moduli of constructible and perverse sheaves in \cite{HPT-derived_moduli_of_perverse_sheaves}.
\Cref{sec:applications_and_examples} is dedicated to applications of our stability theorem for exodromic stratified \topoi (\Cref{thm:stability_properties_of_stratified_topoi}).
We deduce \Cref{intro_thm:properties_of_exodromic_stratified_spaces}, provide many natural examples of exodromic stratified spaces coming from geometry and topology, and prove all of the results stated in \cref{intro-subsec:applications_of_the_stability_theorem}.
In \Cref{appendix:inverting_arrows_over_a_poset}, we prove a number of categorical facts needed to control the localizations of exit-path \categories we consider.
Specifically, the results proven in \Cref{appendix:inverting_arrows_over_a_poset} are needed to prove items (\ref*{intro_thm:properties_of_exodromic_stratified_spaces.3}) and (\ref*{intro_thm:properties_of_exodromic_stratified_spaces.5}) of \Cref{intro_thm:properties_of_exodromic_stratified_spaces}.
In \Cref{appendix:complements_on_topoi}, we collect some background on open and closed subtopoi and recollements. 
We then explain the relationship between hypercompletion and recollements (see \Cref{prop:hypercompletions_of_recollements}).
We need these results in a variety of places, for example, to ensure that the definition of a constructible object of a stratified \topos recovers the more classical notion of a constructible (hyper)sheaf on a topological space.


\subsection{Notational conventions}

We use the following standard notation.
\begin{enumerate}
	\item We write $ \Catinfty $ for the large \category of small \categories, and write $ \Spc \subset \Catinfty $ for the full subcategory spanned by the spaces (i.e., \groupoids or anima).
	We write $ \CATinfty $ for the (very large) \category of large \categories.

	\item We write $ \PrR $ for the \category of presentable \categories and right adjoints and $ \PrL $ for the \category of presentable \categories and left adjoints.

	\item We write $ \RTop $ for the \category of \topoi and \textit{geometric morphisms}, i.e., right adjoints $ \flowerstar $ whose left adjoint $ \fupperstar $ is left exact.
	We write $ \LTop $ for the \category of \topoi and left exact left adjoints.

	\item Given a small \category $ \Ccal $, we write $ \PSh(\Ccal) \colonequals \Fun(\Ccal^{\op},\Spc) $ for the \category of presheaves of spaces on $ \Ccal $.

	\item For an integer $ n \geq 0 $, we write $ [n] $ for the poset $ \{0 < \cdots < n \}$
\end{enumerate}
We later introduce notational conventions for (hyper)sheaves and constructibility; these are consistent with the notational conventions in our previous works \cites{arXiv:2010.06473}{PortaTeyssier}.


\subsection{Acknowledgments}

We thank David Ayala, Clark Barwick, Marc Hoyois, Jesse Huang, Jacob Lurie, Guglielmo Nocera, Marco Volpe, and Mikala Ørsnes Jansen for enlightening discussions around the contents of this paper.

PH gratefully acknowledges support from the NSF Mathematical Sciences Postdoctoral Research Fellowship under Grant \#DMS-2102957 and a grant from the Simons Foundation (816048, LC). 


\section{Background on atomic generation, locally constant objects, and monodromy}\label{sec:background_on_atomic_generation_locally_constant_objects_and_monodromy}

In this section, we recall the necessary background on \textit{atomically generated \categories} (\cref{subsec:recollections_on_atomic_generation}), tensor products of presentable \categories (\cref{subsec:sheaves_with_coefficients_and_tensor_products_of_presentable_categories}), and locally constant objects of \topoi and monodromy (\cref{subsec:locally_constant_objects_and_monodromy}).


\subsection{Recollections on atomic generation}\label{subsec:recollections_on_atomic_generation}

We begin by recalling the background on atomically generated \categories needed in this paper.
In particular, we provide a useful way to check that a full subcategory of an atomically generated \category is also atomically generated and compute its generators (\Cref{prop:atomic_generation}).
For more on this topic, we refer the reader to \cites[\kerodontag{03WR}]{Kerodon}[\HTTsubsec{5.1.6}]{HTT}[\S2.2]{arXiv:2108.01924}.
We begin with some definitions.

\begin{definition}\label{def:atomic_object}
	Let $ \Ccal $ be a presentable \category.
	An object $ c \in \Ccal $ is \defn{atomic}%
	\footnote{Atomic objects are also referred to as \textit{completely compact} objects \HTT{Definition}{5.1.6.2}.}
	if the functor
	\begin{equation*}
		\Map_{\Ccal}(c,-) \colon \Ccal \to \Spc
	\end{equation*}
	preserves colimits.
	We write $ \Ccal^{\at} \subset \Ccal $ for the full subcategory spanned by the atomic objects.
\end{definition}

\begin{observation}
	The subcategory $ \Ccal^{\at} \subset \Ccal $ is always small and idempotent complete.
	However, contrary to what happens to the full subcategory $ \Ccal^{\upomega} \subset \Ccal $ spanned by compact objects, the \category $ \Ccal^{\at} $ typically does not have finite colimits.
\end{observation}

\begin{definition}\label{def:atomically_generated_category}
	Let $ \Ccal $ be a presentable \category.
	We say that a small full subcategory $ \Ccal_0 \subset \Ccal $ \defn{atomically generates} $ \Ccal $ if the unique colimit-preserving extension
	\begin{equation*}
		\fromto{\PSh(\Ccal_0)}{\Ccal}
	\end{equation*} 
	of $ \Ccal_0 \subset \Ccal $ along the Yoneda embedding is an equivalence.
	We say that $ \Ccal $ is \defn{atomically generated} if there exists a small full subcategory $ \Ccal_0 \subset \Ccal $ that atomically generates $ \Ccal $.
\end{definition}

\begin{remark}\label{rem:restricted_yoneda}
	The unique colimit-preserving extension $ \fromto{\PSh(\Ccal_0)}{\Ccal} $ of the inclusion $ \Ccal_0 \subset \Ccal $ is left adjoint to the restricted Yoneda functor
	\begin{equation*}
		y \colon \Ccal \to \PSh(\Ccal_0) \comma \qquad c  \mapsto \Map_{\Ccal}(-,c) \period
	\end{equation*}
	Hence $ \Ccal_0 $ atomically generates $ \Ccal $ if and only if $ y $ is an equivalence.
\end{remark}

\begin{example}[{\HTT{Proposition}{5.1.6.8}}] \label{atomic_presheaf}
	Let $ \Ccal_0 $ be a small \category.
	Then, the atomic objects of $\PSh(\Ccal_0)$ are the retracts of representable functors.
	In particular, the unique atomic object of $ \Spc $ is the point $ \ast $.
\end{example}

\begin{observation}\label{obs:atomic_generation_and_idempotent_completion}
	If $ \Ccal_0 \subset \Ccal $ atomically generates $ \Ccal $, then $ \Ccal_0 \subset \Ccal^{\at} $.
	Moreover, by \kerodon{040X}, the inclusion $ \Ccal_0 \subset \Ccal^{\at} $ exhibits $ \Ccal^{\at} $ as the idempotent completion of $ \Ccal_0 $.
	As a consequence, all of the functors given by extending the obvious inclusions along colimits
	\begin{equation*}
		\begin{tikzcd}
			\PSh(\Ccal_0) \arrow[rr] \arrow[dr] & & \PSh(\Ccal^{\at}) \arrow[dl] \\ 
			& \Ccal & 
		\end{tikzcd} 
	\end{equation*}
	are equivalences. 
	In particular, $ \Ccal^{\at} $ also atomically generates $ \Ccal $.
\end{observation}

The following is an easy reformulation of the definition:

\begin{lemma}\label{lem:atomic_generation_via_conservativity}
	Let $ \Ccal $ be a presentable \category and let $ \Ccal_0 \subset \Ccal^{\at} $ be a full subcategory of atomic objects.
	Then $ \Ccal_0 $ atomically generates $ \Ccal $ if and only if the collection of corepresentable functors \smash{$ \set{\Map_{\Ccal}(c_0,-) \colon \Ccal \to \Spc}_{c_0 \in \Ccal_0} $} is jointly conservative.
\end{lemma}

\begin{proof}
	Since $ \Ccal_0 \subset \Ccal^{\at} $, by \HTT{Proposition}{5.1.6.10} the unique colimit-preserving extension of the inclusion $ \PSh(\Ccal_0) \to \Ccal $ is fully faithful.
	Hence this functor is an equivalence if and only if its right adjoint 
	\begin{equation*}
		y \colon \Ccal \to \PSh(\Ccal_0) \comma \qquad c \mapsto \Map_{\Ccal}(-,c) 
	\end{equation*}
	is conservative.
	This is exactly the condition that the functors $ \Map_{\Ccal}(c_0,-) \colon \Ccal \to \Spc $ for $ c_0 \in \Ccal_0 $ are jointly conservative.
\end{proof}

\begin{definition}\label{def:atomic_functor}
	Let $ L \colon \fromto{\Dcal}{\Ccal} $ be a left adjoint functor of \categories.
	We say that $ L $ is \defn{atomic} if the right adjoint $ \fromto{\Ccal}{\Dcal} $ to $ L $ is also a left adjoint.
\end{definition}

\begin{observation}\label{obs:atomic_functors_preserve_atomic_objects}
	If $ L \colon \fromto{\Dcal}{\Ccal} $ is an atomic functor between presentable \categories, then $ L $ preserves atomic objects, i.e., $ L(\Dcal^{\at}) \subset \Ccal^{\at} $.
	If $R$ denotes the right adjoint to $L$, then the square 
	\begin{equation*}
		\begin{tikzcd}
			\PSh(\Ccal^{\at}) \arrow[r, "L^*"]  \arrow[d] & \PSh(\Dcal^{\at}) \arrow[d] \\
	           \Ccal \arrow[r, "R"']  & \Dcal
		\end{tikzcd}
	\end{equation*}
	commutes.
\end{observation}

The converse is true if $ \Ccal $ and $ \Dcal $ are atomically generated:

\begin{recollection}[{\cite[Lemma 2.6-(3)]{arXiv:2108.01924}}]\label{rec:characterization_of_atomic_functors}
	Let $ L \colon \fromto{\Dcal}{\Ccal} $ be a left adjoint between atomically generated presentable \categories.
	Then $ L $ is atomic if and only if $ L $ preserves atomic objects.
\end{recollection}

In this paper, we repeatedly use the fact that the \category of atomically generated presentable  \categories and atomic functors is equivalent to the \category of idempotent complete \categories:

\begin{notation}
	Write $ \PrLat \subset \PrL $ for the non-full subcategory with objects the atomically generated \categories and morphisms atomic left adjoints.
	Write
	\begin{equation*}
		\Catidem \subset \Catinfty
	\end{equation*}
	for the full subcategory spanned by the idempotent complete \categories.
\end{notation}

\begin{recollection}[{\cite[Proposition 2.7]{arXiv:2108.01924}}]\label{recollection:atomically_generated}
	Consider the functor \smash{$ \PSh \colon \Catidem \to \PrL $} sending a small idempotent complete \category $ \Ccal_0 $ to $ \PSh(\Ccal_0) $ with functoriality given by left Kan extension.
	This functor restricts to an equivalence
	\begin{equation*}
		\PSh \colon \equivto{\Catidem}{\PrLat}  
	\end{equation*}
	with inverse given by
	\begin{equation*}
		(-)^{\at} \colon \equivto{\PrLat}{\Catidem} \period
	\end{equation*}
\end{recollection}

\begin{notation}\label{ntn:Ccal^ex}
	Let $ \Ccal $ be an atomically generated presentable \category.
	To simplify notation later on, we write $ \Ccal^{\ex} \colonequals (\Ccal^{\at})^{\op} $ for the opposite of the full subcategory of $ \Ccal $ spanned by the atomic objects.
	Thus there is a natural equivalence
	\begin{equation*}
		\Ccal \equivalent \Fun(\Ccal^{\ex},\Spc) \period
	\end{equation*}
\end{notation}

The proof of \Cref{intro_thm:properties_of_exodromic_stratified_spaces} relies on the fact that a full subcategory of an atomically generated \category that is closed under limits and colimits is also atomically generated:

\begin{proposition}\label{prop:atomic_generation}
	Let $ \Dcal $ be an atomically generated presentable \category and let $ i \colon \Ccal \hookrightarrow \Dcal $ be the inclusion of a full subcategory.
	If $ \Ccal $ is closed under both limits and colimits in $ \Dcal $, then:
	\begin{enumerate}
		\item\label{prop:atomic_generation.1} The \category $ \Ccal $ is presentable and the inclusion $ i \colon \incto{\Ccal}{\Dcal} $ admits both a left adjoint $ L \colon \fromto{\Dcal}{\Ccal} $ and a right adjoint $ R \colon \fromto{\Dcal}{\Ccal} $.
		
		\item\label{prop:atomic_generation.2} The \category $ \Ccal $ is atomically generated by $ L(\Dcal^{\at}) $.
		
		\item\label{prop:atomic_generation.3} Let $ W_L \subset \Mor(\Dcal) $ be the collection of $ L $-equivalences.
		Let $ W \subset W_L \intersect \Mor(\Dcal^{\at}) $ be a subset of morphisms with the property that $ \Ccal $ coincides with the subcategory of $ W $-local objects of $ \Dcal $.
		Then the functor
		\begin{equation*} 
			L \colon \Dcal^{\at} \to \Ccal^{\at} 
		\end{equation*}
		exhibits $\Ccal^{\at}$ as the idempotent completion of the localization $\Dcal^{\at}[W\inv]$.
	\end{enumerate}
\end{proposition}

\begin{proof}
	Point (1) is a direct consequence of the reflection theorem of Ragimov--Schlank \cite[Theorem 1.1]{RagimovSchlank}.

	To prove (2), first note that by \Cref{obs:atomic_functors_preserve_atomic_objects}, the functor $ L $ preserves atomic objects.
	Hence \HTT{Proposition}{5.1.6.10} implies that the functor
	\begin{equation*} 
		\PSh(L(\Dcal^{\at})) \to \Ccal 
	\end{equation*}
	given by the left Kan extension of the inclusion $ L(\Dcal^{\at}) \subset \Ccal $ along the Yoneda embedding is fully faithful.
	To complete the proof of (2), by \Cref{lem:atomic_generation_via_conservativity} we need to show that the functors 
	\begin{equation*}
		\set{\Map_{\Ccal}(L(d),-) \colon \Ccal \to \Spc}_{d \in \Dcal^{\at}}
	\end{equation*}
	are jointly conservative.
	By adjunction, $ \Map_{\Ccal}(L(d),-) \equivalent \Map_{\Dcal}(d,i(-)) $.
	Since $ \Dcal $ is atomically generated, \Cref{lem:atomic_generation_via_conservativity} shows that the functors
	\begin{equation*}
		\set{\Map_{\Dcal}(d,-) \colon \Dcal \to \Spc}_{d \in \Dcal^{\at}}
	\end{equation*}
	are jointly conservative; since $ i $ is also conservative, the claim follows.
	
	Now we prove (3).
	Combining (2) with \cref{obs:atomic_generation_and_idempotent_completion} shows that the inclusion $ L(\Dcal^{\at}) \subset \Ccal^{\at}$ exhibits $ \Ccal^{\at} $ as the idempotent completion of $ L(\Dcal^{\at}) $.
	Thus it suffices to prove that the functor
	\begin{equation}\label{eq:atomic_objects}
		L \colon \Dcal^{\at} \to L(\Dcal^{\at})
	\end{equation}
	exhibits $L(\Dcal^{\at})$ as the localization of $\Dcal^{\at}$ at $ W $.
	To see this, we apply the three criteria of \cite[Proposition 7.1.11]{Cisinski_Higher_Category}.
	By definition, the functor \eqref{eq:atomic_objects} is essentially surjective.
	Moreover, upon passing to presheaves, precomposition with \eqref{eq:atomic_objects} is identified with $ i \colon \Ccal \hookrightarrow \Dcal$ via the restricted Yoneda functor from \cref{rem:restricted_yoneda}.
	Hence, precomposition with \eqref{eq:atomic_objects} is fully faithful with image contained in the full subcategory of presheaves $ F \in \PSh(\Dcal^{\at}) $ that invert $ W $.
	Via the restricted Yoneda functor, presheaves $ F \in \PSh(\Dcal^{\at}) $ that invert $ W $ correspond to $ W $-local objects of $ \Dcal $, that is objects of $ \Ccal $. 
	Thus, \cite[Proposition 7.1.11]{Cisinski_Higher_Category} applies and concludes the proof of (3).
\end{proof}


\subsection{Sheaves with coefficients \& tensor products of presentable \texorpdfstring{$\infty$}{∞}-categories}\label{subsec:sheaves_with_coefficients_and_tensor_products_of_presentable_categories}

We now fix our conventions for sheaves with coefficients in a presentable \category.
For this, we make use of the tensor product of presentable \categories; we refer the reader to \cite[\HAsubsec{4.8.1}]{HA} for a background.

\begin{notation}\label{ntn:ShX_E}
	Let $ \X $ be \atopos and let $ \Ecal $ be a presentable \category.
	We write $ \Sh(\X;\Ecal) $ for the tensor product of presentable \categories
	\begin{equation*}
		\Sh(\X;\Ecal) \colonequals \X \tensor \Ecal \period
	\end{equation*}
	Since the tensor product $ (-) \tensor \Ecal $ defines a functor $ \fromto{\PrR}{\PrR} $, the assignment $ \goesto{\X}{\Sh(\X;\Ecal)} $ defines a functor
	\begin{equation*}
		\Sh(-;\Ecal) \colon \fromto{\RTop}{\PrR} \period
	\end{equation*}
\end{notation}

\begin{notation}[(sheaves on \sites)]
	Let $ (\Ccal,\tau) $ be \asite and $ \Ecal $ be a presentable \category. 
	We write
	\begin{equation*} 
		\PSh(\Ccal;\Ecal) \colonequals \Fun(\Ccal^{\op}, \Ecal) 
	\end{equation*}
	for the \category of $ \Ecal $-valued presheaves on $ \Ccal $.
	We also write
	\begin{equation*}
		\Sh_{\tau}(\Ccal;\Ecal) \subset \PSh(\Ccal;\Ecal)
	\end{equation*}
	for the full subcategory spanned by $ \Ecal $-valued presheaves that satisfy $\tau$-descent.
	When $\Ecal = \Spc$, we write
	\begin{equation*}
		\Sh_{\tau}(\Ccal) \colonequals \Sh_{\tau}(\Ccal;\Spc) \period
	\end{equation*}
\end{notation}	

\begin{nul}
	The \categories $\PSh(\Ccal;\Ecal)$ and $\Sh_{\tau}(\Ccal;\Ecal)$ are naturally identified with the tensor products of presentable \categories $\PSh(\Ccal) \tensor \Ecal$ and $\Sh_{\tau}(\Ccal) \tensor \Ecal $ \cite[\SAGthm{Remark}{1.3.1.6} \& \SAGthm{Proposition}{1.3.1.7}]{SAG}.
	This justifies \Cref{ntn:ShX_E}.
\end{nul}

\begin{nul}[(hypersheaves)]\label{nul:hypersheaves}
	Let $ (\Ccal,\tau) $ be \asite.
	In this paper, we often make use of the theory of \textit{hypersheaves}.
	When $ \Ecal $ is the \category of spaces, hypersheaves can be defined intrinsically in the \topos $\Sh_{\tau}(\Ccal)$ as \defn{hypercomplete objects}, that is, objects that are local with respect to $ \infty $-connected maps.
	Hypersheaves thus form a full subcategory \smash{$\Shhyp_{\tau}(\Ccal) \subset \Sh_{\tau}(\Ccal)$}.
	It is then possible to \emph{define} hypersheaves with coefficients in $ \Ecal $ as the tensor product
	\begin{equation*} 
		\Shhyp_{\tau}(\Ccal;\Ecal) \colonequals \Shhyp_{\tau}(\Ccal) \tensor \Ecal \period 
	\end{equation*}
	Each of the inclusions
	\begin{equation*} 
		\Shhyp_{\tau}(\Ccal) \subset \PSh(\Ccal) \andeq \Shhyp_{\tau}(\Ccal) \subset \Sh_{\tau}(\Ccal) 
	\end{equation*}
	admits a left exact left adjoint adjoint.
	We refer the reader unfamiliar with hypercomplete objects and hypercompletion to \cite[\S\S \HTTsubseclink{6.5.2}--\HTTsubseclink{6.5.4}]{HTT} or \cite[\S 3.11]{arXiv:1807.03281} for further reading on the subject.
\end{nul}

\begin{notation}[(sheaves on topological spaces)]\label{ntn:sheaves_on_topological_spaces}
	Let $ X $ be a topological space and let $ \Ecal $ be a presentable \category.
	We write $ \Open(X) $ the poset of open subsets of $ X $, ordered by inclusion.
	We regard $ \Open(X) $ as a site with the covering families given by open covers.
	We write
	\begin{equation*}
		\Sh(X;\Ecal) \colonequals \Sh(\Open(X);\Ecal) \andeq \Shhyp(X;\Ecal) \colonequals \Shhyp(\Open(X);\Ecal) \period
	\end{equation*}
\end{notation}

\begin{notation}[(functoriality)]\label{ntn:functoriality_of_sheaves}
	Let $\flowerstar \colon \X \to \Y$ be a geometric morphism.
	We write $\fupperstar$ for its left exact left adjoint.
	If $\flowerstar$ is an étale geometric morphism, we write $\flowersharp$ for the left adjoint to $\fupperstar$.
	Fix a presentable \category $ \Ecal $.
	Then the functoriality of the tensor product in \smash{$\PrL$} provides us with a colimit-preserving functor
	\begin{equation*}
	 \fupperstar \tensor \cE \colon \fromto{\Sh(\Y;\Ecal)}{\Sh(\X;\Ecal)} \period 
	\end{equation*}
	When there is no risk of confusion, we simply write $\fupperstar$ instead of $\fupperstar \tensor \cE$.
	Similarly, we write $\flowerstar$ for its right adjoint, and we apply a similar convention for $\flowersharp$.
\end{notation}


\subsection{Locally constant objects \& monodromy}\label{subsec:locally_constant_objects_and_monodromy}

We now recall the basics of locally constant objects in \topoi and monodromy.
We also prove a few foundational results that we need later in the paper, but are not available elsewhere.
For more background, we refer the reader to \cites[\HAsec{A.1}]{HA}[\S3.1]{arXiv:2105.12417}.

\begin{notation}[(constant objects and global sections)]
	Let $ \X $ be \atopos.
	We write
	\begin{align*}
		\Gamma_{\X,\ast} \colon \X &\to \Spc \\ 
	\intertext{for the \defn{global sections} functor given by}
		U &\mapsto \Map_{\X}(1_{\X},U) \period
	\end{align*}
	The global sections functor admits a left exact left adjoint $ \Gammaupperstar_{\X} \colon \fromto{\Spc}{\X} $ called the \defn{constant sheaf functor}.
	If the \topos $ \X $ is clear from the context, we write $ \Gammaupperstar $ and $ \Gammalowerstar $ for $ \Gammaupperstar_{\X} $ and $ \Gamma_{\X,\ast} $, respectively.

	Given a presentable \category $ \Ecal $, we say that an $ F $ object of $ \Sh(\X;\Ecal) $ is \defn{constant} if $ F $ lies in the image of the functor
	\begin{equation*}
		\Gammaupperstar \tensor \Ecal \colon \fromto{\Ecal}{\Sh(\X;\Ecal)} \period
	\end{equation*}
\end{notation}

\begin{nul}
	Note that $ \Spc $ is the terminal \topos \HTT{Proposition}{6.3.4.1}, so $ \Gammalowerstar $ is the unique geometric morphism $ \fromto{\X}{\Spc} $.
\end{nul}

\begin{recollection}[(products of \topoi)]\label{rec:products_of_topoi}
	The product in $ \RTop $ is given by the \textit{tensor product} in $ \PrR $; see \cites[\HAthm{Example}{4.8.1.19}]{HA}[Theorem 2.1.5]{arXiv:1802.10425}.
	In particular:
	\begin{enumerate}
		\item\label{rec:products_of_topoi.1} If $ \fupperstar \colon \fromto{\X'}{\X} $ and $ \gupperstar \colon \fromto{\Y'}{\Y} $ are left exact left adjoints between \topoi, then 
		\begin{equation*}
			\fupperstar \tensor \gupperstar \colon \fromto{\X' \tensor \Y'}{\X \tensor \Y}
		\end{equation*}
		is also a left exact left adjoint between \topoi.

		\item\label{rec:products_of_topoi.2} The functor
		\begin{equation*}
			\Gammaupperstar_{\X} \tensor \Gammaupperstar_{\Y} \colon \fromto{\Spc \equivalent \Spc \tensor \Spc}{\X \tensor \Y}
		\end{equation*}
		is the constant sheaf functor.
	\end{enumerate}
\end{recollection}

\begin{definition}[(locally constant objects)]
	Let $ \X $ be \atopos and let $ \Ecal $ be a presentable \category.
	An object $ F \in \Sh(\X;\Ecal) $ is \defn{locally constant} if there exists an effective epimorphism $ \surjto{\coprod_{i \in I} U_i}{1_{\X}} $ such that for each $ i \in I $, the image of $ F $ under the natural pullback functor 
	\begin{equation*}
		\fromto{\Sh(\X;\Ecal)}{\Sh(\X_{/U_i};\Ecal)}
	\end{equation*}
	is a constant object.
	We write
	\begin{equation*}
		\LC(\X;\Ecal) \subset \Sh(\X;\Ecal)
	\end{equation*}
	for the full subcategory spanned by the locally constant objects.
	When $ \Ecal = \Spc $, we simply write $ \LC(\X) \subset \X $ for $ \LC(\X;\Spc) $.
\end{definition}

\begin{observation}
	Given a geometric morphism of \topoi $ \flowerstar \colon \fromto{\X}{\Y} $, the pullback functor $ \fupperstar \colon \fromto{\Y}{\X} $ carries $ \LC(\Y;\Ecal) $ to $ \LC(\X;\Ecal) $.
\end{observation}

\noindent This recovers the usual notion of local constancy for (hyper)sheaves on topological spaces:

\begin{example}\label{ex:local_constancy_for_sheaves_on_topological_spaces}
	Let $ X $ be a topological space and let $ \X $ be either $ \Sh(X) $ or $ \Shhyp(X) $.
	An object $ F \in \Sh(\X;\Ecal) $ is locally constant if and only if there exists an open cover $ \{U_i\}_{i \in I} $ of $ X $ such that each restriction $ \restrict{F}{U_i} $ is constant.
	See \cite[Proposition 1.18]{arXiv:2102.12325}.
\end{example}

\begin{definition}[{(monodromic \topos)}]\label{def:monodromic_topos}
	We say that \atopos $ \X $ is \defn{monodromic} or \defn{locally of constant shape} if the constant sheaf functor \smash{$ \Gammaupperstar \colon \fromto{\Spc}{\X} $} admits a left adjoint 
	\begin{equation*}
		\Gammalowersharp \colon \fromto{\X}{\Spc} \period
	\end{equation*}
	In this case, we write $ \Piinfty(\X) \colonequals \Gammalowersharp(1_{\X}) $ and call $ \Piinfty(\X) $ the \defn{shape} of $ \X $.
\end{definition}

The following result of Lurie justifies the terminology in \Cref{def:monodromic_topos}:

\begin{recollection}[(monodromy)]\label{rec:monodromy}
	Let $ \X $ be a monodromic \topos.
	Then the full subcategory $ \LC(\X) \subset \X $ is closed under limits and colimits.
	Moreover, there is a natural equivalence
	\begin{equation*}
		\equivto{\LC(\X)}{\Fun(\Piinfty(\X),\Spc)}
	\end{equation*}
	See \cite[\HAappthm{Proposition}{A.1.6} \& \HAappthm{Theorem}{A.1.15}]{HA}.
	Furthermore, for any presentable \category $ \Ecal $, there is an equivalence
	\begin{equation*}
		\LC(\X) \tensor \Ecal \equivalence \LC(\X;\Ecal) \period
	\end{equation*}
	See \cite[Proposition 3.1.7]{arXiv:2105.12417}.
	In particular, $ \LC(\X;\Ecal) \subset \Sh(\X;\Ecal) $ is closed under limits and colimits.
\end{recollection}	

\begin{example}[(monodromy in topology)]\label{ex:monodrom_in_topology}
	Let $ X $ be a topological space.
	\begin{enumerate}
		\item\label{ex:monodrom_in_topology.1} If $ X $ is locally weakly contractible, then the \topos $ \Shhyp(X) $ is monodromic.
		The functor
		\begin{equation*}
			 \Gammalowersharp \colon \fromto{\Shhyp(X)}{\Spc} 
		\end{equation*}
		is given by extending the functor sending an open $ U \subset X $ to the underlying homotopy type%
		\footnote{Slightly more precisely, the underlying \textit{weak} homotopy type of $ U $.
		That is, the image of $ U $ under the functor $ \Top \to \Spc $ that exhibits the \category of spaces as the \categorical localization of the $ 1 $-category of topological spaces obtained by inverting the weak homotopy equivalences.
		However, since it is standard to do so and we have no need to distinguish between weak and strong homotopy types, we omit the term `weak' in this paper.} 
		of $ U $ along colimits.
		In particular \smash{$ \Piinfty(\Shhyp(X)) $} coincides with the underlying homotopy type of $ X $.
		See \cite[Proposition 2.4]{arXiv:2010.06473}.

		\item\label{ex:monodrom_in_topology.2} If $ X $ is paracompact and \textit{locally of singular shape} in the sense of \HAa{Definition}{A.4.15}, then the \topos $ \Sh(X) $ is monodromic.
		Again, the functor $ \Gammalowersharp \colon \fromto{\Sh(X)}{\Spc} $ is given by extending the functor sending an open $ U \subset X $ to the underlying homotopy type of $ U $ along colimits.
		In particular $ \Piinfty(\Sh(X)) $ coincides with the underlying homotopy type of $ X $.
		See \HAa{Theorem}{A.4.19}.
	\end{enumerate}
\end{example}

An intriguing fact is that any \topos étale over a monodromic \topos is also monodromic:

\begin{observation}\label{obs:X_monodromic_implies_all_slices_are_monodromic}
	Let $ \X $ be a monodromic \topos and let $ U \in \X $. 
	Then the slice \topos $ \X_{/U} $ is monodromic.
	To see this, note that the composite
	\begin{equation*}
		\begin{tikzcd}[sep=2.5em]
			\X_{/U} \arrow[r, "\textup{forget}"] & \X \arrow[r, "\Gamma_{\X,\sharp}"] & \Spc
		\end{tikzcd}
	\end{equation*} 
	is left adjoint to the constant sheaf functor.
	As a consequence, we see that 
	\begin{equation*}
		\Piinfty(\X_{/U}) = \Gamma_{\X,\sharp}(U) \period
	\end{equation*}
\end{observation}

We now explain the functoriality of the monodromy equivalence.
To do so, we need the following lemma.

\begin{lemma}\label{lem:left_exact_left_adjoints_between_slices_over_Spc}
	Let $ K,L \in \Spc $ and let
	\begin{equation*}
		\fupperstar \colon \fromto{\Fun(L,\Spc)}{\Fun(K,\Spc)}
	\end{equation*}
	be a functor.
	The following are equivalent:
	\begin{enumerate}
		\item There exists a map of spaces $ f \colon \fromto{K}{L} $ such that $ \fupperstar $ is equivalent to the functor given by precomposition with $ f $.

		\item The functor $ \fupperstar $ preserves limits and colimits.

		\item The functor $ \fupperstar $ is left exact and preserves colimits.
	\end{enumerate}
\end{lemma}

\begin{proof}
	Since every space is an idempotent complete \category (see \Cref{lem:layered_implies_idempotent_complete}), the equivalence (1) $ \Leftrightarrow $ (2) follows from \Cref{recollection:atomically_generated}.
	Clearly (2) $ \Rightarrow $ (3).
	For the remaining implication (3) $ \Rightarrow $ (2), let $ \flowerstar $ denote the right adjoint to $ \fupperstar $.
	By assumption, $ \flowerstar $ is a geometric morphism.
	Note that by the straightening/unstraightening equivalences
	\begin{equation*}
		\Fun(K,\Spc) \equivalent \Spc_{/K} \andeq \Fun(L,\Spc) \equivalent \Spc_{/L} \comma
	\end{equation*}
	the unique geometric morphisms to the terminal \topos
	\begin{equation*}
		\fromto{\Fun(K,\Spc)}{\Spc} \andeq \fromto{\Fun(L,\Spc)}{\Spc}
	\end{equation*}
	are étale.
	Hence \HTT{Corollary}{6.3.5.9} implies that $ \flowerstar $ is an étale geometric morphism; in particular, $ \fupperstar $ admits a left adjoint.
\end{proof}

\begin{corollary}\label{cor:all_morphisms_are_monodromic}
	Let $ \flowerstar \colon \fromto{\X}{\Y} $ be a geometric morphism between monodromic \topoi.
	Then the functor
	\begin{equation*}
		\fupperstar \colon \fromto{\LC(\X)}{\LC(\Y)}
	\end{equation*}
	preserves limits and colimits.
\end{corollary}

\begin{proof}
	Since $ \X $ and $ \Y $ are monodromic, $ \LC(\X) \subset \X $ and $ \LC(\Y) \subset \Y $ are closed under limits and colimits.
	The claim now follows from the monodromy equivalences for $ \X $ and $ \Y $ combined with \Cref{lem:left_exact_left_adjoints_between_slices_over_Spc}.
\end{proof}

\begin{notation}
	Write \smash{$ \RTopmon \subset \RTop $} for the full subcategory spanned by the monodromic \topoi.
\end{notation}

\begin{notation}\label{ntn:PrRat}
	Write $ \PrRat \subset \PrR $ for the (non-full) subcategory of $ \PrR $ with objects the atomically generated presentable \categories and morphisms functors that are both left and right adjoints.
\end{notation}

\begin{nul}
	Note that the equivalence $ \PrL \equivalent (\PrR)^{\op} $ given by passing to right adjoints restricts to an equivalence
	\begin{equation*}
		\PrLat \equivalent (\PrRat)^{\op} \period
	\end{equation*}
\end{nul}

\begin{observation}[(functoriality of the shape)]
	The assignment $ \goesto{\X}{\Piinfty(\X)} $ refines to a functor
	\begin{equation*}
		\Piinfty \colon \fromto{\RTopmon}{\Spc \subset \Catidem} \period
	\end{equation*}
	Specifically, this functor is given by the composite
	\begin{equation*}
		\begin{tikzcd}[sep=3em]
			\RTopmon \arrow[r, "\LC(-)"] & (\PrRat)^{\op} \equivalent \PrLat \arrow[r, "(-)^{\ex}", "\sim"'{yshift=0.25ex}] & \Catidem \comma
		\end{tikzcd}
	\end{equation*}
	where the left-hand functor sends $ \X $ to the \category \smash{$ \LC(\X) $} with functoriality given by pullback, and the right-hand functor sends an atomically generated \category $ \Ccal $ to the \category \smash{$ \Ccal^{\ex} = (\Ccal^{\at})^{\op} $} given by the opposite of the subcategory of atomic objects.
\end{observation}

We conclude with a Künneth formula for the shape of a product of monodromic \topoi.

\begin{recollection}\label{rec:the_equivalence_of_Spc_tensor_Spc_with_Spc}
	The natural equivalence
	\begin{equation*}
		\equivto{\Spc \tensor \Spc}{\Spc}
	\end{equation*}
	is induced by the functor
	\begin{align*}
		\Spc \cross \Spc &\to \Spc \\ 
		(K,L) &\mapsto K \cross L \period
	\end{align*}
\end{recollection}

\begin{observation}
	Let $ \X $ and $ \Y $ be monodromic \topoi.
	Since the inclusions
	\begin{equation*}
		\incto{\LC(\X)}{\X} \andeq \incto{\LC(\Y)}{\Y}
	\end{equation*}
	are both left and right adjoints, the functor
	\begin{equation*}
		\fromto{\LC(\X) \tensor \LC(\Y)}{\X \tensor \Y}
	\end{equation*}
	induced by the functoriality of the tensor product is also fully faithful and both a left and right adjoint.
\end{observation}

\begin{proposition}[{(Künneth formula for monodromic \topoi)}]\label{prop:Kunneth_formula_for_monodromic_topoi}
	Let $ \X $ and $ \Y $ be monodromic \topoi.
	Write $ \Gamma_{\X,\sharp} \colon \fromto{\X}{\Spc} $ and $ \Gamma_{\Y,\sharp} \colon \fromto{\Y}{\Spc} $ for the left adjoints to the constant sheaf functors.
	Then:
	\begin{enumerate}
		\item \label{prop:Kunneth_formula_for_monodromic_topoi.1} The functor
		\begin{equation*}
			\Gamma_{\X,\sharp} \tensor \Gamma_{\Y,\sharp} \colon \fromto{\X \tensor \Y}{\Spc \tensor \Spc \equivalent \Spc}
		\end{equation*}
		is left adjoint to the constant sheaf functor $ \fromto{\Spc}{\X \tensor \Y} $.
		In particular, the \topos $ \X \tensor \Y $ is monodromic.

		\item \label{prop:Kunneth_formula_for_monodromic_topoi.2} The natural map $ \fromto{\Piinfty(\X \tensor \Y)}{\Piinfty(\X) \cross \Piinfty(\Y)} $ is an equivalence.

		\item \label{prop:Kunneth_formula_for_monodromic_topoi.3} The natural fully faithful functor
		\begin{equation*}
			\fromto{\LC(\X) \tensor \LC(\Y)}{\X \tensor \Y}
		\end{equation*}
		has image $ \LC(\X \tensor \Y) $.
	\end{enumerate}
\end{proposition}

\begin{proof}
	For (1), note that by the functoriality of the tensor product of presentable \categories, $ \Gamma_{\X,\sharp} \tensor \Gamma_{\Y,\sharp} $ is left adjoint to the functor
	\begin{equation*}
		\Gammaupperstar_{\X} \tensor \Gammaupperstar_{\Y} \colon \fromto{\Spc \equivalent \Spc \tensor \Spc}{\X \tensor \Y} \period
	\end{equation*}
	By \enumref{rec:products_of_topoi}{2}, $ \Gammaupperstar_{\X} \tensor \Gammaupperstar_{\Y} $ is the constant sheaf functor; hence $ \Gamma_{\X,\sharp} \tensor \Gamma_{\Y,\sharp} $ is left adjoint to the constant sheaf functor, as desired.

	For (2), note that by \Cref{rec:products_of_topoi,rec:the_equivalence_of_Spc_tensor_Spc_with_Spc}, the functor 
	\begin{equation*}
		\Gamma_{\X,\sharp} \tensor \Gamma_{\Y,\sharp} \colon \fromto{\X \tensor \Y}{\Spc \tensor \Spc \equivalent \Spc}
	\end{equation*}
	is induced by the functor
	\begin{align*}
		\X \cross \Y &\to \Spc \\ 
		(F,G) &\mapsto \Gamma_{\X,\sharp}(F) \cross \Gamma_{\Y,\sharp}(G) \period
	\end{align*}
	In particular, applying $ \Gamma_{\X,\sharp} \tensor \Gamma_{\Y,\sharp} $ to the terminal object $ 1_{\X \tensor \Y} = 1_{\X} \tensor 1_{\Y} $, we have natural identifications
	\begin{align*}
		\Piinfty(\X \tensor \Y) &= (\Gamma_{\X,\sharp} \tensor \Gamma_{\Y,\sharp})(1_{\X} \tensor 1_{\Y}) \\
		&= \Gamma_{\X,\sharp}(1_{\X}) \cross \Gamma_{\Y,\sharp}(1_{\Y}) \\ 
		&= \Piinfty(\X) \cross \Piinfty(\Y) \period
	\end{align*}

	Item (3) is an immediate consequence of (2) and the formula
	\begin{equation*}
		\Fun(\Ccal,\Spc) \tensor \Fun(\Dcal,\Spc) \equivalent \Fun(\Ccal \cross \Dcal,\Spc) \period \qedhere
	\end{equation*}
\end{proof}


\section{Exit-path \texorpdfstring{$\infty$}{∞}-categories}\label{sec:exit-path_categories}

In this section, we introduce \textit{exodromic stratified \topoi} and their \textit{exit-path \categories}.
See \Cref{def:exit_path}.
These definitions are topos-theoretic generalizations of Clausen and Jansen's definition in the topological setting \cite[Definition 3.5]{arXiv:2108.01924}.

In \cref{subsec:stratified_topoi_and_stratified_spaces}, we start by reviewing the basics of the theory of stratified \topoi introduced in \cite{arXiv:1807.03281}.
In \cref{subsec:constructible_objects_and_exit-path_categories}, we explain the basics of constructible objects in stratified \topoi; we also define exodromic stratified \topoi their exit-path \categories.
In \cref{subsec:exodromic_morphisms}, we discuss stratified morphisms that induce morphisms on the level of exit-path \categories.
In \cref{subsec:exodromy_and_hypercompletion}, we conclude with some results on the interaction between exodromic stratified \topoi and hypercompletion.


\subsection{Stratified \texorpdfstring{$\infty$}{∞}-topoi \& stratified spaces}\label{subsec:stratified_topoi_and_stratified_spaces}

We now recall the theory of stratifications of \topoi introduced in \cite[\S8.2]{arXiv:1807.03281}.
This theory directly generalizes the theory of stratifications of topological spaces, but also applies to more general contexts such as stratifications of schemes and topological stacks.
The starting point for the theory is the observation that hypersheaves on a poset $ P $ equipped with the Alexandroff topology are just functors out of $ P $:

\begin{recollection}[{\cites[Example A.11]{MR4549105}[Example 3.12.15]{arXiv:1807.03281}}]\label{rec:hypersheaves_on_a_poset}
	Let $ \Ecal $ be a presentable \category and let $ P $ be a poset.
	Regard $ P $ as a topological space with the Alexandroff topology.
	Then there is a natural equivalence of \categories
	\begin{equation*}
		\equivto{\Fun(P,\Ecal)}{\Shhyp(P;\Ecal)} \period
	\end{equation*}
\end{recollection}

\begin{warning}
	It is necessary that we work with hypersheaves in \Cref{rec:hypersheaves_on_a_poset}; in general, $ \Sh(P) $ is not hypercomplete.
	See \cite[Example A.13]{MR4549105}.
\end{warning}

\begin{example}[{\cite[Lemma 5.21]{MR4651175}}]\label{ex:sheaves_on_a_noetherian_poset}
	If $ P $ is a noetherian poset, then $ \Sh(P) $ is hypercomplete, hence
	\begin{equation*}
		\Sh(P) \equivalent \Fun(P,\Spc) \period
	\end{equation*}
\end{example}

\begin{definition}[(stratified \topos)]
	Let $ \X $ be \atopos and let $ P $ be a poset.
	A \defn{$ P $-stratification} of $ \X $ is a geometric morphism
	\begin{equation*}
		\slowerstar \colon \fromto{\X}{\Fun(P,\Spc)} \period
	\end{equation*}
	To simplify notation, we often abusively denote a stratified \topos by $ (\X,P) $.
\end{definition}

Morphisms of stratified \topoi are commutative squares.
Here is the easiest way to formulate this.

\begin{notation}
	We write $ \Poset $ for the category of posets.
\end{notation}

\begin{definition}\label{def:StrTop}
	The \defn{\category of stratified \topoi} is the pullback
	\begin{equation*}
		\begin{tikzcd}
			\StrTop \arrow[r] \arrow[d] \arrow[dr, phantom, very near start, "\lrcorner", xshift=-0.75em, yshift=0.25em] & \Poset \arrow[d, "{\Fun(-,\Spc)}"] \\ 
			\Fun([1],\RTop) \arrow[r] & \RTop \period
		\end{tikzcd}
	\end{equation*}
	Here, the bottom horizontal functor sends a geometric morphism $ \slowerstar \colon \fromto{\X}{\Pcal} $ to its target $ \Pcal $.
\end{definition}

\begin{nul}
	Said differently, given stratified \topoi $ \slowerstar \colon \fromto{\X}{\Fun(P,\Spc)}  $ and $ \tlowerstar \colon \fromto{\Y}{\Fun(Q,\Spc)}  $, a \defn{morphism of stratifed \topoi} $ \fromto{(\X,P)}{(\Y,Q)} $ consists of a commutative square of geometric morphisms
	\begin{equation*} 
		\begin{tikzcd}[column sep=4em]
			\X \arrow[r, "\flowerstar"] \arrow[d, "\slowerstar"'] & \Y \arrow[d, "\tlowerstar"] \\
			\Fun(P,\Spc) \arrow[r, "\philowerstar"'] & \Fun(Q,\Spc) 
		\end{tikzcd} 
	\end{equation*}
	such that the pushforward functor $ \philowerstar $ is induced by a map of posets $ \phi \colon \fromto{P}{Q} $ (equivalently, $ \phiupperstar $ preserves limits).
	To simplify notation, we abusively denote a morphism of stratified \topoi by $ \flowerstar \colon \fromto{(\X,P)}{(\Y,Q)} $.
\end{nul}

It is often convenient to pull back a $ P $-stratified \topos to a locally closed subposet of $ P $:

\begin{recollection}[(locally closed subposets)]
	Let $ P $ be a poset.
	Then a subset $ S \subset P $ is locally closed in the Alexandroff topology if and only if $ S $ is an \defn{interval}: given $ p,q \in S $ with $ p \leq q $, $ S $ contains all $ x \in P $ such that $ p \leq x \leq q $.
\end{recollection}

\begin{notation}\label{ntn:inclusion_i_S}
	Let $ \slowerstar \colon \fromto{\X}{\Fun(P,\Spc)} $ be a stratified \topos and let $ i \colon \incto{S}{P} $ be a locally closed subposet.
	We write $ \XS $ for the pullback
	\begin{equation*}
		\begin{tikzcd}[sep=2.25em]
			\XS \arrow[dr, phantom, very near start, "\lrcorner", xshift=-0.5em] \arrow[d] \arrow[r, "i_{S,*}", hooked] & \X  \arrow[d, "\slowerstar"] \\ 
			\Fun(S,\Spc) \arrow[r, hooked, "\ilowerstar"'] & \Fun(P,\Spc) \period
		\end{tikzcd}
	\end{equation*}
	computed in $ \RTop $.
	Observe that $ \iSlowerstar $ and $ \ilowerstar $ define a morphism of stratified \topoi $ \incto{(\XS,S)}{(\X,P)} $.
	For each $ p \in P $, we call $ \X_{p} \colonequals \X_{\{p\}} $ the \defn{$ p $-th stratum} of $ (\X,P) $.	
\end{notation}

\begin{nul}
	Note that if $ S \subset P $ is open, then \smash{$ \iSlowerstar $} is an open immersion of \topoi and if $ S \subset P $ is closed, then \smash{$ \iSlowerstar $} is a closed immersion of \topoi.
	Hence for $ S \subset P $ locally closed, \smash{$ \iSlowerstar $} is a locally closed immersion of \topoi.
	(See \Cref{appendix:complements_on_topoi} for background on locally closed immersions of \topoi.)
\end{nul}

\begin{observation}\label{obs:pulling_back_stratified_geometric_morphisms}
	Let $ (\flowerstar,\phi) \colon \fromto{(\X,P)}{(\Y,Q)} $ be a morphism of stratified \topoi and let $ T \subset Q $ be a locally closed subposet.
	Write $ P_{T} \colonequals \phi\inv(T) $, so that $ P_T $ is a locally closed subposet of $ P $.
	Then we have a commutative cube of \topoi and geometric morphisms
	\begin{equation*}
        \begin{tikzcd}[column sep={15ex,between origins}, row sep={10ex,between origins}]
            \X_{P_T} \arrow[rr, "f_{T,\ast}"] \arrow[dd, hooked]  \arrow[dr, hooked, "i_{P_T,*}" description] & & \Y_{T} \arrow[dd]  \arrow[dr, hooked, "i_{T,*}"] \\
            & \X \arrow[rr, "\flowerstar"{near start}, crossing over] & & \Y \arrow[dd, hooked] \\
            \Fun(P_T,\Spc) \arrow[rr] \arrow[dr, hooked] & & \Fun(T,\Spc) \arrow[dr, hooked] \\
            & \Fun(P,\Spc) \arrow[rr, "\philowerstar"'] \arrow[from=uu, crossing over] & & \Fun(Q,\Spc) 
        \end{tikzcd}
    \end{equation*} 
    In particular, the induced geometric morphism on pullbacks $ f_{T,*} \colon \fromto{\X_{P_T}}{\Y_{T}} $ refines to a morphism of stratified \topoi
    \begin{equation*}
    	(f_{T,*}, (\restrict{\phi}{P_T})_{*}) \colon \fromto{(\X_{P_T},P_T}{(\Y_{T},T)} \period
    \end{equation*}
\end{observation}

In this paper, our main examples of stratified \topoi come from stratified topological spaces.

\begin{example}[(stratified \topoi attached to stratified spaces)]
	Let $ s \colon \fromto{X}{P} $ be a stratified space.
	\begin{enumerate}
		\item Then
		\begin{equation*}
			\slowerstarhyp \colon \fromto{\Shhyp(X)}{\Shhyp(P) \equivalent \Fun(P,\Spc)}
		\end{equation*}
		is a $ P $-stratified \topos.

		\item If $ P $ is noetherian, then
		\begin{equation*}
			\slowerstar \colon \fromto{\Sh(X)}{\Sh(P) \equivalent \Fun(P,\Spc)}
		\end{equation*}
		is a $ P $-stratified \topos.
	\end{enumerate}
\end{example}

\begin{example}
	Let $ s \colon \fromto{X}{P} $ be a stratified topological stack in the sense of \cite[Definition 3.1]{arXiv:2308.09550}.
	If $ P $ is noetherian, then $ \slowerstar \colon \fromto{\Sh(X)}{\Fun(P,\Spc)} $ is a $ P $-stratified \topos.
\end{example}

\begin{notation}
	Let $ (X,P) $ be a stratified space and $ S \subset P $ a locally closed subposet.
	Write $ X_{S} \colonequals X \cross_P S $.
	Then $ X_S $ is naturally an $ S $-stratified space.
	Moreover, the inclusions $ \incto{X_S}{X} $ and $ \incto{S}{P} $ define a morphism of stratified spaces $ i_S \colon \incto{(X,S)}{(X,P)} $.
\end{notation}

An important fact is that pulling back to a locally closed subposet commutes with taking (hyper)sheaves:

\begin{lemma}\label{lem:sheaves_on_strata}
	Let $ (X,P) $ be a stratified space and $ S \subset P $ a locally closed subposet.
	\begin{enumerate}
		\item The natural geometric morphism $ \fromto{\Shhyp(X_S)}{\Shhyp(X)_S} $ is an equivalence.

		\item If $ P $ is noetherian, then the natural geometric morphism $ \fromto{\Sh(X_S)}{\Sh(X)_S} $ is an equivalence.
	\end{enumerate} 
\end{lemma}

\begin{proof}
	Immediate from \Cref{rec:hypersheaves_on_a_poset}, \Cref{ex:sheaves_on_a_noetherian_poset}, \Cref{cor:Sh_preserves_pullbacks_along_locally_closed_embeddings}, \Cref{cor:Shhyp_preserves_pullbacks_along_locally_closed_embeddings}, and the definitions.
\end{proof}

Another useful fact is that in the noetherian setting, pulling back to strata is jointly conservative:

\begin{lemma}\label{lem:pulling_back_to_strata_is_jointly_conservative_for_noetherian_posets}
	Let $ (\X,P) $ be a stratified \topos.
	If the poset $ P $ is noetherian, then the pullback functors
	\begin{equation*}
		\set{\iupperstar_{p} \colon \fromto{\X}{\X_{p}}}_{p \in P}
	\end{equation*}
	are jointly conservative.
\end{lemma}

\begin{proof}
	Let $ \phi $ be a morphism in $ \X $ such that for each $ p \in P $, the morphism $ \iupperstar_p(\phi) $ is an equivalence; we need to show that $ \phi $ is an equivalence.
	For each $ p \in P $, write 
	\begin{equation*}
		P_{\geq p} \colonequals \setbar{q \in P}{q \geq p} \andeq P_{>p} \colonequals P_{\geq p} \sminus \{p\} \period
	\end{equation*}
	Since the open subsets $ \{P_{\geq p}\}_{p \in P} $ cover $ P $, it suffices to show:
	\begin{enumerate}
		\item[($\ast$)]\label{item:noetherian_induction_claim} For each $ p \in P $, the restriction $ \iupperstar_{P_{\geq p}}(\phi) $ is an equivalence in $ \X_{P_{\geq p}} $.
	\end{enumerate}
	We prove (\hyperref[item:noetherian_induction_claim]{$ \ast $}) by noetherian induction on $ p \in P $.
	We need to show that if the restriction \smash{$ \iupperstar_{P_{\geq q}}(\phi) $} is an equivalence for each $ q > p $, then \smash{$ \iupperstar_{P_{\geq p}}(\phi) $} is an equivalence.
	Note that
	\begin{equation*}
		P_{\geq p} \sminus \{p\} = P_{>p} = \Union_{q \in P_{>p}} P_{\geq q} \period
	\end{equation*}
	Hence the inductive hypothesis implies that the restriction \smash{$ \iupperstar_{P_{> p}}(\phi) $} is an equivalence.
	By assumption \smash{$ \iupperstar_{p}(\phi) $} is also an equivalence.
	By recollement, the restriction functors
	\begin{equation*}
		\iupperstar_p \colon \fromto{\X_{P_{\geq p}}}{\X_{p}} \andeq \iupperstar_{P_{> p}} \colon \fromto{\X_{P_{\geq p}}}{\X_{P_{>p}}}
	\end{equation*}
	are jointly conservative, completing the proof.
\end{proof}


\subsection{Constructible objects \& exit-path \texorpdfstring{$\infty$}{∞}-categories}\label{subsec:constructible_objects_and_exit-path_categories}

We now recall the basics of constructible objects of stratified \topoi introduced in \cite[\S9.4]{arXiv:1807.03281}.
We also define exit-path \categories at this level of generality.

\begin{definition}[(constructible objects)]\label{def:constructible_objects}
	Let $ (\X,P) $ be a stratified \topos and let $ \Ecal $ be a presentable \category.
	An object $ F \in \Sh(\X;\Ecal) $ is \defn{$ P $-constructible} if for each $ p \in P $, the restriction $ \iupperstar_p(F) \in \Sh(\X_p;\Ecal) $ is locally constant.
	We write
	\begin{equation*}
		\ConsP(\X;\Ecal) \subset \Sh(\X;\Ecal)
	\end{equation*}
	for the full subcategory spanned by the $ P $-constructible objects. 
	If $ \Ecal = \Spc $, we simply write $ \ConsP(\X) \subset \X $ for $ \ConsP(\X;\Spc) $.
\end{definition}

\begin{remark}
	Our terminology differs from the terminology used in \cite[\S9.4]{arXiv:1807.03281}.
	There, Barwick--Glasman--Haine use the term \textit{formally constructible objects} for what we call constructible objects; their \textit{constructible objects} are formally constructible objects that satisfy additional finiteness hypotheses.
	The reason for this is that \cite{arXiv:1807.03281} is mostly about \topoi coming from algebraic geometry, where these finiteness hypotheses are necessary for a well-behaved theory.
\end{remark}

\begin{observation}
	Given a morphism of stratified \topoi $ \flowerstar \colon \fromto{(\X,P)}{(\Y,Q)} $, the pullback functor $ \fupperstar \colon \fromto{\Y}{\X} $ carries $ \ConsQ(\Y;\Ecal) $ to $ \ConsP(\X;\Ecal) $.
\end{observation}

\noindent It is often useful to write the \category of constructible objects as a pullback:

\begin{observation}
	The \category $ \ConsP(\X;\Ecal) $ is the pullback
	\begin{equation*}
		\begin{tikzcd}[column sep =4.5em]
			\ConsP(\X;\Ecal) \arrow[r] \arrow[d, hooked] & \prod_{p \in P} \LC(\X_p;\Ecal) \arrow[d, hooked, shorten <= -.75em] \\
			\Sh(\X;\Ecal) \arrow[r, "\prod_{p \in P} \iupperstar_p"'] & \prod_{p \in P} \Sh(\X_p;\Ecal)
		\end{tikzcd} 
	\end{equation*}
\end{observation}

We use similar notation for constructible sheaves on stratified topological spaces.

\begin{notation}
	Let $ (X,P) $ be a stratified topological space and let $ \Ecal $ be a presentable \category.
	\begin{enumerate}
		\item For the natural stratified \topos $ (\X,P) = (\Shhyp(X),P) $, we write
		\begin{equation*}
			\ConsPhyp(X;\Ecal) \colonequals \ConsP(\Xcal;\Ecal) \period
		\end{equation*}

		\item If $ P $ is noetherian, then for the natural stratified \topos $ (\X,P) = (\Sh(X),P) $, we write
		\begin{equation*}
			\ConsP(X;\Ecal) \colonequals \ConsP(\Xcal;\Ecal) \period
		\end{equation*}
	\end{enumerate}
\end{notation}

\noindent \Cref{def:constructible_objects} recovers the usual notion of constructibility:

\begin{observation}
	Let $ (X,P) $ be a stratified topological space and let $ \Ecal $ be a presentable \category.
	In light of \Cref{ex:local_constancy_for_sheaves_on_topological_spaces,lem:sheaves_on_strata}:
	\begin{enumerate}
		\item An object $ F \in \Shhyp(X;\Ecal) $ is $ P $-hyperconstructible in the sense of \cite[Definition 5.2]{arXiv:2010.06473} if and only if $ F $ is $ P $-constructible in the sense of \Cref{def:constructible_objects}.

		\item Assume that $ P $ is noetherian.
		An object $ F \in \Sh(X;\Ecal) $ is $ P $-constructible in the sense of \cite[Definition 5.2]{arXiv:2010.06473} if and only if $ F $ is $ P $-constructible in the sense of \Cref{def:constructible_objects}.
	\end{enumerate}
\end{observation}

\begin{example}
	Let $ P $ be a poset.
	Then every hypersheaf on $ P $ is $ P $-constructible, i.e.,
	\begin{equation*}
		\ConsPhyp(P) = \Shhyp(P) \period
	\end{equation*}
	In light of \Cref{rec:hypersheaves_on_a_poset}, we deduce that $ \ConsPhyp(P) \equivalent \Fun(P,\Spc) $.
\end{example}

\begin{convention}\label{conven:hypersheaves_on_a_poset}
	Let $ P $ be a poset.
	We almost always implicitly identify the \categories \smash{$ \Shhyp(P;\Ecal) $}, \smash{$ \ConsPhyp(P) $}, and $ \Fun(P,\Ecal) $.
\end{convention}

For the next result, recall \Cref{ntn:Ccal^ex}.

\begin{lemma}\label{atomic_Cons_on_Poset}
	For every poset $ P $, we have natural equivalences
	\begin{equation*}
		\ConsPhyp(P)^{\ex} \equivalent \Fun(P,\Spc)^{\ex} = P \period \qedhere
	\end{equation*}
\end{lemma}

\begin{proof}
	By \Cref{lem:layered_implies_idempotent_complete}, $ P $ is idempotent complete.
	Hence the claim follows from \cref{recollection:atomically_generated,rec:hypersheaves_on_a_poset}.
\end{proof}

The following definition is a generalization of \cites[Definition 3.5]{arXiv:2108.01924}[Definition 3.10]{arXiv:2308.09550}:

\begin{definition}[(exodromic stratified \topos \& exit-path \category)]\label{def:exit_path}
	A stratified \topos
	\begin{equation*}
		\slowerstar \colon \fromto{\X}{\Fun(P,\Spc)}
	\end{equation*}
	is \defn{exodromic} if the following conditions are satisfied:
	\begin{enumerate}
		\item\label{def:exit_path.1} The \category $ \ConsP(\X) $ is atomically generated.

		\item\label{def:exit_path.2} The subcategory $\ConsP(\X) \subset \X $ is closed under both limits and colimits.

		\item\label{def:exit_path.3} The pullback functor $ \supperstar \colon \fromto{\Fun(P,\Spc)}{\X} $ preserves limits.
	\end{enumerate}
	In this case we write
	\begin{equation*}
		\Piinfty(\X,P) \colonequals \ConsP(\X)^{\ex}
	\end{equation*}
	for the \textit{opposite} of the full subcategory of $ \ConsP(\X) $ spanned by atomic objects (see \Cref{ntn:Ccal^ex}).
	We refer to $ \Piinfty(\X,P) $ as the \defn{exit-path \category} of $ (\X,P) $.
\end{definition}

\noindent The importance of the last condition of \Cref{def:exit_path} is that it provides a functor from the exit-path \category of $ (\X,P) $ to the poset $ P $.

\begin{observation}
	Let $ \slowerstar \colon \fromto{\X}{\Fun(P,\Spc)} $ be an exodromic stratified \topos.
	Then the left adjoint
	\begin{equation*}
		\slowersharpcons \colon \ConsP(\X) \to \Fun(P,\Spc)
	\end{equation*}
	to $ \supperstar $ supplied by condition (3) of \cref{def:exit_path} is atomic.
	By \cref{obs:atomic_functors_preserve_atomic_objects,atomic_Cons_on_Poset}, the functor $ \slowersharpcons $ restricts to a functor 
	\begin{equation*}
		s^{\ex} \colon \fromto{\Piinfty(\X,P)}{P} \period
	\end{equation*}
\end{observation}

Now, some important examples.

\begin{example}
	In light of \Cref{rec:monodromy}, a trivially stratified \topos $ \Gammalowerstar \colon \fromto{\X}{\Spc} $ is exodromic if and only if $ \X $ is monodromic in the sense of \Cref{def:monodromic_topos}.
\end{example}

\begin{example}[(exodromy for conically stratified spaces)]\label{ex:conically_stratified_spaces_are_exodromic}
	Let $ (X,P) $ be a conically stratified topological space in the sense of \HAa{Definition}{A.5.5}.
	\begin{enumerate}
		\item If the strata of $ (X,P) $ are locally weakly contractible, then the stratified \topos $ (\Shhyp(X),P) $ is exodromic.
		Moreover, the exit-path \category \smash{$ \Piinfty(\Shhyp(X),P) $} is given by Lurie's simplicial model for exit-paths $ \Sing(X,P) $.
		See \cite[Theorem 5.4.1]{arXiv:2211.05004}.

		\item If $ P $ is noetherian and $ X $ is paracompact and locally of singular shape, then the stratified \topos $ (\Sh(X),P) $ is exodromic.
		Again, the exit-path \category $ \Piinfty(\Sh(X),P) $ is given by Lurie's simplicial model for exit-paths $ \Sing(X,P) $.
		See \HAa{Theorem}{A.9.3}.
	\end{enumerate}
\end{example}

Jansen has also given incredible computations of exit-path \categories of some important compactifications naturally arising in geometry:

\begin{example}[{(the work of Jansen \cites{arXiv:2308.09551}{MR4651892})}]
	\hfill
	\begin{enumerate}
		\item Let $ G $ be a connected reductive linear algebraic group defined over $ \QQ $ whose center is anisotropic over $ \QQ $. 
		Let $ \Gamma \subset G(\QQ) $ be a neat arithmetic subgroup.
		Write $ X $ for the symmetric space of maximal compact subgroups of $ G(\RR) $ with $ \Gamma $-action given by conjugation.
		Jansen showed that the \topos of sheaves on the reductive Borel--Serre compactification $ \Gamma\backslash \Xbar{}^{\RBS} $ is exodromic and gave an explicit description of its exit-path \category.
		See \cite[Theorem 4.3]{MR4651892}.

		\item Let $ g,n \geq 0 $ be such that $ 2g - 2 + n > 0 $.
		Write \smash{$ \Mgnbar $} for the moduli stack of stable genus $ g $ nodal curves with $ n $ marked points (also called the \textit{Deligne--Mumford--Knudsen compactification}).
		Write \smash{$ \Mgnbartop $} for its underlying topological stack.
		The topological stack \smash{$ \Mgnbartop $} has a natural stratification by the poset of stable genus $ g $ dual graphs with $ n $ marked points.
		Jansen showed that the \topos of sheaves on the topological stack \smash{$ \Mgnbartop $} is exodromic.
		Moreover, the exit-path \category is equivalent to the opposite of the Charney--Lee category of stable genus $ g $ curves with $ n $ marked points \cites{MR772131}{MR4248723}{MR2452916}.
		See \cite[Corollary 6.6 \& Theorem 6.7]{arXiv:2308.09551}.
	\end{enumerate}
\end{example}	

Lurie and Tanaka have also provided a remarkable example related to Morse theory:

\begin{example}
	Let $ \Broken $ denote Lurie and Tanaka's moduli stack of \textit{broken lines} \cite{arXiv:1805.09587}. 
	This is a topological stack with a natural stratification by $ \NN^{\op} $.
	Write $ \Deltasurj \subset \DDelta $ for the non-full subcategory containing all objects, but only the surjective maps.
	One of Lurie and Tanaka's main results \cite[Theorem 1.0.5]{arXiv:1805.09587} is that there is an equivalence of \categories
	\begin{equation*}
		\Sh(\Broken) \equivalent \Fun(\Deltasurj,\Spc)
	\end{equation*}
	between sheaves on $ \Broken $ and functors out of $ \Deltasurj $.
	Said differently, every sheaf on $ \Broken $ is constructible, and the stratified \topos $ (\Sh(\Broken),\NN^{\op}) $ is exodromic with exit-path \category $ \Deltasurj $.
\end{example}

Another feature of \Cref{def:exit_path} is that the inclusion of constructible objects admits both a left and right adjoint:

\begin{notation}[(constructibilization)]\label{notation:constructibilization}
	Let $ (\X,P) $ be an exodromic stratified \topos.
	Since $ \ConsP(\X) \subset \X $ is closed under limits and colimits, \cite[Theorem 1.1]{RagimovSchlank} implies that $ \ConsP(\X) $ is presentable and the inclusion
	\begin{equation*}
		i_{\X,P} \colon \incto{\ConsP(\X)}{\X}
	\end{equation*}
	has both a left adjoint $\Lup_{\X,P}$ and a right adjoint $\Rup_{\X,P}$.
	We refer to these adjoints as the \defn{left} and \defn{right constructibilization functors}, respectively.
	In particular, $\ConsP(\X)$ is a localization of $\X$, and it coincides with the full subcategory of $\X$ spanned by $\Lup_{\X,P}$-equivalences.
\end{notation}

\begin{example}[(equational criterion for constructibility)]
	Let $ (X,P) $ be a conically stratified topological space with locally weakly contractible strata.
	Then \cite[Corollary 5.4.7]{arXiv:2211.05004} provides an explicit set of generating $\Lup_{X,P}$-equivalences in terms of conical charts.
	When $P = \ast$, we can take as a generating set all the inclusions $U \subset V$ between weakly contractible open subsets.
\end{example}

If $ (\X,P) $ is exodromic, then the subcategory $ \ConsP(\X) $ is automatically \atopos:

\begin{lemma}\label{lem:left_exact_colocalization_of_a_topos_is_a_topos}
	Let $ \X $ be \atopos and let $ \Y \subset \X $ be a full subcategory.
	If $ \Y $ is presentable and closed under colimits and finite limits in $ \X $, then $ \Y $ is also \atopos.
\end{lemma}

\begin{proof}
	It suffices to show that $ \Y $ satisfies the Giraud--Lurie axioms for \topoi \HTT{Proposition}{6.1.0.6}.
	By assumption $ \Y $ is presentable; the remaining axioms only involve statements about the interaction between colimits and finite limits.
	Since $ \X $ is \atopos, these statements are true in $ \X $.
	By assumption, $ \Y $ is closed under colimits and finite limits in $ \X $, hence these statements also hold in $ \Y $.
\end{proof}

\begin{example}\label{ex:constructible_objects_in_an_exodromic_topos_form_a_topos}
	Let $ (\Xcal,P) $ be an exodromic stratified \topos.
	Combining the observations made in \Cref{notation:constructibilization} with \Cref{lem:left_exact_colocalization_of_a_topos_is_a_topos}, we see that the inclusion $ \ConsP(\X) \subset \X $ admits both a left and a right adjoint and $ \ConsP(\X) $ is \atopos.
\end{example}

A very important fact is that exodromic stratified \topoi are automatically monodromic:

\begin{lemma}[(exodromy implies monodromy)]\label{lem:exodromy_implies_monodromy}
	Let $ \slowerstar \colon \fromto{\X}{\Fun(P,\Spc)} $ be an exodromic stratified \topos.
	Then:
	\begin{enumerate}
		\item\label{lem:exodromy_implies_monodromy.1} The \topos $ \X $ is monodromic.

		\item\label{lem:exodromy_implies_monodromy.2} The full subcategory $ \LC(\X) \subset \ConsP(\X) $ is closed under limits and colimits.
	\end{enumerate} 
\end{lemma}

\begin{proof}
	First we prove (1).
	In light of \Cref{rec:monodromy}, we need to show that the constant sheaf functor $ \Gammaupperstar \colon \fromto{\Spc}{\X} $ preserves limits.
	Note that $ \Gammaupperstar $ factors as a composite
	\begin{equation*}
		\begin{tikzcd}
			\Spc \arrow[r] & \Fun(P,\Spc) \arrow[r, "\supperstar"] & \ConsP(\X) \arrow[r, hooked] & \X \comma
		\end{tikzcd}
	\end{equation*}
	where the left-most functor is the constant functor.
	The constant functor $ \fromto{\Spc}{\Fun(P,\Spc)} $ preserves limits, and by assumption both $ \supperstar $ and the inclusion $ \ConsP(\X) \subset \X $ preserve limits.
	Hence $ \Gammaupperstar $ preserves limits, as desired.

	For (2), note that both $ \LC(\X) $ and $ \ConsP(\X) $ are closed under limits and colimits in $ \X $.
\end{proof}


\subsection{Exodromic morphisms}\label{subsec:exodromic_morphisms}

We now discuss the functoriality of exit-path \categories.
The main point of this subsection is that given a morphism $ \flowerstar \colon \fromto{(\X,P)}{(\Y,Q)} $ between exodromic stratified \topoi, it is not \textit{a priori} clear if $ \flowerstar $ induces a functor
\begin{equation*}
	\fromto{\Piinfty(\X,P)}{\Piinfty(\Y,Q)}
\end{equation*}
on exit-path \categories.

\begin{observation}[{(constructible $ \ast $-pushforward)}]\label{obs:constructible_pushforward}
	Let $ \flowerstar \colon \fromto{(\X,P)}{(\Y,Q)} $ be a morphism between exodromic stratified \topoi.
	Since the functor \smash{$\fupperstar \colon \fromto{\Y}{\X}$} preserves colimits, we deduce that
	\begin{equation*} 
		\fupperstar \colon \fromto{\ConsQ(\Y)}{\ConsP(\X)} 
	\end{equation*}
	preserves colimits as well.
	In particular, it admits a right adjoint
	\begin{equation*} 
		\fconslowerstar \colon \fromto{\ConsP(\X)}{\ConsQ(\Y)} \period 
	\end{equation*}
	Unraveling the definitions, we see that $\fconslowerstar$ is related to the pushforward functor $\flowerstar$ by the formula
	\begin{equation*}
		\fconslowerstar = \Rup_{\Y,Q} \circ \flowerstar \circ i_{\X,P} \comma 
	\end{equation*}
	where $\Rup_{\Y,Q}$ is the right constructibilization functor of \cref{notation:constructibilization}.
	In particular, if $\flowerstar$ takes $ P $-constructible objects to $Q$-constructible objects, then $\fconslowerstar \equivalent \flowerstar $.
\end{observation}

The following is a generalization of \cite[Definition 3.5-(3)]{arXiv:2108.01924}:

\begin{definition}\label{def:exodromic_morphism}
	Let $ \flowerstar \colon \fromto{(\X,P)}{(\Y,Q)} $ be a morphism between exodromic stratified \topoi.
	We say that \defn{$ \flowerstar $ is exodromic} if the left adjoint
	\begin{equation*}
		\fupperstar \colon \fromto{\ConsQ(\Y)}{\ConsP(\X)}
	\end{equation*}
	also preserves limits.
	In this case, we denote its left adjoint by
	\begin{equation*}
		\flowersharpcons \colon \fromto{\ConsP(\X)}{\ConsQ(\Y)} \period
	\end{equation*}
	As a consequence of the equivalence \smash{$ \Catidem \equivalent \PrLat $} of \Cref{recollection:atomically_generated}, the functor $ \flowersharpcons $ restricts to a functor
	\begin{equation*} 
		f^{\ex} \colon \Piinfty(\X,P) \to \Piinfty(\Y,Q) \period 
	\end{equation*}
\end{definition}

The following are two important examples of exodromic morphisms:

\begin{example}\label{ex:morphisms_between_posets_are_exodromic}
	Let $ \phi \colon \fromto{P}{Q} $ be a map of posets.
	Equip both $ P $ and $ Q $ with the identity stratifications.
	Then \cref{rec:hypersheaves_on_a_poset} shows that the morphism of stratified \topoi
	\begin{equation*}
		\philowerstar \colon \fromto{(\Fun(P,\Spc),P)}{(\Fun(Q,\Spc),Q)}
	\end{equation*}
	is exodromic.
\end{example}

\begin{example}\label{ex:morphisms_between_trivially_stratified_topoi_are_exodromic}
	Let $ \flowerstar \colon \fromto{\X}{\Y} $ be a geometric morphism of \topoi.
	If $ \X $ and $ \Y $ are monodromic, then \Cref{cor:all_morphisms_are_monodromic} shows that the morphism of trivially stratified \topoi
	\begin{equation*}
		\flowerstar \colon \fromto{(\X,\pt)}{(\Y,\pt)}
	\end{equation*}
	is exodromic.
\end{example}

In fact, we will see that \Cref{def:exodromic_morphism} is superfluous: one of the goals of \cref{sec:stability_properties_of_exodromic_stratified_topoi} is to show that \textit{every} morphism between exodromic stratified \topoi is exodromic. 
However, this is not obvious; see \Cref{thm:all_morphisms_are_exodromic} for details.

We conclude this subsection with a few useful observations about exodromic morphisms.

\begin{observation}\label{warning:constructible_exceptional_pushforward}
	Let $ \flowerstar \colon \fromto{(\X,P)}{(\Y,Q)} $ be a morphism of stratified \topoi.
	Assume the following:
	\begin{enumerate}
		\item\label{warning:constructible_exceptional_pushforward.1} $ (\X,P) $ and $(\Y,Q)$ are exodromic.

		\item\label{warning:constructible_exceptional_pushforward.2} $ \fupperstar \colon \fromto{\Y}{\X} $ admits a left adjoint $ \flowersharp \colon \fromto{\X}{\Y} $.
	\end{enumerate}
	Then $ \flowerstar $ is exodromic.
	Moreover, the functors
	\begin{equation*} 
		\flowersharpcons \colon \fromto{\ConsP(\X)}{\ConsQ(\Y)} \andeq \flowersharp \colon \fromto{\X}{\Y}
	\end{equation*} 
	are related by the formula
	\begin{equation*}  
		\flowersharpcons \equivalent \Lup_{\Y,Q} \circ \flowersharp \circ i_{\X,P} \comma 
	\end{equation*} 
	where $ \Lup_{\Y,Q} $ is the left constructibilization functor of \cref{notation:constructibilization}.
	In particular, if \smash{$ \flowersharp $} carries $ P $-constructible objects to $ Q $-constructible objects, then there is a canonical identification \smash{$ \flowersharpcons \equivalent \flowersharp $}.
\end{observation}

\begin{observation}[(pullback functoriality)]\label{obs:meaning_of_respecting_exit-paths}
	Let $ \flowerstar \colon \fromto{(\X,P)}{(\Y,Q)} $ be a morphism between exodromic stratified \topoi.
	If $ \flowerstar $ is exodromic, then \cref{obs:atomic_functors_preserve_atomic_objects} yields a commutative square
	\begin{equation*} 
		\begin{tikzcd}[column sep=4em]
			\Fun(\Piinfty(\Y,Q),\Spc) \arrow[r, "- \of f^{\ex}"] \arrow[d, "\wr"'{xshift=0.25ex}] & \Fun(\Piinfty(\X,P),\Spc) \arrow[d, "\wr"{xshift=-0.25ex}] \\
			\ConsQ(\Y) \arrow[r, "\fupperstar"'] & \ConsP(\X) \comma
		\end{tikzcd} 
	\end{equation*}
	where the vertical equivalences exhibit the exit-path \categories $ \Piinfty(\Y,Q) $ and $ \Piinfty(\X,P) $ as the opposites of the subcategories of atomic objects of the targets.
\end{observation}

\begin{observation}[{($ \sharp $-pushforward functoriality)}]\label{obs:morphisms_that_respect_exit-paths_commute_with_the_Yoneda_embedding}
	As a consequence of \Cref{obs:meaning_of_respecting_exit-paths}, there is also a commutative square
	\begin{equation*}
		\begin{tikzcd}[column sep=4em]
			\Fun(\Piinfty(\X,P),\Spc) \arrow[r, "\flowershriek^{\ex}"] \arrow[d, "\wr"'{xshift=0.25ex}] & \Fun(\Piinfty(\Y,Q),\Spc) \arrow[d, "\wr"{xshift=-0.25ex}] \\
			\ConsP(\X) \arrow[r, "\flowersharpcons"'] & \ConsQ(\Y) \comma
		\end{tikzcd} 
	\end{equation*}
	where $ \flowershriek^{\ex} $ denotes left Kan extension along $ f^{\ex} $.
	Since left Kan extension commutes with the Yoneda embedding, we also deduce that there is a commutative square
	\begin{equation*}
		\begin{tikzcd}[column sep=4em]
			\Piinfty(\X,P)^{\op} \arrow[r, "f^{\ex, \op}"] \arrow[d, hooked] & \Piinfty(\Y,Q)^{\op} \arrow[d, hooked] \\
			\ConsP(\X) \arrow[r, "\flowersharpcons"'] & \ConsQ(\Y) \comma
		\end{tikzcd} 
	\end{equation*}
	where the vertical functors are the inclusions of the subcategories of atomic objects.
\end{observation}


\subsection{Exodromy \& hypercompletion}\label{subsec:exodromy_and_hypercompletion}

Let $ (\X,P) $ be an exodromic stratified \topos.
The goal of this subsection is to show that the hypercompletion $ \Xhyp $ with the induced stratification is also exodromic, the \categories $ \ConsP(\X) $ and $ \ConsP(\Xhyp) $ coincide, and the exit-path \categories $ \Piinfty(\X,P) $ and $ \Piinfty(\Xhyp,P) $ coincide.
We do not accomplish this in complete generality, however, we prove that this the case under an additional assumption on $ (\X,P) $; see \Cref{defin:weakly_conical,prop:for_X_weakly_conical_X_exodromic_implies_Xhyp_exodromic}.
This assumption is satisfied, for example, when $ P $ is noetherian and $ \X $ is the \topos of sheaves associated to a conically stratified space for which exodromy is already known.

\begin{notation}
	Let $ \slowerstar \colon \fromto{\X}{\Fun(P,\Spc)} $ be a stratified \topos.
	Then the composite
	\begin{equation*}
		\begin{tikzcd}
			\Xhyp \arrow[r, hooked] & \X \arrow[r, "\slowerstar"] & \Fun(P,\Spc)
		\end{tikzcd}
	\end{equation*}
	defines a $ P $-stratification of $ \Xhyp $.
	We always regard the hypercompletion of a stratified \topos with this induced stratification.
	Also note that since the \topos $ \Fun(P,\Spc) $ is hypercomplete, the stratification
	\begin{equation*}
		\fromto{\Xhyp}{\Fun(P,\Spc)}
	\end{equation*}
	coincides with the geometric morphism $ \slowerstarhyp $ obtained by applying the hypercompletion functor to the stratification $ \slowerstar $.
\end{notation}

We start by showing that if $ (\X,P) $ is exodromic, then every $ P $-constructible object of $ \X $ is hypercomplete.
For this, we need a few lemmas.

\begin{lemma}\label{lem:hypercompletion_commutes_with_taking_strata}
	Let $ (\X,P) $ be a stratified \topos and $ S \subset P $ a locally closed subposet.
	Then the natural geometric morphism
	\begin{equation*}
		\fromto{(\XS)^{\hyp}}{(\Xhyp)_S}
	\end{equation*}	
	is an equivalence of $ S $-stratified \topoi.
\end{lemma}

\begin{proof}
	This is a special case of \Cref{prop:hypercompletion_commutes_with_pullback_along_locally_closed_immersions}.
\end{proof}

\begin{lemma}\label{lem:functors_that_are_both_left_and_right_adjoints_preserve_hypercompleteness_and_convergence_of_Postnikov_towers}
	Let $ \X $ and $ \Y $ be \topoi and let $ \fupperstar \colon \fromto{\Y}{\X} $ be a functor that preserves both limits and colimits (i.e., $ \fupperstar $ is the pullback functor in an essential geometric morphism).
	Let $ G \in \Y $.
	\begin{enumerate}
		\item\label{lem:functors_that_are_both_left_and_right_adjoints_preserve_hypercompleteness_and_convergence_of_Postnikov_towers.1} If $ G $ is hypercomplete, then $ \fupperstar(G) $ is hypercomplete.

		\item\label{lem:functors_that_are_both_left_and_right_adjoints_preserve_hypercompleteness_and_convergence_of_Postnikov_towers.2} If $ G $ is the limit of its Postnikov tower, then $ \fupperstar(G) $ is the limit of its Postnikov tower.
	\end{enumerate}
\end{lemma}

\begin{proof}
	Item (1) is the content of \HAa{Lemma}{A.2.6}.
	For (2), note that since $ \fupperstar $ is a left exact left adjoint, \HTT{Proposition}{5.5.6.28} shows that for each $ n \geq 0 $, we have
	\begin{equation*}
		\fupperstar \trun_{\leq n}^{\Y} \equivalent \trun_{\leq n}^{\X} \fupperstar \period
	\end{equation*}
	Since $ G $ is the limit of its Postnikov tower and $ \fupperstar $ preserves limits, we see that
	\begin{align*}
		\fupperstar(G) &\equivalence \fupperstar\paren{ \lim_{n \in \NN^{\op}}  \trun_{\leq n}^{\Y}(G)} \\ 
		&\equivalence \lim_{n \in \NN^{\op}} \fupperstar\trun_{\leq n}^{\Y}(G) \\
		&\equivalent \lim_{n \in \NN^{\op}} \trun_{\leq n}^{\X} \fupperstar(G) \period \qedhere
	\end{align*}
\end{proof}

\begin{corollary}\label{cor:exodromy_implies_that_constructible_sheaves_are_hypercomplete}
	Let $ (\X,P) $ be an exodromic stratified \topos.
	\begin{enumerate}
		\item\label{cor:exodromy_implies_that_constructible_sheaves_are_hypercomplete.1} If $ F \in \ConsP(\X) $, then $ F $ is the limit of its Postnikov tower in $ \X $.
		In particular, we have
		\begin{equation*}
			\ConsP(\X) \subset \Xhyp \period
		\end{equation*}

		\item\label{cor:exodromy_implies_that_constructible_sheaves_are_hypercomplete.2} We have $ \ConsP(\X) \subset \ConsP(\Xhyp) $ as full subcategories of $ \Xhyp $.

		\item\label{cor:exodromy_implies_that_constructible_sheaves_are_hypercomplete.3} The functor $ \supperstar \colon \fromto{\Fun(P,\Spc)}{\X} $ factors through $ \Xhyp $.

		\item\label{cor:exodromy_implies_that_constructible_sheaves_are_hypercomplete.4} The constant sheaf functor $ \Gammaupperstar \colon \fromto{\Spc}{\X} $ factors through $ \Xhyp \subset \X $.
	\end{enumerate}
\end{corollary}

\begin{proof}
	Recall from \Cref{ex:constructible_objects_in_an_exodromic_topos_form_a_topos} that since $ (\X,P) $ is exodromic, the \category $ \ConsP(\X) $ is \atopos and the inclusion $ \ConsP(\X) \subset \X $ preserves limits and colimits.
	Hence (1) is a special case of \enumref{lem:functors_that_are_both_left_and_right_adjoints_preserve_hypercompleteness_and_convergence_of_Postnikov_towers}{2}.
	For (2), note that the inclusion $ \incto{(\Xhyp,P)}{(\X,P)} $ is a morphism of stratified \topoi.
	Hence the hypercompletion functor $ \fromto{\X}{\Xhyp} $ carries $ \ConsP(\X) $ to $ \ConsP(\Xhyp) $.
	By (1), every object of $ \ConsP(\X) $ is already hypercomplete, hence
	\begin{equation*}
		 \ConsP(\X) \subset \ConsP(\Xhyp) 
	\end{equation*}
	as full subcategories of $ \Xhyp $.

	Item (3) is an immediate consequence of item (1) and the fact that $ \supperstar $ factors through $ \ConsP(\X) $.
	Item (4) is immediate from (3) and the fact that $ \Gammaupperstar $ factors as the composite
	\begin{equation*}
		\begin{tikzcd}
			\Spc \arrow[r] & \Fun(P,\Spc) \arrow[r, "\supperstar"] & \ConsP(\X) \arrow[r, hooked] & \X \comma
		\end{tikzcd}
	\end{equation*}
	where the left-most functor is the constant functor.
\end{proof}

\begin{observation}\label{observation:hyper_constructibilization}
	Let $ (\X,P) $ be an exodromic stratified \topos.
	\Cref{cor:exodromy_implies_that_constructible_sheaves_are_hypercomplete} implies that the left constructibilization functor $\Lup_{\X,P} \colon \X \to \ConsP(\X)$ factors as a the composite of hypercompletion $ \fromto{\X}{\Xhyp} $ with a localization 
	\begin{equation*} 
		\Lup^{\hyp}_{\X,P} \colon \Xhyp \to \ConsP(\X) \period 
	\end{equation*}
	In turn, $\ConsP(\X)$ can be identified with the full subcategory of $ \Xhyp $ spanned by objects that are local with respect to \smash{$\Lup^{\hyp}_{\X,P}$}-equivalences.
	Hence a morphism $\phi \colon F \to G$ in $\X$ is an $\Lup_{\X,P}$-equivalence if and only if its hypercompletion \smash{$\phi^{\hyp}$} is an \smash{$\Lup^{\hyp}_{\X,P}$}-equivalence.
\end{observation}

Our next goal is to show that if $ \X $ is monodromic, then $ \Xhyp $ is also monodromic and $ \LC(\X) = \LC(\Xhyp) $.
For this, we need the following lemma.

\begin{lemma}\label{lem:inclusion_of_Xhyp_preserves_effective_epimorphisms_and_coproducts}
	Let $ \X $ be \atopos and write $ \ilowerstar \colon \incto{\Xhyp}{\X} $ for the inclusion.
	\begin{enumerate}
		\item\label{lem:inclusion_of_Xhyp_preserves_effective_epimorphisms_and_coproducts.1} A morphism $ f \colon \fromto{U}{V} $ in $ \Xhyp $ is an effective epimorphism if and only if $ \ilowerstar(f) $ is an effective epimorphism in $ \X $.

		\item\label{lem:inclusion_of_Xhyp_preserves_effective_epimorphisms_and_coproducts.2} The functor $ \ilowerstar \colon \incto{\Xhyp}{\X} $ preserves coproducts.

		\item\label{lem:inclusion_of_Xhyp_preserves_effective_epimorphisms_and_coproducts.3} Given an effective epimorphism $ \fromto{\coprod_{\alpha \in A} U_{\alpha}}{1_{\Xhyp}} $ in $ \Xhyp $, the induced map
		\begin{equation*}
			\fromto{\coprod_{\alpha \in A} \ilowerstar(U_{\alpha})}{1_{\X}} 
		\end{equation*}
		is an effective epimorphism in $ \X $.
	\end{enumerate}
\end{lemma}

\begin{proof}
	For (1), first assume that $ \ilowerstar(f) $ is an effective epimorphism.
	Then since $ \iupperstar $ preserves effective epimorphisms and $ \ilowerstar $ is fully faithful, $ f \equivalent \iupperstar\ilowerstar(f) $ is also an effective epimorphism.
	Conversely, assume that $ f $ is an effective epimorphism.
	Note that 
	\begin{equation*}
		\trun_{\leq 0}^{\X} \ilowerstar(f) = \trun_{\leq 0}^{\Xhyp}(f) \period
	\end{equation*}
	Since the property of a morphism being an effective epimorphism only depends on the $ 0 $-truncation \HTT{Proposition}{7.2.1.14} and $ f $ is an effective epimorphism, we deduce that $ \ilowerstar(f) $ is an effective epimorphism.

	Item (2) is the content of \SAG{Lemma}{D.6.7.2}.
	Finally, (3) is immediate from (1) and (2).
\end{proof}

\begin{lemma}\label{lem:coincidences_when_every_constant_object_is_hypercomplete}
	Let $ \X $ be \atopos.
	If every constant object of $ \X $ is hypercomplete, then:
	\begin{enumerate}
		\item For each $ U \in \X $, every constant object of $ \X_{/U} $ is hypercomplete.

		\item Every locally constant object of $ \X $ is hypercomplete.

		\item The inclusion $ \incto{\Xhyp}{\X} $ carries $ \LC(\Xhyp) $ to $ \LC(\X) $.

		\item We have $ \LC(\Xhyp) = \LC(\X) $ as full subcategories of $ \X $.
	\end{enumerate}
\end{lemma}

\begin{proof}
	For (1), write $ \pupperstar \colon \fromto{\X}{\X_{/U}} $ for the pullback functor.
	Observe that the constant sheaf functor $ \fromto{\Spc}{\X_{/U}} $ factors as a composite
	\begin{equation*}
		\begin{tikzcd}
			\Spc \arrow[r, "\Gammaupperstar"] & \X \arrow[r, "\pupperstar"] & \X_{/U} \period
		\end{tikzcd}
	\end{equation*}
	Since the pullback functor $ \pupperstar $ is both a left and a right adjoint, \Cref{lem:functors_that_are_both_left_and_right_adjoints_preserve_hypercompleteness_and_convergence_of_Postnikov_towers} shows that $ \pupperstar $ preserves hypercompleteness.
	Hence the claim follows from the assumption that every constant object of $ \X $ is hypercomplete.

	For (2), let $ L \in \LC(\X) $ and choose an effective epimorphism $ \surjto{\coprod_{\alpha \in A} U_{\alpha}}{1_{\X}} $ such that for each $ \alpha \in A $, the pullback $ L \cross U_{\alpha} $ is a constant object of $ \X_{/U_{\alpha}} $.
	Then by (1), for each $ \alpha \in A $, the object $ L \cross U_{\alpha} \in \X_{/U_{\alpha}} $ is hypercomplete.
	The claim now follows from the fact that hypercompleteness is a local property \HTT{Remark}{6.5.2.22}.

	For (3), let $ L \in \LC(\Xhyp) $; we wish to show that $ L \in \LC(\X) $.
	Choose an effective epimorphism
	\begin{equation*}
		\phi \colon \surjto{\coprod_{\alpha \in A} U_{\alpha}}{1_{\Xhyp} = 1_{\X}} 
	\end{equation*}
	in $ \Xhyp $ such that for each $ \alpha \in A $, the pullback
	\begin{equation*}
		L \cross U_{\alpha} \in (\Xhyp)_{/U_{\alpha}} = (\X_{/U_{\alpha}})^{\hyp}
	\end{equation*}
	is constant.
	By \enumref{lem:inclusion_of_Xhyp_preserves_effective_epimorphisms_and_coproducts}{3} the effective epimorphism $ \phi \colon \surjto{\coprod_{\alpha \in A} U_{\alpha}}{1_{\X}} $ in $ \Xhyp $ is also an effective epimorphism in the larger \topos $ \X $.
	Hence it suffices to show that each $ L \cross U_{\alpha} $ is also a constant object of the larger \topos $ \X_{/U_{\alpha}} $.
	For this, note that by  (1), every constant object of $ \X_{/U_{\alpha}} $ is hypercomplete.

	Item (4) is immediate from items (2) and (3).
\end{proof}

\begin{proposition}\label{prop:X_monodromic_implies_Xhyp_monodromic}
	Let $ \X $ be a monodromic \topos.
	Then:
	\begin{enumerate}
		\item \label{prop:X_monodromic_implies_Xhyp_monodromic.1} The composite
		\begin{equation*}
			\begin{tikzcd}
				\Xhyp \arrow[r, hooked, "\ilowerstar"] & \X \arrow[r, "\Gammalowersharp"] & \Spc
			\end{tikzcd}
		\end{equation*}
		is left adjoint to the constant hypersheaf functor $ \fromto{\Spc}{\Xhyp} $.
		In particular, $ \Xhyp $ is monodromic.

		\item \label{prop:X_monodromic_implies_Xhyp_monodromic.2} The inclusion $ \incto{\Xhyp}{\X} $ carries $ \LC(\Xhyp) $ to $ \LC(\X) $.
		Moreover, we have $ \LC(\Xhyp) = \LC(\X) $ as full subcategories of $ \X $.

		\item \label{prop:X_monodromic_implies_Xhyp_monodromic.3} The natural map $ \fromto{\Piinfty(\Xhyp)}{\Piinfty(\X)} $ is an equivalence.
	\end{enumerate}
\end{proposition}

\begin{proof}
	For (1), note that since $ \Gammaupperstar \colon \fromto{\Spc}{\X} $ factors through $ \Xhyp $, for $ F \in \Xhyp $ and $ K \in \Spc $, we have natural equivalences
	\begin{align*}
		\Map_{\Spc}(\Gammalowersharp \ilowerstar(F),K) &\equivalent \Map_{\X}(\ilowerstar(F),\Gammaupperstar(K)) \\ 
		&\equivalent \Map_{\Xhyp}(F,\Gammaupperstar(K)) \period
	\end{align*}
	Item (2) is a special case of \Cref{lem:coincidences_when_every_constant_object_is_hypercomplete}.
	Finally, by (2), the pullback functor $ \fromto{\LC(\X)}{\LC(\Xhyp)} $ is an equivalence (in fact, the identity).
	Hence (3) follows from the definition of the shape. 
\end{proof}

\begin{warning}\label{warn:an_object_whose_hypercompletion_is_locally_constant_need_not_be_locally_constant}
	Let $ \X $ be a monodromic \topos and $ F \in \X $.
	If the hypercompletion of $ F $ is a locally constant object of $ \Xhyp $, then it is not necessarily the case that $ F $ is a locally constant object of $ \X $.
	
	For example, let $ Q = \prod_{\NN} [0,1] $ be the Hilbert cube.
	Then $ Q $ is locally of singular shape; in fact, any locally compact absolute neighborhood retract is locally of singular shape \cite[Example 2.8]{arXiv:2510.24629}.
	In particular, the \topos $ \Xcal = \Sh(Q) $ is monodromic.
	We claim that there exists an object $ F $ of $ \Sh(Q) $ that is not locally constant, but $ F^{\hyp} $ is the terminal object.
	By \Cref{prop:X_monodromic_implies_Xhyp_monodromic}, every locally constant object of $ \Sh(Q) $ is hypercomplete, so it suffices to provide a sheaf $ F $ that is not terminal, but whose stalks are all terminal.
	 
	Consider the dualizing sheaf $ \upomega_Q $ with $ \ZZ $-coefficients.
	That is, $ \upomega_Q $ is the sheaf of Borel--Moore chains with values in the \category $ \Mod_{\ZZ} $ of $ \ZZ $-module spectra.
	As explained in \HTT{Counterexample}{6.5.4.8}, $ \upomega_Q $ vanishes on a basis of $ Q $, hence has trivial stalks.
	On the other hand, since $ Q $ is compact, the Borel--Moore homology of $ Q $ agrees with ordinary homology, so $ \upomega_Q(Q) \equivalent \ZZ $; in particular, $ \upomega_Q $ is a nontrivial sheaf valued in $ \Mod_{\ZZ} $ whose hypercompletion is $ 0 $.
	Now consider the underlying sheaf of spaces $ F = \Omega^{\infty} \upomega_Q $.
	Since $ \Omega^{\infty} $ commutes with pulling back to open subspaces, we deduce that the sheaf $ F $ vanishes on a basis of $ Q $, hence has terminal stalks.
	However, $ F(Q) = \Omega^{\infty} \ZZ $ is the set of integers, so $ F $ is not the terminal object.
\end{warning}

Let $ (\X,P) $ be an exodromic stratified \topos.
\enumref{cor:exodromy_implies_that_constructible_sheaves_are_hypercomplete}{2} shows that $ \ConsP(\X) \subset \ConsP(\Xhyp) $.
In the case of a trivial stratification, we have just seen that this inclusion is an equality.
For a general stratification, we do not know if this holds; we offer the following simple sufficient condition for this to hold.
This condition covers many concrete cases of interest.

\begin{definition}\label{defin:weakly_conical}
	Let $ (\X,P) $ be a stratified \topos.
	We say that $ (\X,P) $ is \emph{weakly conical} if for every locally closed subset $S \subset P$, the functor
	\begin{equation*}
		i_{S,\ast} \colon \X_S \to \X 
	\end{equation*}
	takes $\ConsS(\X)$ to $\ConsP(\X)$.
\end{definition}

This definition is motivated by the following:

\begin{example}\label{eg:weakly_conical}
	Let $ (X,P) $ be a conically stratified space with locally weakly contractible strata.
	Then \smash{$(\Shhyp(X),P)$} is weakly conical by \cite[Proposition 6.8.1]{arXiv:2211.05004}; this ultimately relies on \cite[Lemma 5.3.4]{arXiv:2211.05004}, which is the hard step needed to prove the exodromy equivalence in the conical setting.
	On the other hand, consider the non-conical stratification of a circle pictured on the right-hand side of \Cref{fig:circle}: in this case, the pushforward of a constant sheaf on the open stratum is not hyperconstructible with respect to the given stratification.
	Thus, this property is a special feature of the conical situation.
\end{example}

\begin{lemma}\label{lem:weakly_conical_equational_criterion}
	Let $ (\X,P) $ be a weakly conical exodromic stratified \topos.
	Let $ \phi \colon F_1 \to F_2$ be a $\Lup_{\X,P}$-equivalence (see \cref{notation:constructibilization}).
	Then for every locally closed subset $S \subset P$, the morphism $i_S^\ast(\phi)$ is an $\Lup_{\X_S,S}$-equivalence.
\end{lemma}

\begin{proof}
	We have to show that for all $G \in \ConsS(\X_S)$, the map $i_S^\ast(\phi)$ induces an equivalence
	\begin{equation*}
		\Map_{\X_S}( i_S^\ast(F_2), G ) \to \Map_{\X_S}( i_S^\ast(F_1), G ) \period
	\end{equation*}
	By adjunction, this follows immediately from the fact that $ \phi $ is a $P$-equivalence and that $i_{S,\ast}(G) \in \ConsP(\X)$.
\end{proof}

\begin{lemma}\label{lem:axiomatic_X_exodromic_implies_Xhyp_exodromic}
	Let $ (\X,P) $ be an exodromic stratified \topos.
	If the inclusion $ \incto{\Xhyp}{\X} $ carries $ \ConsP(\Xhyp) $ to $ \ConsP(\X) $, then:
	\begin{enumerate}
		\item We have $ \ConsP(\Xhyp) = \ConsP(\X) $ as full subcategories of $ \X $.

		\item The stratified \topos $ (\Xhyp,P) $ is exodromic.
		
		\item The natural functor $ \fromto{\Piinfty(\Xhyp,P)}{\Piinfty(\X,P)} $ is an equivalence of \categories.
	\end{enumerate}
\end{lemma}

\begin{proof}
	Since $ (\X,P) $ is exodromic, \enumref{cor:exodromy_implies_that_constructible_sheaves_are_hypercomplete}{2} guarantees that
	\begin{equation*}
		\ConsP(\X) \subset \ConsP(\Xhyp) \period
	\end{equation*}
	Our assumption guarantees that this inclusion is an equality.

	For (2), note that by (1) and the assumption that $ (\X,P) $ is exodromic, the \category $ \ConsP(\Xhyp) $ is atomically generated.
	In light of \enumref{cor:exodromy_implies_that_constructible_sheaves_are_hypercomplete}{3}, all that remains to be shown is that the full subcategory
	\begin{equation*}
		\ConsP(\Xhyp) \subset \Xhyp
	\end{equation*}
	is closed under limits and colimits.
	Again by (1), we have $ \ConsP(\Xhyp) = \ConsP(\X) $.
	Moreover, since $ (\X,P) $ is exodromic, $ \ConsP(\X) \subset \X $ is closed under limits and colimits.
	The claim now follows from the fact that $ \Xhyp $ is a localization of $ \X $.
	
	Item (3) is immediate from items (1) and (2) and the definition of the exit-path \category of an exodromic stratified \topos.
\end{proof}

The following is the main result of this subsection.

\begin{proposition}\label{prop:for_X_weakly_conical_X_exodromic_implies_Xhyp_exodromic}
	Let $ (\X,P) $ be a stratified \topos.
	Assume that $P$ is noetherian and that $ (\X,P) $ is both exodromic and weakly conical.
	Then:
	\begin{enumerate}
		\item The inclusion $ \incto{\Xhyp}{\X} $ carries $ \ConsP(\Xhyp) $ to $ \ConsP(\X) $.
		
		\item The stratified \topos $ (\Xhyp,P) $ is exodromic.
		
		\item The natural functor $ \fromto{\Piinfty(\Xhyp,P)}{\Piinfty(\X,P)} $ is an equivalence of \categories.
	\end{enumerate}
\end{proposition}

\begin{proof}
	First note that by \Cref{lem:axiomatic_X_exodromic_implies_Xhyp_exodromic}, it suffices to prove (1).
	Since $ (\X,P) $ is exodromic, \enumref{cor:exodromy_implies_that_constructible_sheaves_are_hypercomplete}{2} guarantees that
	\begin{equation*}
		\ConsP(\X) \subset \ConsP(\Xhyp) \period
	\end{equation*}
	We prove the other inclusion by noetherian induction, observing that the case $ P = \ast $ has already been dealt with in \cref{prop:X_monodromic_implies_Xhyp_monodromic}-(1).
	Fix $F \in \ConsP(\Xhyp)$ and $p \in P$.
	Set $Q \colonequals P_{\geq p}$.
	Then $ Q $ is an open subset of $P$; in particular $i_{Q}^{\ast}$ preserves hypercomplete objects.
	Thus,
	\begin{equation*}
		i_Q^\ast(F) \equivalent i_Q^{\ast,\hyp}(F) \in \ConsP(\Xhyp_Q) \period 
	\end{equation*}
	In other words, we can assume without loss of generality that $ p $ is a minimal element of $P$.
	
	Now set $S \colonequals P_{>p}$.
	Again, $ S $ is an open subset of $P$.
	Moreover, \smash{$ \Xhyp $} is the recollement of \smash{$\Xhyp_p$} and \smash{$\Xhyp_S$}.
	In particular, for each $ F \in \ConsP(\Xhyp) $, there is a pullback square
	\begin{equation}\label{square:facture_for_Xp_and_complement}
		\begin{tikzcd}[sep=2em]
			F \arrow[r] \arrow[d] \arrow[dr, phantom, very near start, "\lrcorner", xshift=-0.25em, yshift=-0.25em] & i_{p, \ast} i_p^{\ast,\hyp}(F) \arrow[d] \\
			i_{S,\ast} i_S^{\ast}(F) \arrow[r] & i_{p,\ast} i_{p}^{\ast,\hyp} i_{S,\ast} i_S^{\ast}(F) \period
		\end{tikzcd} 
	\end{equation}
	Thanks to \cref{observation:hyper_constructibilization}, it is enough to prove that for every $\Lup^{\hyp}_{\X,P}$-equivalence $ \phi \colon G_1 \to G_2$ in $ \Xhyp $, the object $F$ is $ \phi $-local.
	By virtue of the pullback square \eqref{square:facture_for_Xp_and_complement}, it suffices to prove that the other three terms are $ \phi $-local.
	The inductive hypothesis guarantees that
	\begin{equation*}
		i_S^{\ast}(F) \equivalent i_S^{\ast,\hyp}(F)
	\end{equation*}
	belongs to $\ConsS(\X_S)$.
	Since $ (\X,P) $ is weakly conical, it follows that $i_{S,\ast} i_S^{\ast}(F) \in \ConsP(\X)$; in particular, $ i_{S,\ast} i_S^{\ast}(F) $ is $ \phi $-local.
	As for the other two terms, first recall from \cref{observation:hyper_constructibilization} that $ \phi $, seen as a morphism in $\X$, is an $\Lup_{\X,P}$-equivalence.
	In particular, \cref{lem:weakly_conical_equational_criterion} guarantees that $i_p^\ast(\phi)$ is an $\Lup_{\X_p}$-equivalence.
	Applying \cref{observation:hyper_constructibilization} once more, we deduce that
	\begin{equation*}
		i_p^{\ast,\hyp}(\phi) \equivalent (i_p^\ast(\phi))^{\hyp} 
	\end{equation*}
	is an $\Lup^{\hyp}_{\X_p}$-equivalence as well.
	Thus, it immediately follows from adjunction that $i_{p,\ast} i_p^{\ast,\hyp}(F)$ is $ \phi $-local.
	To conclude, observe that since $i_{S,\ast} i_S^{\ast}(F) \in \ConsP(\X)$, then
	\begin{equation*}
		i_p^\ast i_{S,\ast} i_S^\ast(F) \in \LC(\X_p) = \LC(\Xhyp_p) \period 
	\end{equation*}
	In particular we have
	\begin{equation*}
		i_p^{\ast,\hyp} i_{S,\ast} i_S^\ast(F) = i_p^\ast i_{S,\ast} i_S^\ast(F) \comma 
	\end{equation*}
	and the conclusion follows.
\end{proof}

We conclude with a question about generalizing \Cref{prop:for_X_weakly_conical_X_exodromic_implies_Xhyp_exodromic}.

\begin{question}
	Let $ (\X,P) $ be an exodromic stratified \topos.
	Does the inclusion $ \incto{\Xhyp}{\X} $ carry $ \ConsP(\Xhyp) $ to $ \ConsP(\X) $?
	(If so, then $ (\Xhyp,P) $ is exodromic and $ \equivto{\Piinfty(\Xhyp,P)}{\Piinfty(\X,P)} $.)
\end{question}


\section{Stability properties of exodromic stratified \texorpdfstring{$\infty$}{∞}-topoi}\label{sec:stability_properties_of_exodromic_stratified_topoi}

The goal of this section is to prove the following `stability theorem' for the class of exodromic stratified \topoi:

\begin{theorem}[(stability properties of exodromic stratified \topoi)]\label{thm:stability_properties_of_stratified_topoi}
	\noindent 
	\begin{enumerate}
		\item\label{thm:stability_properties_of_stratified_topoi.2} \emph{Stability under pulling back to locally closed subposets:} If $ (\X,P) $ is an exodromic stratified \topos, then for each locally closed subposet $ S \subset P $, the stratified \topos $ (\XS, S) $ is exodromic and the induced functor
		\begin{equation*}
			\fromto{\Piinfty(\XS,S)}{\Piinfty(\X,P) \cross_P S}
		\end{equation*}
		is an equivalence.
		In particular, the induced functor $ \fromto{\Piinfty(\X,P)}{P} $ is conservative.
		See \Cref{cor:stability_under_pulling_back_to_locally_closed_subposets}. 

		\item\label{thm:stability_properties_of_stratified_topoi.5} \emph{Every morphism between exodromic stratified \topoi is exodromic.}
		See \Cref{thm:all_morphisms_are_exodromic}.

		\item\label{thm:stability_properties_of_stratified_topoi.3} \emph{Stability under coarsening and localization formula:} 
		Let $ (\X,R) $ be an exodromic stratified \topos and let $ \phi \colon \fromto{R}{P} $ be a map of posets. 
		Write $ W_P $ for the collection of morphisms in $ \Piinfty(\X,R) $ that the composite $ \Piinfty(\X,R) \to R \to P $ sends to equivalences.
		Then the stratified \topos $ (\X,P) $ is exodromic and the natural functor $ \fromto{\Piinfty(\X,R)}{\Piinfty(\X,P)} $ induces an equivalence
		\begin{equation*}
			\equivto{\Piinfty(\X,R)[W_P\inv]}{\Piinfty(\X,P)} 
		\end{equation*}
		See \Cref{thm:stability_under_coarsening}.

		\item\label{thm:stability_properties_of_stratified_topoi.4} \emph{van Kampen:} Existence of exit-path \categories can be checked by descent.
		See \Cref{prop:van_Kampen} for a precise formulation.

		\item\label{thm:stability_properties_of_stratified_topoi.5} \emph{Künneth formula:} Let $ (\X,P) $ and $ (\Y,Q) $ be exodromic stratified \topoi.
		If $ P $ and $ Q $ are noetherian, then the stratified \topos $ (\X \tensor \Y, P \cross Q) $ is exodromic and there are natural equivalences of \categories
		\begin{align*}
			\ConsP(\X) \tensor \ConsQ(\Y) &\equivalence \Cons_{P \cross Q}(\X \tensor \Y) \\ 
			\shortintertext{and}
			\Piinfty(\X \tensor \Y, P \cross Q) &\equivalence \Piinfty(\X,P) \cross \Piinfty(\Y,Q) \period
		\end{align*}
		See \Cref{prop:Kunneth_formula_for_exodromic_topoi}.

		\item\label{thm:stability_properties_of_stratified_topoi.6} \emph{Stability of finiteness/compactness:} The property of an exit-path \category being finite (resp., compact) is stable under pulling back to a locally closed subposet, is stable under coarsening, and can be checked on a finite cover.
		See \cref{subsec:stability_properties_of_categorical_finiteness_and_compactness} for a precise formulation.
	\end{enumerate}
\end{theorem}

\Cref{subsec:stability_under_pulling_back_to_locally_closed_subposets} proves (1), \cref{subsec:all_morphisms_are_exodromic} proves (2), \cref{subsec:stability_under_coarsening} proves (3), \cref{subsec:checking_exodromy_locally} proves (4), \cref{subsec:Kunneth_formula} proves (5), and \cref{subsec:stability_properties_of_categorical_finiteness_and_compactness} proves (6).
Before moving on, we also pose two question related to \Cref{thm:stability_properties_of_stratified_topoi}.
First:

\begin{question}\label{quest:is_conservativity_necessary_for_Kunneth}
	Can one prove the Künneth formula without the extra noetherian hypothesis? 
\end{question}

\noindent Second, as noted earlier (see \Cref{obs:X_monodromic_implies_all_slices_are_monodromic}), if $ \X $ is a monodromic \topos and $ U \in \X $, then the slice \topos $ \X_{/U} $ is also monodromic.
We have not listed the analogous stability property for exodromic \topoi in \Cref{thm:stability_properties_of_stratified_topoi}; we do not know if it is true.
Thus we ask:

\begin{question}\label{quest:are_topoi_etale_over_an_exodromic_topos_exodromic}
	Let $ (\X,P) $ be a stratified \topos and $ U \in \X $.
	Then composing the natural geometric morphism $ \fromto{\X_{/U}}{\X} $ with the stratification of $ \X $ gives $ \X_{/U} $ a natural $ P $-stratification.
	Is the stratified \topos $ (\X_{/U},P) $ exodromic? 
\end{question}


\subsection{Stability under pulling back to locally closed subposets}\label{subsec:stability_under_pulling_back_to_locally_closed_subposets}

Let $ (\X,P) $ be an exodromic stratified \topos. 
The purpose of this subsection is to show that for each locally closed subposet $ S \subset P $, the stratified \topos $ (\XS,S) $ is exodromic, the inclusion $ \iSlowerstar \colon \incto{(\XS,S)}{(\X,P)} $ is exodromic, and the natural functor
\begin{equation*}
	\fromto{\Piinfty(\XS,S)}{\Piinfty(\X,P) \cross_P S}
\end{equation*}
is an equivalence (see \Cref{cor:stability_under_pulling_back_to_locally_closed_subposets}).
This result generalizes \cites[Proposition 3.6-(2)]{arXiv:2108.01924}[Proposition 3.13-(1)]{arXiv:2308.09550} to the setting of exodromic stratified \topoi; the proof is essentially the same as theirs, just adapted to our more general setting.
A key step is to show that both constructible objects and functors out of exit-path \categories satisfy \textit{recollement}.
We refer the reader to \cites[\HAsec{A.8}]{HA}[\SAGsec{7.2}]{SAG}[\S6.1]{arXiv:2109.12250}[\S2]{arXiv:2110.06567} for background on recollements.

We start by proving a general recollement result for \categories of functors out of \acategory with a functor to a poset.  

\begin{notation}\label{ntn:fibers_of_a_functor_to_a_poset}
	Let $ F \colon \fromto{\Ccal}{P} $ be a functor from \acategory to a poset.
	Given a full subposet $ S \subset P $, we write $ \Ccal_S \colonequals \Ccal \cross_P S $.
\end{notation}

\begin{observation}\label{obs:C_S_to_C_fully_faithful}
	In the setting of \Cref{ntn:fibers_of_a_functor_to_a_poset}, note that since the inclusion $ S \subset P $ is fully faithful, its basechange $ \fromto{\Ccal_S}{\Ccal} $ is fully faithful with image those objects lying over $ S $.
\end{observation}

\begin{proposition}\label{prop:recollement_for_functor_to_a_poset}
	Let $ F \colon \fromto{\Ccal}{P} $ be a functor from \acategory to a poset, and let $ Z \subset P $ be a closed subposet with open complement $ U = P \sminus Z $.
	Write $ i \colon \incto{\Ccal_Z}{\Ccal} $ and $ j \colon \incto{\Ccal_U}{\Ccal} $ for the inclusions.
	Then the restriction functors
	\begin{equation*}
		\iupperstar \colon \fromto{\Fun(\Ccal,\Spc)}{\Fun(\Ccal_Z,\Spc)} 
		\andeq 
		\jupperstar \colon \fromto{\Fun(\Ccal,\Spc)}{\Fun(\Ccal_U,\Spc)} 
	\end{equation*}
	exhibit $ \Fun(\Ccal,\Spc) $ as the recollement of $ \Fun(\Ccal_Z,\Spc) $ and $ \Fun(\Ccal_U,\Spc) $.
\end{proposition}

\begin{proof}
	Note that since every object of $ \Ccal $ belongs to either $ \Ccal_U $ or $ \Ccal_Z $ and equivalences in $ \Fun(\Ccal,\Spc) $ are detected pointwise, the functors $ \jupperstar $ and $ \iupperstar $ are jointly conservative.
	Hence the only nontrivial point to check is that the composite $ \jupperstar\ilowerstar $ is constant with value the terminal object of $ \Fun(\Ccal_U,\Spc) $.

	For this, consider the pullback square of \categories
	\begin{equation*}
	    \begin{tikzcd}[sep=2.25em]
	       \emptyset \arrow[dr, phantom, very near start, "\lrcorner", xshift=-0.25em, yshift=0.12em] \arrow[d, "a"'] \arrow[r, "b"] & \Ccal_Z  \arrow[d, hooked, "i"] \\ 
	       \Ccal_U \arrow[r, hooked, "j"'] & \Ccal \period
	    \end{tikzcd}
	\end{equation*}
	Since $ i $ is a right fibration (\Cref{lem:locally_closed_subposets_are_exponentiable}), $ i $ is a \textit{proper} functor in the sense of \cites[Definition 2.22]{arXiv:2108.01924}. 
	(See also \cites[\HTTsubsec{4.1.2}]{HTT}[\S4.4]{MR3931682}.)
	Hence proper basechange \cite[Theorem 2.27]{arXiv:2108.01924} implies that the exchange transformation 
	\begin{equation*}
		\fromto{\jupperstar\ilowerstar}{\alowerstar\bupperstar}
	\end{equation*}
	is an equivalence.
	To complete the proof, notice that the functor
	\begin{equation*}
		\bupperstar \colon \fromto{\Fun(\Ccal_Z,\Spc)}{\Fun(\emptyset,\Spc) \equivalent \pt}
	\end{equation*}
	is the unique functor and the functor $ \alowerstar \colon \fromto{\pt}{\Fun(\Ccal_U,\Spc)} $ picks out the terminal object.
\end{proof}

We now turn to showing that constructible objects satisfy recollement.
Let us introduce a special class of coefficients we are interested in:

\begin{definition}\label{def:compatibility_with_recollements}
	We say that a presentable \category $ \Ecal $ is \defn{compatible with recollements} if for every recollement datum of \topoi
	\begin{equation*}
		\iupperstar \colon \X \to \Z \andeq \jupperstar \colon \X \to \U \comma 
	\end{equation*}
	the induced functors
	\begin{equation*}
		\iupperstar \tensor \Ecal \colon \X \tensor \Ecal \to \Z \tensor \Ecal \andeq \jupperstar \tensor \Ecal \colon \X \tensor \Ecal \to \U \tensor \Ecal 
	\end{equation*}
	exhibit $\X \tensor \Ecal$ as the recollement of $\Z \tensor \Ecal$ and $\U \tensor \Ecal$.
\end{definition}

\begin{recollection}\label{observation:recollement_for_sheaves}
	It follows respectively from \cite[Corollary 2.18 and Proposition 2.26]{arXiv:2108.03545} that if $ \Ecal $ is either compactly assembled or stable, then it is compatible with recollements.
\end{recollection}

\begin{observation}[(see \Cref{rec:open-closed_recollement,prop:formula_for_pullbacks_along_etale_morphisms_and_closed_immersions})]
	Let $ (\X,P) $ be a stratified \topos and let $ Z \subset P $ be a closed subposet with open complement $ U = P \sminus Z $.
	Then the functors
	\begin{equation*}
		\iupperstar_Z \colon \fromto{\X}{\XZ} \andeq \iupperstar_U \colon \fromto{\X}{\X_{U}}
	\end{equation*}
	exhibit $ \X $ as the recollement of $ \XZ $ and $ \XU $.
\end{observation}

In order to prove recollement for constructible objects, we need the following fact:

\begin{lemma}\label{lem:terminal_object_constructible}
	Let $ \X $ be \atopos and let $ \Ecal $ be a presentable \category.
	Then the terminal object of $ \Sh(\X;\Ecal) $ is constant and hypercomplete.
\end{lemma}

\begin{proof}
	For hypercompleteness, observe that since the inclusion $ \Sh(\X^{\hyp};\Ecal) \inclusion \Sh(\X;\Ecal) $ is a right adjoint, it carries the terminal object to the terminal object; hence the terminal object of $ \Sh(\X;\Ecal) $ is hypercomplete.
	For constancy, choose a small \category $ \Ccal $ and a fully faithful embedding $ i \colon \X \inclusion \PSh(\Ccal) $ with left exact left adjoint.
	Note that the constant functor $\Gammaupperstar \colon \Spc \to \PSh(\Ccal) $ is both a left and a right adjoint.
	By the functoriality of the tensor product, the induced functor $ \Gammaupperstar \tensor \Ecal $ is also both a left and a right adjoint.
	In particular, $ \Gammaupperstar \tensor \Ecal $ preserves the terminal object; hence the terminal object of
	\begin{equation*}
		\Sh(\PSh(\Ccal);\Ecal) \equivalent \PSh(\Ccal;\Ecal)
	\end{equation*}
	is constant.
	Since $ i \tensor \Ecal \colon \Sh(\X;\Ecal) \inclusion \Sh(\PSh(\Ccal);\Ecal) $ is a fully faithful right adjoint, it carries the terminal object to the terminal object; we deduce that terminal object of $ \Sh(\X;\Ecal) $ is constant.
\end{proof}

\begin{lemma}\label{lem:preservation_of_constructibility_and_recollement}
	Let $ (\X,P) $ be a stratified \topos, let $ Z \subset P $ be a closed subposet with open complement $ U = P \sminus Z $, and let $ \Ecal $ be a presentable \category.
	If $ \Ecal $ is compatible with recollements, then:
	\begin{enumerate}
		\item\label{lem:preservation_of_constructibility_and_recollement.1} If $ F \in \Cons_U(\XU;\Ecal) $, then $ i_{U,!}(F) \in \ConsP(\X;\Ecal) $.

		\item\label{lem:preservation_of_constructibility_and_recollement.2} If $ G \in \Cons_Z(\XZ;\Ecal) $, then $ i_{Z,*}(G) \in \ConsP(\X;\Ecal) $.

		\item\label{lem:preservation_of_constructibility_and_recollement.3} The composite $ \iupperstar_U i_{Z,*} \colon \fromto{\ConsZ(\XZ;\Ecal)}{\ConsU(\XU;\Ecal)} $ is constant with value the terminal object.

		\item\label{lem:preservation_of_constructibility_and_recollement.4} The functors 
		\begin{equation*}
			\iupperstar_Z \colon \fromto{\ConsP(\X;\Ecal)}{\ConsZ(\XZ;\Ecal)} \andeq \iupperstar_U \colon \fromto{\ConsP(\X;\Ecal)}{\ConsU(\X_{U};\Ecal)}
		\end{equation*}
		are jointly conservative.
	\end{enumerate}
\end{lemma}

\begin{proof}
	All of these claims essentially follow from \cref{observation:recollement_for_sheaves}.
	For (1), note that since $ \iupperstar_U i_{U,!}(F) \equivalent F $, it suffices to show that $ \iupperstar_Z i_{U,!}(F) $ is locally constant on $ \XZ $.
	By recollement, the functor
	\begin{equation*}
		\iupperstar_Z i_{U,!} \colon \fromto{\Sh(\XU;\Ecal)}{\Sh(\XZ;\Ecal)} 
	\end{equation*}
	is constant with value the initial object, which is $U$-constructible.
	For (2), note that since $ \iupperstar_Z i_{Z,*}(G) \equivalent G $, it suffices to show that $ \iupperstar_U i_{Z,*}(G) $ is $U$-constructible on $ \XU $.
	Again by recollement, the functor
	\begin{equation*}
		\iupperstar_U i_{Z,*} \colon \fromto{\Sh(\XZ;\Ecal)}{\Sh(\XU;\Ecal)} 
	\end{equation*}
	is constant with value the terminal object.
	Since $\iupperstar_U \colon \Sh(\X;\Ecal) \to \Sh(\XU;\Ecal)$ is a right adjoint and since the terminal object in $ \Sh(\X;\Ecal) $ is $P$-constructible (\Cref{lem:terminal_object_constructible}), it follows that the terminal object in $\Sh(\XU;\Ecal)$ is $U$-constructible.
	In particular, $ \iupperstar_U i_{Z,*} $ carries $ \ConsZ(\XZ;\Ecal) $ to $ \ConsU(\XU;\Ecal) $, thus proving at the same time (2) and (3).
	Item (4) is immediate from recollement.
\end{proof}

\begin{lemma}\label{lem:lower_shriek_as_kernel}
	Let $ (\X,P) $ be a stratified \topos, let $ U \subset P $ be an open subposet, and let $ \Ecal $ be a presentable \category.
	If $ \Ecal $ is compatible with recollements, then:
	\begin{enumerate}
		\item\label{lem:lower_shriek_as_kernel.1} Write $\emptyset$ for the initial object of $\ConsZ(\XZ;\Ecal)$ and set
		\begin{equation*}
			\ker(\iupperstar_Z) \colonequals \big\{ X \in \ConsP(\X;\Ecal) \mid \iupperstar_Z(X) \equivalent \emptyset \big\} \period 
		\end{equation*}
		Then the induced functor
		\begin{equation*}
			i_{U,!} \colon \ConsU(\XU;\Ecal) \hookrightarrow \ker( \iupperstar_Z ) 
		\end{equation*}
		is an equivalence.
		
		\item\label{lem:lower_shriek_as_kernel.2} Write $\ast$ for the terminal object of $\ConsU(\XU;\Ecal)$ and set
		\begin{equation*}
			\ker(\iupperstar_U) \colonequals \big\{ X \in \ConsP(\X;\Ecal) \mid \iupperstar_U(X) \equivalent \ast \big\} \period 
		\end{equation*}
		Then the induced functor
		\begin{equation*}
			i_{Z,\ast} \colon \ConsZ(\XZ;\Ecal) \hookrightarrow \ker(\iupperstar_U)  
		\end{equation*}
		is an equivalence
	\end{enumerate}
\end{lemma}

\begin{proof}
	In both cases, it suffices to check essential surjectivity.
	So let $X \in \ker(\iupperstar_Z)$ and consider the counit $\counit \colon i_{U,!} \iupperstar_U(X) \to X$ in $\ConsP(\X;\Ecal)$.
	By \enumref{lem:preservation_of_constructibility_and_recollement}{4}, it suffices to show that $\iupperstar(\counit)$ and $\iupperstar_Z(\counit)$ are equivalences.
	The former follows from the full faithfulness of $i_{U,!}$, and the latter follows from the definition of $\ker(\iupperstar_Z)$.
	For (2), the same argument applies, starting with the unit $\unit \colon F \to i_{Z,\ast} \iupperstar_Z(F)$ in place of the counit.
\end{proof}

\begin{lemma}\label{lem:recollements}
	Let $ (\X,P) $ be a stratified \topos, let $ Z \subset P $ be a closed subposet with open complement $ U = P \sminus Z $, and let $ \Ecal $ be a presentable \category.
	If $ \Ecal $ is compatible with recollements, then:
	\begin{enumerate}
		\item\label{lem:recollements.1} If $\ConsP(\X;\Ecal)$ is presentable, then $\ConsP(\X_Z;\Ecal)$ and $\ConsP(\X_U;\Ecal)$ are also presentable.
		
		\item\label{lem:recollements.2} If $\ConsP(\X;\Ecal)$ is closed under colimits in $ \Sh(\X;\Ecal) $, then the functor $\iupperstar_U \colon \fromto{\ConsP(\X;\Ecal)}{\ConsU(\XU;\Ecal)}$ preserves colimits.
		
		\item\label{lem:recollements.3} If $\ConsP(\X;\Ecal)$ is closed under finite limits in $ \Sh(\X;\Ecal) $, then the functor $\iupperstar_Z \colon \fromto{\ConsP(\X;\Ecal)}{\ConsZ(\XZ;\Ecal)}$ is left exact.
		
		\item\label{lem:recollements.4} If $\ConsP(\X;\Ecal)$ is presentable and closed under colimits and finite limits in $ \Sh(\X;\Ecal) $, then the functors $\iupperstar_Z$ and $\iupperstar_U$ exhibit $\ConsP(\X;\Ecal)$ as the recollement of $\ConsZ(\XZ;\Ecal)$ and $\ConsU(\XU;\Ecal)$.
	\end{enumerate}
\end{lemma}

\begin{proof}
	For (1), notice that \enumref{lem:preservation_of_constructibility_and_recollement}{2} implies that $\ConsZ(\XZ;\Ecal)$ is a localization of $\ConsP(\X;\Ecal)$.
	Moreover, \enumref{lem:lower_shriek_as_kernel}{2} immediately implies that $\ConsZ(\X;\Ecal)$ is closed under weakly contractible colimits inside $\ConsP(\X;\Ecal)$; in particular
	\begin{equation*}
		\ConsZ(\X;\Ecal) \subset \ConsP(\X;\Ecal)
	\end{equation*}
	is closed under filtered colimits.
	Thus, the \categorical reflection theorem \cite[Theorem 1.1]{arXiv:2207.09244} implies that $\ConsZ(X;\Ecal)$ is presentable.
	Then, \enumref{lem:lower_shriek_as_kernel}{1} implies that $\ConsU(\XU;\Ecal)$ is presentable.
	
	Item (2) follows from the given assumption, the full faithfulness of $\ConsU(\XU;\Ecal)$ inside of $\Sh(\XU;\Ecal)$, and the fact that $\iupperstar_U \colon \Sh(\X;\Ecal) \to \Sh(\XU;\Ecal)$ preserves colimits and preserve constructible objects.
	A similar argument shows (3) as well.

	We are left to prove (4).
	In virtue of \cref{lem:preservation_of_constructibility_and_recollement}, all we are left to do is to check that $\iupperstar_U$ admits a right adjoint and that $\iupperstar_Z$ is left exact.
	The first statement follows from (1), (2), and the adjoint functor theorem, while the second follows directly from (3).
\end{proof}

In what follows, we will need to use the fact that given an open immersion of \topoi $ \jlowerstar \colon \incto{\U}{\Y} $, the \topos $ \U $ is naturally identified with the slice $ \Y_{/\jlowershriek(1)} $.
Hence we recall some basic results about slice \categories.

\begin{recollection}\label{rec:slicing_adjunctions}
	Let $ i \colon \incto{\Ccal}{\Dcal} $ be a fully faithful functor of \categories and let $ c \in \Ccal $.
	Then:
	\begin{enumerate}
		\item The induced functor $ i \colon \fromto{\Ccal_{/c}}{\Dcal_{/i(c)}} $ is fully faithful.

		\item If $ i \colon \incto{\Ccal}{\Dcal} $ admits a left adjoint $ L \colon \fromto{\Dcal}{\Ccal} $, then $ i \colon \fromto{\Ccal_{/c}}{\Dcal_{/i(c)}} $ admits a left adjoint given by the induced functor
		\begin{equation*}
			L \colon \fromto{\Dcal_{/i(c)}}{{\Ccal_{/Li(c)}} \equivalent \Ccal_{/c}} \period
		\end{equation*}

		\item If $ i \colon \incto{\Ccal}{\Dcal} $ admits a right adjoint $ R \colon \fromto{\Dcal}{\Ccal} $, then $ i \colon \fromto{\Ccal_{/c}}{\Dcal_{/i(c)}} $ admits a right adjoint given by the induced functor
		\begin{equation*}
			R \colon \fromto{\Dcal_{/i(c)}}{{\Ccal_{/Ri(c)}} \equivalent \Ccal_{/c}} \period
		\end{equation*}
	\end{enumerate}
	See \HTT{Proposition}{5.2.5.1}.
\end{recollection}

\begin{lemma}\label{lem:pullbacks_of_slice_categories}
	Let $ \Dcal $ be \acategory, $ \Ccal \subset \Dcal $ a full subcategory, and $ c \in \Ccal $.
	Then the natural square
	\begin{equation*}
	    \begin{tikzcd}[sep=2.25em]
	       \Ccal_{/c} \arrow[d] \arrow[r, hooked] & \Dcal_{/c}  \arrow[d] \\ 
	       \Ccal \arrow[r, hooked] & \Dcal
	    \end{tikzcd}
	\end{equation*}
	is a pullback square of \categories.
	Here the vertical functors are the forgetful functors.
\end{lemma}

\begin{proof}
	Consider the commutative cube
	\begin{equation*}
        \begin{tikzcd}[column sep={12ex,between origins}, row sep={8ex,between origins}]
            \Ccal_{/c} \arrow[rr] \arrow[dd]  \arrow[dr, hooked] & & \Fun([1],\Ccal) \arrow[dd, "{(\source,\target)}"' near end]  \arrow[dr, hooked] \\
            & \Dcal_{/c} \arrow[rr, crossing over] & & \Fun([1],\Dcal) \arrow[dd, "{(\source,\target)}"]  \\
            \Ccal \cross \{c\} \arrow[rr] \arrow[dr, hooked] & & \Ccal \cross \Ccal \arrow[dr, hooked] \\
            & \Dcal \cross \{c\} \arrow[rr] \arrow[from=uu, crossing over] & & \Dcal \cross \Dcal  \period
        \end{tikzcd}
    \end{equation*}
    By definition, the front and back vertical faces are pullbacks.
    Since $ \Ccal \subset \Dcal $ is a full subcategory, the right-hand vertical face is a pullback.
    Hence the left-hand vertical face is also a pullback.
\end{proof}

Let us now give an alternative description of constructible objects in a stratified \topos obtained by pulling back to an open subposet.

\begin{lemma}\label{lem:pullback_description_of_constructible_objects_on_an_open}
	Let $ (\X,P) $ be a stratified \topos and let $ U \subset P $ be an open subposet.
	Then: 
	\begin{enumerate}
		\item\label{lem:pullback_description_of_constructible_objects_on_an_open.1} The square
		\begin{equation*}
		    \begin{tikzcd}[sep=2.25em]
		       \ConsU(\XU) \arrow[d, hooked, "i_{U,!}"'] \arrow[r, hooked] & \XU \arrow[d, hooked, "i_{U,!}"]  \\ 
		       \ConsP(\X) \arrow[r, hooked] & \X
		    \end{tikzcd}
		\end{equation*}
		is a pullback square of \categories.

		\item\label{lem:pullback_description_of_constructible_objects_on_an_open.2} There is a commutative square
		\begin{equation*}
			\begin{tikzcd}[sep=2.25em]
		       \ConsU(\XU) \arrow[d, "\iUlowershriek"', "\wr"{xshift=-0.25ex}] \arrow[r, hooked] & \XU \arrow[d, "\iUlowershriek", "\wr"'{xshift=0.25ex}]  \\ 
		       \ConsP(\X)_{/\iUlowershriek(1)} \arrow[r, hooked] & \X_{/\iUlowershriek(1)}
		    \end{tikzcd}
		\end{equation*}
		where the vertical functors are equivalences and the horizontal functors are the natural inclusions.
	\end{enumerate}
\end{lemma}

\begin{proof}
		For (1), note that it suffices to show that the fully faithful functor
	\begin{equation*}
		i_{U,!} \colon \incto{\ConsU(\XU)}{\ConsP(\X) \intersect i_{U,!}(\XU)}
	\end{equation*}
	is essentially surjective.
	For this, let $ G \in \XU $ be such that $ i_{U,!}(G) $ is $ P $-constructible.
	Write $ Z \colonequals P \sminus U $.
	Then $ \iupperstar_Z i_{U,!}(G) = \emptyset $ and $ \iupperstar_U i_{U,!}(G) $ is $ U $-constructible.
	Hence $ G \in \ConsU(\XU) $, completing the proof.
	
	For (2), note that $ \iUlowerstar \colon \incto{\XU}{\X} $ is an open immersion of \topoi, the exceptional left adjoint $ \iUlowershriek \colon \incto{\XU}{\X} $ induces an equivalence $ \equivto{\XU}{\X_{/\iUlowershriek(1)}} $ fitting into a commutative triangle 
	\begin{equation*}
		\begin{tikzcd}
			\XU \arrow[dr, "\iUlowershriek"', hooked] \arrow[rr, "\sim"{yshift=-0.25em}] & & \X_{/\iUlowershriek(1)} \arrow[dl, hooked', "\text{forget}"] \\ 
			& \X & \phantom{\X_{/\iUlowershriek(1)}} \period
		\end{tikzcd}
	\end{equation*}
	Since $ \iUlowershriek(1) \in \ConsP(\X) $, the claim follows from item (1) combined with \Cref{lem:pullbacks_of_slice_categories}.
\end{proof}

\begin{proposition}[(recollement)]\label{prop:recollement_for_constructible_objects}
	Let $s_\ast \colon \X \to \Fun(P,\Spc)$ be an exodromic stratified \topos and let $ Z \subset P $ be a closed subposet with open complement $ U = P \sminus Z $.
	Then:
	\begin{enumerate}
		\item\label{prop:recollement_for_constructible_objects.2} The functors 
		\begin{equation*}
			\iupperstar_Z \colon \fromto{\ConsP(\X)}{\ConsZ(\XZ)}
			\andeq 
			\iupperstar_U \colon \fromto{\ConsP(\X)}{\ConsU(\XU)} 
		\end{equation*}
		exhibit $ \ConsP(\X) $ as the recollement of $ \ConsZ(\XZ) $ and $ \ConsU(\XU) $.
		
		\item\label{prop:recollement_for_constructible_objects.3} The stratified \topos $ (\XU,U) $ is exodromic, the morphism $ i_{U,\ast} \colon \incto{(\XU,U)}{(\X,P)} $ is exodromic, and the induced functor
		\begin{equation*}
			\fromto{\Piinfty(X_U,U)}{\Piinfty(X,P)_U}
		\end{equation*}
		is an equivalence.

		\item\label{prop:recollement_for_constructible_objects.4} The stratified \topos $ (\XZ,Z) $ is exodromic, the morphism $ i_{Z,\ast} \colon \incto{(\XZ,Z)}{(\X,P)} $ is exodromic, and the induced functor
		\begin{equation*}
			\fromto{\Piinfty(X_Z,Z)}{\Piinfty(X,P)_Z}
		\end{equation*}
		is an equivalence.
	\end{enumerate}
\end{proposition}

\begin{proof}
	Item (1) follows directly from \enumref{lem:recollements}{4}.
	For (2), let us first prove that $\ConsU(\XU)$ is closed under limits and colimits in $\XU$.
	By \enumref{lem:pullback_description_of_constructible_objects_on_an_open}{2}, we have a commutative square
	\begin{equation*}
		\begin{tikzcd}[sep=2.25em]
	       \ConsU(\XU) \arrow[d, "\iUlowershriek"', "\wr"{xshift=-0.25ex}] \arrow[r, hooked] & \XU \arrow[d, "\iUlowershriek", "\wr"'{xshift=0.25ex}]  \\ 
	       \ConsP(\X)_{/\iUlowershriek(1)} \arrow[r, hooked] & \X_{/\iUlowershriek(1)}
	    \end{tikzcd}
	\end{equation*}
	where the vertical functors are equivalences.
	Since $ (\X,P) $ is exodromic, the inclusion $ \ConsP(\X) \subset \X $ admits both a left and right adjoint.
	Hence \Cref{rec:slicing_adjunctions} shows that the inclusion $ \ConsU(\XU) \subset \XU $ admits both a left and right adjoint.

	Write $ s_{U,*} \colon \fromto{\XU}{\Fun(U,\Spc)} $ for the induced stratification and $ j \colon \incto{U}{P} $ for the inclusion.
	All we are left to show is that the \category $ \ConsU(\XU) $ is atomically generated by $ \Piinfty(\X,P)_U $ and that the pullback functor $ \supperstar_U \colon \fromto{\Fun(U,\Spc)}{\ConsU(\XU)} $ preserves limits.
	To see that $ \ConsU(\XU) $ is atomically generated by $ \Piinfty(\X,P)_U $, notice that since $ \iZupperstar \iUlowershriek(1) = \emptyset $ and $ \iUupperstar\iUlowershriek(1) = 1 $, the fully faithful functor 
	\begin{equation*}
		\iUlowershriek \colon \incto{\ConsU(\XU)}{\ConsP(\X) \equivalent \Fun(\Piinfty(\X,P),\Spc)}
	\end{equation*}
	has image those functors $ F \colon \fromto{\Piinfty(\X,P)}{\Spc} $ such that the composite
	\begin{equation*}
		\begin{tikzcd}
			\Piinfty(\X,P)_Z \arrow[r, hooked] & \Piinfty(\X,P) \arrow[r, "F"] & \Spc
		\end{tikzcd}
	\end{equation*}
	is constant with value the initial object.
	Now note that this full subcategory coincides with the image of the fully faithful functor
	\begin{equation*}
		\incto{\Fun(\Piinfty(\X,P)_U,\Spc)}{\Fun(\Piinfty(\X,P),\Spc)}
	\end{equation*}
	given by left Kan extension along the inclusion $ \incto{\Piinfty(\X,P)_U}{\Piinfty(\X,P)} $.

	To see that $ \supperstar_U \colon \fromto{\Fun(U,\Spc)}{\ConsU(\XU)} $ preserves limits, notice that we have a commutative square
	\begin{equation*}
		\begin{tikzcd}[sep=2.25em]
	       \Fun(P,\Spc) \arrow[d, "\supperstar"'] \arrow[r, "\jupperstar"] & \Fun(U,\Spc) \arrow[d, "\supperstar_U"]  \\ 
	       \ConsP(\X) \arrow[r, "\iupperstar_U"'] & \ConsU(\XU) \period
	    \end{tikzcd}
	\end{equation*}
	Since $ \jlowerstar $ is fully faithful, we see that there are equivalences
	\begin{equation*}
		\supperstar_U \equivalent \supperstar_U \jupperstar \jlowerstar \equivalent \iupperstar_U \supperstar \jlowerstar \period
	\end{equation*}
	Since the functors $ \iupperstar_U $, $ \supperstar $, and $ \jlowerstar $ all preserve limits, we deduce that $ \supperstar_U $ preserves limits, as desired.

	For (3), recall from \enumref{lem:lower_shriek_as_kernel}{1} that
	\begin{equation}\label{eq:ConsZ_as_a_kernel}
		\ConsZ(\XZ) \equivalent \ker\Big(
		\iupperstar_U \colon \fromto{\ConsP(\X)}{\ConsU(\XU)}
		\Big) 
		\period
	\end{equation}
	Since $ (\X,P) $ is exodromic by assumption, $ (\XU,U) $ is exodromic by (2), and $ \iupperstar_U $ preserves limits and colimits, we deduce that $ \ConsZ(\XZ) \subset \XZ $ is closed under limits and colimits.
	\Cref{prop:recollement_for_functor_to_a_poset} and the identification \eqref{eq:ConsZ_as_a_kernel} show that the \category $ \ConsZ(\XZ) $ is atomically generated by $ \Piinfty(\X,P)_Z $ and the functor $ \iupperstar_Z \colon \fromto{\ConsP(\X)}{\ConsZ(\X)} $ preserves limits and colimits.
	
	Write $ s_{Z,*} \colon \fromto{\XZ}{\Fun(Z,\Spc)} $ for the induced stratification and $ i \colon \incto{Z}{P} $ for the inclusion.
	All that remains to be shown is that the pullback functor $ \supperstar_Z \colon \fromto{\Fun(Z,\Spc)}{\ConsZ(\XZ)} $ preserves limits. 
	For this, notice that we have a commutative square
	\begin{equation*}
		\begin{tikzcd}[sep=2.25em]
	       \Fun(P,\Spc) \arrow[d, "\supperstar"'] \arrow[r, "\iupperstar"] & \Fun(Z,\Spc) \arrow[d, "\supperstar_Z"]  \\ 
	       \ConsP(\X) \arrow[r, "\iupperstar_Z"'] & \ConsZ(\XZ) \period
	    \end{tikzcd}
	\end{equation*}
	Since $ \ilowerstar $ is fully faithful, we see that there are equivalences
	\begin{equation*}
		\supperstar_Z \equivalent \supperstar_Z \iupperstar \ilowerstar \equivalent \iupperstar_Z \supperstar \ilowerstar \period
	\end{equation*}
	Since the functors $ \iupperstar_Z $, $ \supperstar $, and $ \ilowerstar $ all preserve limits, we deduce that $ \supperstar_Z $ preserves limits, as desired.
\end{proof}

\begin{nul}
	In the setting of \Cref{prop:recollement_for_constructible_objects}, the recollement takes the following form: 
	\begin{equation*}
		\begin{tikzcd}[sep=6em]
			\Cons_Z(\XZ) \arrow[r, "i_{Z,*}"', shift right=2ex, hooked] \arrow[r, "\iZlowersharpcons", shift left=2ex, hooked] & \ConsP(\X) \arrow[l, "\iupperstar_Z" description] \arrow[r, "\iupperstar_U " description]  & \Cons_U(\XU) \period \arrow[l, shift left=2ex, hooked', "\iUlowerstarcons"] \arrow[l, shift right=2ex, hooked', "\iUlowershriek"']
		\end{tikzcd}
	\end{equation*}
	Here the functors \smash{$ i_{Z,*} $, $ \iupperstar_Z $, $ i_{U,!} $, and $ \iupperstar_U $} agree with the ones at the level of the \topoi $ \XZ $, $ \XU $, and $ \X $.
	The functor \smash{$ \iUlowerstarcons $} does not necessarily agree with the pushforward \smash{$ i_{U,\ast} \colon \incto{\X_{U}}{\X} $}, and the functor \smash{$ \iZlowersharpcons $} is `extra' in the sense that it does not come for free from the theory of recollements.
\end{nul}

For the next result, we need the following useful characterization of when a functor of exit-path \categories is fully faithful in terms of the constructible pushforwards:

\begin{lemma}\label{lem:equivalent_conditions_for_full_faithfulness_of_exit-paths}
	Let $ \flowerstar \colon \fromto{(\X,P)}{(\Y,Q)} $ be a morphism between exodromic stratified \topoi.
	If $ \flowerstar $ is exodromic, then the following are equivalent:
	\begin{enumerate}
		\item\label{lem:equivalent_conditions_for_full_faithfulness_of_exit-paths.1} The functor $f^{\ex} \colon  \fromto{\Piinfty(\X,P)}{\Piinfty(\Y,Q)} $ is fully faithful.

		\item\label{lem:equivalent_conditions_for_full_faithfulness_of_exit-paths.2} The functor $ \flowersharpcons \colon \fromto{\ConsP(\X)}{\ConsQ(\Y)} $ is fully faithful.

		\item\label{lem:equivalent_conditions_for_full_faithfulness_of_exit-paths.3} The functor $ \flowerstarcons \colon \fromto{\ConsP(\X)}{\ConsQ(\Y)} $ is fully faithful.
	\end{enumerate}
\end{lemma}

\begin{proof}
	Immediate from the fact that a functor $ F \colon \fromto{\Ccal}{\Dcal} $ is fully faithful if and only if either of the functors
	\begin{equation*}
		F_{!}, F_{*} \colon \fromto{\Fun(\Ccal,\Spc)}{\Fun(\Dcal,\Spc)}
	\end{equation*}
	given by left or right Kan extension along $ F $ is fully faithful.
\end{proof}

\noindent By writing a locally closed immersion of posets as the composite of a closed immersion and an open immersion, we deduce the main result of this subsection:

\begin{corollary}[(stability under pulling back to locally closed subposets)]\label{cor:stability_under_pulling_back_to_locally_closed_subposets}
	Let $ (\X,P) $ be an exodromic stratified \topos and let $ S \subset P $ be a locally closed subposet.
	Then: 
	\begin{enumerate}
		\item\label{cor:stability_under_pulling_back_to_locally_closed_subposets.1} The stratified \topos $ (\XS,S) $ is exodromic and the morphism of stratified \topoi $ i_{S,\ast} \colon \incto{(\XS,S)}{(\X,P)} $ is exodromic.

		\item\label{cor:stability_under_pulling_back_to_locally_closed_subposets.2} The \topos $ \XS $ is monodromic. 

		\item\label{cor:stability_under_pulling_back_to_locally_closed_subposets.3} The natural functor \smash{$ \fromto{\Piinfty(\XS,S)}{\Piinfty(\X,P)_S} $} is an equivalence.

		\item\label{cor:stability_under_pulling_back_to_locally_closed_subposets.4} The functors \smash{$ \iSlowersharpcons, \iconsSlowerstar \colon \ConsS(\XS) \to \ConsP(\X) $} are both fully faithful.

		\item\label{cor:stability_under_pulling_back_to_locally_closed_subposets.5} The natural functor $ \fromto{\Piinfty(\X,P)}{P} $ is conservative.
	\end{enumerate}
\end{corollary}

\begin{proof}
	Choose an open subposet $ U \subset P $ containing $ S $ such that $ S $ is closed in $ U $.
	For (1), apply \enumref{prop:recollement_for_constructible_objects}{3} to both the open inclusion $ U \subset P $ and closed inclusion $ S \subset U $.
	Item (2) follows from (1) and \Cref{lem:exodromy_implies_monodromy}.
	For (3), applying \enumref{prop:recollement_for_constructible_objects}{4} to the closed inclusion $ S \subset U $ and the open inclusion $ U \subset P $, we see that there are equivalences
	\begin{align*}
		\Piinfty(\XS,S) &\equivalence \Piinfty(\XU,U) \cross_U S \\
		&\equivalence (\Piinfty(\X,P) \cross_P U) \cross_U S \\
		&\equivalent \Piinfty(\X,P)_S \period
	\end{align*}
	By \cref{obs:C_S_to_C_fully_faithful}, the natural functor $ \Piinfty(\X,P)_S \to \Piinfty(\X,P) $ is fully faithful; hence \Cref{lem:equivalent_conditions_for_full_faithfulness_of_exit-paths} shows that (4) follows from (3).
	For (5), note that by \Cref{rec:characterization_of_conservative_functors_to_posets}, we need to show that each fiber $ \Piinfty(\X,P)_{p} $ is \agroupoid.
	Since each $ p \in P $ is locally closed, item (1) shows that
	\begin{equation*}
		\Piinfty(\X,P)_{p} \equivalent \Piinfty(\X_p,\{p\}) \period
	\end{equation*}
	The conclusion now follows from the fact that $ \Piinfty(\X_p,\{p\}) $ is \agroupoid (\Cref{rec:monodromy}).
\end{proof}


We conclude by recording a few consequences of \Cref{cor:stability_under_pulling_back_to_locally_closed_subposets}.
First, we can describe the objects of the exit-path \category. 

\begin{observation}[{(the objects of $ \Piinfty(\X,P)$)}]\label{obs:objects_of_Exit}
	Let $ (\X,P) $ be an exodromic stratified space.
	\Cref{cor:stability_under_pulling_back_to_locally_closed_subposets} implies that there is a natural identification 
	\begin{equation*}
		\Piinfty(\X,P)^{\equivalent} \equivalent \coprod_{p \in P} \Piinfty(\X_p)
	\end{equation*}
	between the maximal \subgroupoid of $ \Piinfty(\X,P) $ and the coproduct of the shapes of the \topoi $ \X_p $.
\end{observation}

\noindent Second, equivalences of constructible objects can be checked by pulling back to strata:

\begin{corollary}\label{cor:restricting_to_strata_is_jointly_conservative}
	Let $ (\X,P) $ be an exodromic stratified \topos and let $ \{S_{\alpha}\}_{\alpha \in A} $ be a collection of locally closed subposets of $ P $ such that $ \Union_{\alpha \in A} S_{\alpha} = P $.
	Then the restriction functors 
	\begin{equation*}
		\set{ \iupperstar_{S_{\alpha}} \colon \ConsP(\X) \to \Cons_{S_{\alpha}}(\X_{S_{\alpha}}) }_{\alpha \in A}
	\end{equation*}
	are jointly conservative.
\end{corollary}

\begin{proof}
	Since each $ p \in P $ is locally closed, by further restricting to the strata, it suffices to show that the restriction functors 
	\begin{equation*}
		\set{ \iupperstar_{p} \colon \ConsP(\X) \to \LC(\X_{p}) }_{p \in P}
	\end{equation*}
	are jointly conservative.
	By \Cref{cor:stability_under_pulling_back_to_locally_closed_subposets}, the stratified \topos $ (\X_{p},\{p\}) $ is exodromic and the inclusion $ i_{p,\ast} \colon \incto{(\X_{p},\{p\})}{(\X,P)} $ is exodromic.
	Hence the claim follows from the identification of the restriction functor \smash{$ \iupperstar_{p} \colon \ConsP(\X) \to \LC(\X_{p}) $} with the functor
	\begin{equation*}
		\fromto{\Fun(\Piinfty(\X,P),\Spc)}{\Fun(\Piinfty(\X_p),\Spc)}
	\end{equation*}
	given by precomposition with the inclusion $ \incto{\Piinfty(\X_p) \equivalent \Piinfty(\X,P)_p}{\Piinfty(\X,P)} $.
\end{proof}

\noindent Finally, the \category of constructible objects with arbitrary presentable coefficients is still presentable:

\begin{lemma}\label{lem:closure_under_colimits_for_monodromic_strata}
	Let $ (\X,P) $ be a stratified \topos and let $ \Ecal $ be a presentable \category.
	If for each $ p \in P $, the stratum $ \X_p $ is monodromic, then the the \category $\ConsP(\X;\Ecal)$ is presentable and closed under colimits in $ \Sh(\X;\Ecal) $.
\end{lemma}

\begin{proof}
	By definition, $\ConsP(\X;\Ecal)$ fits into a pullback square of \categories
	\begin{equation*}
		\begin{tikzcd}[column sep =4.5em]
			\ConsP(\X;\Ecal) \arrow[r] \arrow[d, hooked] & \prod_{p \in P} \LC(\X_p;\Ecal) \arrow[d, hooked, shorten <= -.75em] \\
			\Sh(\X;\Ecal) \arrow[r, "\prod_{p \in P} \iupperstar_p"'] & \prod_{p \in P} \Sh(\X_p;\Ecal)
		\end{tikzcd} 
	\end{equation*}
	Since each $\X_p$ is monodromic, by \cref{rec:monodromy}, $\LC(\X_p;\Ecal)$ is presentable and closed under limits and colimits in $\Sh(\X_p;\Ecal)$.
	The fact that the forgetful functor \smash{$ \fromto{\PrL}{\CATinfty} $} preserves limits \HTT{Proposition}{5.5.3.13} completes the proof. 
\end{proof}

\begin{corollary}\label{cor:exodromic_implies_universally_presentable}
	Let $ (\X,P) $ be an exodromic stratified topos.
	Then for any presentable \category $ \Ecal $, the \category $\ConsP(\X;\Ecal)$ is presentable and closed under colimits in $ \Sh(\X;\Ecal) $.
\end{corollary}

\begin{proof}
	Combine \Cref{cor:stability_under_pulling_back_to_locally_closed_subposets,lem:closure_under_colimits_for_monodromic_strata}.
\end{proof}


\subsection{All morphisms are exodromic}\label{subsec:all_morphisms_are_exodromic}

We now use \Cref{cor:stability_under_pulling_back_to_locally_closed_subposets} to show that \textit{every} morphsim between exodromic stratified \topoi is exodromic.
We start by proving this in the special case where the target is trivially stratified.

\begin{lemma}\label{lem:morphisms_to_trivially_stratified_topoi_are_exodromic}
	Let $ \flowerstar \colon \fromto{(\X,P)}{(\Y,\pt)} $ be a morphism of stratified \topoi, where the target is trivially stratified.
	If the stratified \topoi $ (\X,P) $ and $ (\Y,\pt) $ are exodromic, then the morphism $ \flowerstar $ is exodromic.
\end{lemma}

\begin{proof}
	Since $ (\X,P) $ is exodromic, \enumref{lem:exodromy_implies_monodromy}{1} shows that the trivially stratified \topos $ (\X,\pt) $ is exodromic.
	The morphism $ \flowerstar $ factors as a composite
	\begin{equation*}
		\begin{tikzcd}
			(\X,P) \arrow[r] & (\X,\pt) \arrow[r] & (\Y,\pt) \period
		\end{tikzcd}
	\end{equation*}
	By \enumref{lem:exodromy_implies_monodromy}{2}, the left-hand morphism is exodromic, and by \Cref{ex:morphisms_between_trivially_stratified_topoi_are_exodromic} the right-hand morphism is exodromic.
	Hence the composite is exodromic.
\end{proof}

For the following result, we introduce the following variant of \Cref{ntn:inclusion_i_S}.

\begin{notation}\label{ntn:R_p}
	Let $ (\X,R) $ be a stratified \topos and $ \phi \colon \fromto{R}{P} $ be a map of posets.
	Given $ p \in P $, we write $ R_p \colonequals \phi\inv(p) $ for the full subposet of $ R $ given by the fiber of $ \phi $ over $ p $.
	Note that $ \X_p = \X_{R_p} $.
	Hence the stratum $ \X_p $ is naturally a $ R_p $-stratified \topos and the geometric morphism \smash{$ i_{p,\ast} \colon \incto{\X_p}{\X} $} defines a morphism of stratified \topoi $ \incto{(\X_p,R_p)}{(\X,R)} $.
\end{notation}

\begin{theorem}[(all morphisms are exodromic)]\label{thm:all_morphisms_are_exodromic}
	Let $ \flowerstar \colon \fromto{(\X,P)}{(\Y,Q)} $ be a morphism between exodromic stratified \topoi.
	Then $ \flowerstar $ is exodromic.
\end{theorem}

\begin{proof}
	By \Cref{cor:restricting_to_strata_is_jointly_conservative}, the functors
	\begin{equation*}
		\set{ \iupperstar_{P_q} \colon \ConsP(\X) \to \Cons_{P_q}(\X_q) }_{q \in Q}
	\end{equation*}
	are jointly conservative.
	Moreover, since the subposet $ P_q \subset P $ is locally closed, by \enumref{cor:stability_under_pulling_back_to_locally_closed_subposets}{1} these functors also preserve limits and colimits.
	Hence it suffices to show that for each $ q \in Q $, the composite \smash{$ \iupperstar_{P_q} \fupperstar $} preserves limits and colimits.
	
	As in \Cref{obs:pulling_back_stratified_geometric_morphisms}, write $ f_{q} \colon \fromto{(\X_{q},P_{q})}{(\Y_q,\{q\})} $ for the induced morphism of stratified \topoi.
	Note that we have a commutative square
	\begin{equation*} 
		\begin{tikzcd}[row sep=3em, column sep=4.5em]
			\ConsQ(\Y) \arrow[r, "\fupperstar"] \arrow[d, "\iupperstar_{q}"'] & \ConsP(\X) \arrow[d, "\iupperstar_{P_q}"] \\
			\LC(\Y_q) \arrow[r, "\fupperstar_q"'] & \Cons_{P_q}(\X_q) \period
		\end{tikzcd} 
	\end{equation*}
	Again by \enumref{cor:stability_under_pulling_back_to_locally_closed_subposets}{1}, the functor \smash{$ \iupperstar_{q} $} preserves limits and colimits.
	To complete the proof, note that by \enumref{cor:stability_under_pulling_back_to_locally_closed_subposets}{1} the stratified \topoi $ (\X_q,P_q) $ and $ (\Y_q, \{q\}) $ are exodromic; hence \Cref{lem:morphisms_to_trivially_stratified_topoi_are_exodromic} shows that functor \smash{$ \fupperstar_q $} preserves limits and colimits.
	Thus $ \fupperstar_q\iupperstar_{q} $ preserves limits and colimits.
\end{proof}

We can now cleanly state the functoriality of exit-path \categories.
For this, recall \Cref{ntn:PrRat,def:StrTop}.

\begin{notation}\label{ntn:StrTopex}
	Write \smash{$ \StrTopex \subset \StrTop $} for the full subcategory spanned by the exodromic stratified \topoi.
\end{notation}

\begin{observation}[(functoriality of exit-path \categories)]
	The assignment $ \goesto{(\X,P)}{\Piinfty(\X,P)} $ refines to a functor
	\begin{equation*}
		\Piinfty(-,-) \colon \fromto{\StrTopex}{\Catidem} \period
	\end{equation*}
	Specifically, this functor is given by the composite
	\begin{equation*}
		\begin{tikzcd}[sep=3em]
			\StrTopex \arrow[r, "\Cons"] & (\PrRat)^{\op} \equivalent \PrLat \arrow[r, "(-)^{\ex}", "\sim"'{yshift=0.25ex}] & \Catidem \comma
		\end{tikzcd}
	\end{equation*}
	where the left-hand functor sends $ (\X,P) $ to the \category \smash{$ \ConsP(\X) $} with functoriality given by pullback, and the right-hand functor sends an atomically generated \category $ \Ccal $ to the \category \smash{$ \Ccal^{\ex} = (\Ccal^{\at})^{\op} $} given by the opposite of the subcategory of atomic objects.
\end{observation}


\subsection{Stability under coarsening}\label{subsec:stability_under_coarsening}

Let $ (\X,R) $ be an exodromic stratified \topos, and let $ \phi \colon \fromto{R}{P} $ be a map of posets.
In this subsection, show that $ (\X,P) $ is also exodromic and express $ \Piinfty(\X,P) $ as a localization of $ \Piinfty(\X,R) $.

\begin{observation}\label{obs:pullback_along_refinement_is_inclusion}
	Let $ (\X,R) $ be a stratified \topos and let $ \phi \colon \fromto{R}{P} $ be a map of posets.
	Since the morphism of stratified \topoi $ \fromto{(\X,R)}{(\X,P)} $ is the identity on the underlying \topos $ \X $, the pullback along $ \fromto{(\X,R)}{(\X,P)} $ is simply the inclusion
	\begin{equation*}
		\incto{\ConsP(\X)}{\ConsR(\X)} \period
	\end{equation*}
\end{observation}

\begin{lemma}\label{lem:coarsenings_of_exodromic_stratifications}
	Let $ (\X,R) $ be a stratified \topos and let $ \phi \colon \fromto{R}{P} $ be a map of posets.
	If $ (\X,R) $ is exodromic, then the following conditions are equivalent:
	\begin{enumerate}
		\item\label{lem:coarsenings_of_exodromic_stratifications.1} The stratified \topos $ (\X,P) $ is exodromic.

		\item\label{lem:coarsenings_of_exodromic_stratifications.3} The full subcategory $ \ConsP(\X) \subset \ConsR(\X) $ is closed under both limits and colimits.
	\end{enumerate}
\end{lemma}

\begin{proof}
	Note that by \Cref{obs:pullback_along_refinement_is_inclusion} we immediately have (1) $ \Rightarrow $ (2).

	To show is that (2) $ \Rightarrow $ (1), we check the three conditions of \Cref{def:exit_path}.
	First note that since $ (\X,R) $ is exodromic, the \category \smash{$ \ConsR(\X) $} is atomically generated.
	Hence (2) and \Cref{prop:atomic_generation} imply that the full subcategory \smash{$ \ConsP(\X) $} is atomically generated and the inclusion 
	\begin{equation*}
		\ConsP(\X) \subset \ConsR(\X)
	\end{equation*}
	admits both a left and a right adjoint.
	Since $ (\X,R) $ is exodromic, the full subcategory
	\begin{equation*}
		\ConsR(\X) \subset \X
	\end{equation*}
	is closed under limits and colimits; hence \smash{$ \ConsP(\X) \subset \X $} is also closed under limits and colimits.

	Write $ \tlowerstar \colon \fromto{\X}{\Fun(R,\Spc)} $ for the stratification, and $ \slowerstar \colon \fromto{\X}{\Fun(P,\Spc)} $ for the composite $ \philowerstar \tlowerstar $.
	All that remains to be shown is that the pullback functor
	\begin{equation*}
		\supperstar \colon \fromto{\Fun(P,\Spc)}{\ConsP(\X)}
	\end{equation*}
	preserves limits and colimits.
	For this, note that we have a commutative square
	\begin{equation*}
		\begin{tikzcd}[column sep=4.5em, row sep=2.5em]
			\Fun(P,\Spc) \arrow[r, "\phiupperstar"] \arrow[d, "\supperstar"'] & \Fun(R,\Spc) \arrow[d, "\tupperstar"] \\ 
			\ConsP(\X) \arrow[r, hooked] & \ConsR(\X) \period
		\end{tikzcd}
	\end{equation*}
	Here, the bottom horizontal functor is the inclusion, which is also the pullback along the refinement map $ \fromto{(\X,R)}{(\X,P)} $.
	The functor $ \phiupperstar $ preserves limits and colimits; by assumption both the bottom horizontal functor and $ \tupperstar $ preserve limits and colimits.
	Hence $ \supperstar $ also preserves limits and colimits.
\end{proof}

To compute the exit-path \category of a coarsening, we also make use of the following:

\begin{lemma}\label{lem:composite_of_a_localization_with_a_conservative_functor}
	Let $ F \colon \fromto{\Ccal}{\Dcal} $ and $ G \colon \fromto{\Dcal}{\Ecal} $ be functors between \categories.
	Write $ W \subset \Mor(\Ccal) $ for the collection of morphisms that $ GF $ carries to equivalences in $ \Ecal $.
	If $ F $ is a localization and $ G $ is conservative, then $ F $ induces an equivalence 
	\begin{equation*}
		\equivto{\Ccal[W\inv]}{\Dcal} \period
	\end{equation*} 
\end{lemma}

\begin{proof}
	Since $ F $ is a localization, it suffices to show that given a morphism $ f $ in $ \Ccal $, the morphism $ F(f) $ is an equivalence if and only if $ f \in W $.
	To see this, note that since $ G $ is conservative, $ F(f) $ is an equivalence if and only if $ GF(f) $ is an equivalence.
\end{proof}

For the proof of stability under coarsening, recall \Cref{ntn:inclusion_i_S,ntn:R_p}.
We also introduce the following notation:

\begin{notation}\label{ntn:W_P}
	Let $ (\X,R) $ be a stratified \topos and $ \phi \colon \fromto{R}{P} $ be a map of posets.
	If $ (\X,R) $ is exodromic, write $ W_P \subset \Mor(\Piinfty(\X,R)) $ for the collection of morphisms sent to equivalences by the composite 
	\begin{equation*}
		\Piinfty(\X,R) \to R \to P \period
	\end{equation*}
\end{notation}

\begin{theorem}[(stability under coarsening)]\label{thm:stability_under_coarsening}
	Let $ (\X,R) $ be an exodromic stratified \topos, and let $ \phi \colon \fromto{R}{P} $ be a map of posets.
	Then:
	\begin{enumerate}
		\item\label{thm:stability_under_coarsening.1} The stratified \topos $ (\X,P) $ is exodromic.

		\item\label{thm:stability_under_coarsening.2} The natural functor $ \fromto{\Piinfty(\X,R)}{\Piinfty(\X,P)} $ induces an equivalence $ \equivto{\Piinfty(\X,R)[W_P\inv]}{\Piinfty(\X,P)} $.
	\end{enumerate}
\end{theorem}

\begin{proof}
	First we prove (1).
	Since $ (\X,R) $ is exodromic, \Cref{cor:exodromic_implies_universally_presentable} shows that the subcategory
	\begin{equation*}
		\ConsP(\X) \subset \ConsR(\X)
	\end{equation*}
	is closed under colimits.
	To prove closure under limits, let $ F_{\bullet} \colon A \to \ConsP(\X) $ be a diagram.
	Write
	\begin{equation*} 
		F_{-\infty} \colonequals \lim_{\alpha \in A} F_{\alpha} 
	\end{equation*}
	for the limit computed in \smash{$ \ConsR(\X) $}.
	We have to prove that for each $ p \in P $, the restriction \smash{$ i_p^{\ast}(F_{-\infty}) $} is locally constant.
	Again by \enumref{cor:stability_under_pulling_back_to_locally_closed_subposets}{1}, the functor
	\begin{equation*} 
		i_p^{\ast} \colon \ConsR(\X) \to \Cons_{R_p}(\X_p)
	\end{equation*}
	preserves limits.
	Therefore,
	\begin{equation*} 
		i_p^{\ast}(F_{-\infty}) \equivalent \lim_{\alpha \in A} i_p^{\ast}(F_{\alpha}) \period 
	\end{equation*}
	By assumption, each \smash{$ i_p^{\ast}(F_{\alpha}) $} is a locally constant object of $ \X_p $.
	Since $ \X_p = \X_{R_p} $, by \enumref{cor:stability_under_pulling_back_to_locally_closed_subposets}{2}, the trivially stratified \topos $ (\X_p,\{p\}) $ is exodromic.
	Hence the subcategory
	\begin{equation*}
		\LC(\X_p) \subset \X_p
	\end{equation*}
	is closed under limits (\Cref{rec:monodromy}).
	Therefore, \smash{$ i_p^{\ast}(F_{-\infty}) $} is locally constant, as desired.

	For item (2), note that (1) and \Cref{prop:atomic_generation} imply that the induced functor $ \fromto{\Piinfty(\X,R)}{\Piinfty(\X,P)} $ exhibits $ \Piinfty(\X,P) $ as the idempotent completion of the localization of $ \Piinfty(\X,R) $ at the class of morphisms that the functor $ \fromto{\Piinfty(\X,R)}{\Piinfty(\X,P)} $ carries to equivalences.
	Moreover, \enumref{cor:stability_under_pulling_back_to_locally_closed_subposets}{5} implies that the induced functor $ \fromto{\Piinfty(\X,P)}{P} $ is conservative.
	Hence \Cref{lem:composite_of_a_localization_with_a_conservative_functor} shows that the induced functor
	\begin{equation*}
		\fromto{\Piinfty(\X,R)}{\Piinfty(\X,P)}
	\end{equation*}
	exhibits $ \Piinfty(\X,P) $ as the idempotent completion of the localization $ \Piinfty(\X,R)[W_P\inv] $. 
	\enumref{cor:stability_under_pulling_back_to_locally_closed_subposets}{5} shows that the natural functor $ \fromto{\Piinfty(\X,R)}{R} $ is conservative.
	Thus \Cref{prop:ff_and_localization} shows that $ \Piinfty(\X,R)[W_P\inv] $ is already idempotent complete, concluding the proof.
\end{proof}

\begin{notation}\label{ntn:enveloping_groupoids}
	Write $ \Env \colon \fromto{\Catinfty}{\Spc} $ for the left adjoint to the inclusion $ \Spc \subset \Catinfty $.
	For \acategory $ \Ccal $, we can compute $ \Env(\Ccal) $ as the localization $ \Ccal[\Ccal\inv] $ at all morphisms in $ \Ccal $ \cite[Corollary 2.10]{arXiv:2108.01924}.
\end{notation}

\begin{corollary}
	Let $ (\X,P) $ be an exodromic stratified \topos.
	Then there is a natural equivalence
	\begin{equation*}
		\equivto{\Env(\Piinfty(\X,P))}{\Piinfty(\X)} \period
	\end{equation*} 
\end{corollary}

\begin{proof}
	Apply \Cref{thm:stability_under_coarsening} to the map of posets $ \fromto{P}{\pt} $.
\end{proof}


\subsection{Checking exodromy locally}\label{subsec:checking_exodromy_locally}

We now observe that the existence of an exit-path \category can be checked by descent.
This generalizes \cites[Proposition 3.6-(2)]{arXiv:2108.01924}[Proposition 3.13-(2)]{arXiv:2308.09550} to the setting of stratified \topoi.
We first recall two fundamental facts about \topoi.

\begin{recollection}\label{rec:limits_of_topoi}
	\hfill
	\begin{enumerate}
		\item\label{rec:limits_of_topoi.1} The \category $ \LTop $ has all limits and colimits. 
		Moreover, the forgetful functor $ \fromto{\LTop}{\CATinfty} $ preserves limits.
		See \cite[\HTTthm{Proposition}{6.3.2.3} \& \HTTthm{Corollary}{6.3.4.7}]{HTT}.

		\item\label{rec:limits_of_topoi.2} A colimit in \acategory $ \X $ with pullbacks is \defn{van Kampen} if the functor
		\begin{equation*}
			\fromto{\X^{\op}}{\CATinfty} \comma \quad \goesto{U}{\X_{/U}} 
		\end{equation*}
		transforms it into a limit in $ \CATinfty $.
		A presentable \category $ \X $ is \atopos if and only if all colimits in $ \X $ are van Kampen; see \cites[\HTTthm{Proposition}{5.5.3.13}, \href{http://www.math.ias.edu/~lurie/papers/HTT.pdf\#theorem.6.1.3.9}{Theorem 6.1.3.9(3)}, \& \HTTthm{Proposition}{6.3.2.3}]{HTT}{MR3935451}.
	\end{enumerate}
\end{recollection}

\begin{proposition}[(van Kampen)]\label{prop:van_Kampen}
	Let $ A $ be \acategory and let $ (\X_{\bullet},P_{\bullet}) \colon \fromto{A}{\StrTop} $ be a diagram of stratified \topoi.
	Let $ (\X_{\infty},P_{\infty}) $ be a cone under $ (\X_{\bullet},P_{\bullet}) $.
	Assume that the following conditions are satisfied:
	\begin{enumerate}
		\item For each $ \alpha \in A $, the stratified \topos $ (\X_{\alpha},P_{\alpha}) $ is exodromic.

		\item The natural pullback functors
		\begin{equation*}
			\X_{\infty} \to \lim_{\alpha \in A^{\op}} \X_{\alpha} \andeq \Cons_{P_{\infty}}(\X_{\infty}) \to \lim_{\alpha \in A^{\op}} \Cons_{P_{\alpha}}(\X_{\alpha})
		\end{equation*}
		are equivalences.
	\end{enumerate}
	Then the stratified \topos $ (\X_{\infty},P_{\infty}) $ is exodromic and the natural functor
	\begin{equation*}
		\colim_{\alpha \in A} \Piinfty(\X_{\alpha},P_{\alpha}) \to \Piinfty(\X_{\infty},P_{\infty})
	\end{equation*}
	is an equivalence of \categories.
	Here the colimit is formed in $ \Catidem $.
\end{proposition}

\begin{proof}
	Immediate from the definitions and the equivalence $ \PrLat \equivalent \Catidem $ of \Cref{recollection:atomically_generated}.
\end{proof} 

\begin{remark}[(on idempotent completion)]
	Let $ P $ be a poset and write \smash{$ \CatPcons \subset \CatP $} for the full subcategory spanned by those objects such that the specified functor $ \fromto{\Ccal}{P} $ is conservative.
	The forgetful functor
	\begin{equation*}
		\fromto{\CatP}{\Catinfty}
	\end{equation*}
	preserves colimits. 
	The inclusion $ \incto{\CatPcons}{\CatP} $ preserves colimits (\Cref{obs:iota_P}).
	Hence, the forgetful functor
	\begin{equation*}
		\fromto{\CatPcons}{\Catinfty}
	\end{equation*}
 	preserves colimits.
	By \Cref{lem:layered_implies_idempotent_complete}, every object of $ \CatPcons $ is idempotent complete.
	Hence in \Cref{prop:van_Kampen}, if the diagram of stratifying posets is constant, then the colimit in $ \Catinfty $ is already idempotent complete. 
\end{remark}

An important application of \Cref{prop:van_Kampen} is for diagrams obtained by slicing over a diagram in $ \X $ whose colimit is the terminal object.
To explain this, we first explain why local constancy and constructibility can, unsurprisingly, be checked locally.

\begin{lemma}\label{lem:local_constancy_is_local}
	Let $ \X $ be \atopos and let $ U_{\bullet} \colon \fromto{A}{\X} $ be a diagram with $ \colim_{\alpha \in A} U_{\alpha} \equivalent 1_{\X} $.
	Then an object $ L \in \X $ is locally constant if and only if for each $ \alpha \in A $, the restriction $ L \cross U_{\alpha} \in \X_{/U_{\alpha}} $ is locally constant.
\end{lemma}

\begin{proof}
	Since locally constant objects are stable under pullback, the forward implication is clear.
	For the reverse implication, for each $ \alpha \in A $, write $ f_{\alpha,\ast} \colon \X_{/U_{\alpha}} \to \X $ for the induced geometric morphism, so that $ \fupperstar_{\alpha} = (-) \cross U_{\alpha} $.
	Note that the induced map $ \coprod_{\alpha \in A} U_{\alpha} \to 1_{\X} $ is an effective epimorphism.
	By assumption, for each $ \alpha \in A $, the object $ \fupperstar_{\alpha}(L) $ is locally constant; choose objects $ (V_{\beta})_{\beta \in B_{\alpha}} $ of $ \X_{/U_{\alpha}} $ such that the induced map $ \coprod_{\beta \in B_{\alpha}} V_{\beta} \to 1_{\X_{/U_{\alpha}}} $ is an effective epimorphism and the further restriction of $ \fupperstar_{\alpha}(L) $ to each $ \X_{/V_{\beta}} $ is constant. 
	Since the forgetful functor $ f_{\alpha,\sharp} \colon \Xcal_{/U_{\alpha}} \to \X $ preserves colimits and pullbacks, it preserves effective epimorphisms.
	Hence the induced map $ \coprod_{\beta \in B_{\alpha}} V_{\beta} \to U_{\alpha} $ is an effective epimorphism in $ \X $.
	Consequently, the induced map
	\begin{equation*}
		\coprod_{\alpha \in A} \coprod_{\beta \in B_{\alpha}} V_{\beta} \to 1_{\X}
	\end{equation*}
	is an effective epimorphism.
	By construction, the restriction of $ L $ to each $ \X_{/V_{\beta}} $ is constant, so $ L $ is locally constant, as desired.
\end{proof}

\begin{lemma}\label{lem:constructibility_is_local}
	Let $ (\X,P) $ be a stratified \topos and let $ U_{\bullet} \colon \fromto{A}{\X} $ be a diagram with $ \colim_{\alpha \in A} U_{\alpha} \equivalent 1_{\X} $.
	Then:
	\begin{enumerate}
		\item An object $ F \in \X $ is $ P $-constructible if and only if for each $ \alpha \in A $, the restriction $ F \cross U_{\alpha} \in \X_{/U_{\alpha}} $ is $ P $-constructible (with respect to the induced stratification).

		\item The natural functor $ \ConsP(\X) \to \lim_{\alpha \in A^{\op}} \ConsP(\X_{/U_{\alpha}}) $ is an equivalence.
	\end{enumerate}
\end{lemma}

\begin{proof}
	For (1), note that since constructibility is preserved by pullback, the forward implication is clear.
	For the converse, for each $ \alpha \in A $, write $ f_{\alpha,\ast} \colon \X_{/U_{\alpha}} \to \X $ for the induced geometric morphism.
	Let $ p \in P $; we need to show that $ \iupperstar_p(F) $ is locally constant.
	Note that since $ \iupperstar_p $ is preserves colimits, $ \colim_{\alpha \in A} \iupperstar_p(U_{\alpha}) \equivalent 1_{\X_{p}} $.
	By \Cref{lem:local_constancy_is_local}, to check that $ \iupperstar_p(F) $ is a locally constant object of $ \X_{p} $, it suffices to check that the further restriction to each $ (\X_{p})_{/\iupperstar_p(U_{\alpha})} $ is locally constant.

	To see this, observe that by \enumref{prop:formula_for_pullbacks_along_etale_morphisms_and_closed_immersions}{1}, we have pullback squares
	\begin{equation*}
		\begin{tikzcd}[row sep=2.25em, column sep=3.5em]
			(\X_{p})_{/\iupperstar_p(U_{\alpha})} \arrow[dr, phantom, very near start, "\lrcorner", xshift=-1em, yshift=0.25em] \arrow[d, "f_{p,\alpha,*}"'] \arrow[r, "i_{p,\alpha,*}", hooked] & \X_{/U_{\alpha}}  \arrow[d, "f_{\alpha,*}"] \\ 
			\X_{p} \arrow[dr, phantom, very near start, "\lrcorner", xshift=-0.5em] \arrow[d] \arrow[r, "i_{p,*}"{description}, hooked] & \X  \arrow[d, "\slowerstar"] \\ 
			\Fun(\{p\} ,\Spc) \arrow[r, hooked] & \Fun(P,\Spc) \period
		\end{tikzcd}
	\end{equation*}
	in $ \RTop $.
	In particular, the $ p $-th stratum of the induced stratification on $ \X_{/U_{\alpha}} $ is $ (\X_{p})_{/\iupperstar_p(U_{\alpha})} $.
	By assumption, $ \iupperstar_{\alpha,p} \fupperstar_{\alpha}(F) $ is a locally constant object of $ (\X_{p})_{/\iupperstar_p(U_{\alpha})} $.
	By the commutativity of the top square, we deduce that the restriction of $ \iupperstar_p(F) $ to $ (\X_{p})_{/\iupperstar(U_{\alpha})} $ is locally constant, as desired.

	For (2), first note that since colimits in \atopos are van Kampen (\enumref{rec:limits_of_topoi}{2}), the natural functor $ \Xcal \to \lim_{\alpha \in A^{\op}} \X_{/U_{\alpha}} $ is an equivalence.
	Since the assignment $ \alpha \mapsto \ConsP(\X_{/U_{\alpha}}) $ is a diagram of full subcategories of $ \alpha \mapsto \X_{/U_{\alpha}} $, we deduce that $ \lim_{\alpha \in A^{\op}} \ConsP(\X_{/U_{\alpha}}) $ identifies as the full subcategory of $ \Xcal $ spanned by those objects $ F \in \X $ such that for each $ \alpha \in A $, we have $ \fupperstar_{\alpha}(F) \in \ConsP(\X_{/U_{\alpha}}) $ (see \cite[Proposition 2.1]{arXiv:2503.03916}).
	By (1), this full subcategory coincides with $ \ConsP(\X) $.
\end{proof}

\begin{corollary}\label{cor:van_Kampen_internally}
	Let $ (\X,P) $ be a stratified \topos and let $ U_{\bullet} \colon \fromto{A}{\X} $ be a diagram with $ \colim_{\alpha \in A} U_{\alpha} \equivalent 1_{\X} $.
	If for each $ \alpha \in A $, the stratified \topos $ (\X_{/U_{\alpha}},P) $ is exodromic, then the stratified \topos $ (\X,P) $ is exodromic and the natural functor
	\begin{equation*}
		\colim_{\alpha \in A} \Piinfty(\X_{/U_{\alpha}},P) \to \Piinfty(\X,P)
	\end{equation*}
	is an equivalence of \categories.
\end{corollary}

\begin{proof}
	Immediate from \Cref{prop:van_Kampen,lem:constructibility_is_local}.
\end{proof}


 \subsection{The Künneth formula}\label{subsec:Kunneth_formula}

 We now prove a \textit{Künneth formula} for the exit-path \category of the product of exodromic stratified \topoi.
 For this subsection, it may be useful to review \Cref{rec:products_of_topoi} on products of \topoi and tensor products of presentable \categories.
 One key input is the Künneth formula in the unstratified setting (\Cref{prop:Kunneth_formula_for_monodromic_topoi}).

 We start by noting that the product of stratified \topoi is naturally stratified:

 \begin{definition}[(stratification of a product)]\label{def:stratification_of_a_product}
 	Let $ \slowerstar \colon \fromto{\X}{\Fun(P,\Spc)} $ and $ \tlowerstar \colon \fromto{\Y}{\Fun(Q,\Spc)} $ be stratified \topoi.
 	We write $ (\X \tensor \Y,P \cross Q) $ for the stratified \topos
 	\begin{equation*}
 		\slowerstar \tensor \tlowerstar \colon \fromto{\X \tensor \Y}{\Fun(P,\Spc) \tensor \Fun(Q,\Spc) \equivalent \Fun(P \cross Q,\Spc)} \period
 	\end{equation*}
 \end{definition}

 \begin{observation}\label{obs:stratification_of_the_product_of_exodromic_topoi}
 	In the setting of \Cref{def:stratification_of_a_product}, assume that $ (\X,P) $ and $ (\Y,Q) $ are exodromic stratified \topoi.
 	Then:
 	\begin{enumerate}
 		\item Since $ \supperstar $ and $ \tupperstar $ preserve limits and colimits,
 		\begin{equation*}
 			\supperstar \tensor \tupperstar \colon \fromto{\Fun(P \cross Q,\Spc)}{\X \tensor \Y}
 		\end{equation*}
 		preserves limits and colimits.

 		\item Since the inclusions $ \incto{\ConsP(\X)}{\X} $ and $ \incto{\ConsQ(\Y)}{\Y} $ are both left and right adjoints, the induced functor
 		\begin{equation*}
 			\fromto{\ConsP(\X) \tensor \ConsQ(\Y)}{\X \tensor \Y} 
 		\end{equation*}
 		is fully faithful and both a left and right adjoint.
 	\end{enumerate}
 \end{observation}

\begin{lemma}
	Let $ (\X,P) $ and $ (\Y,Q) $ be exodromic stratified \topoi. 
	The inclusion
 	\begin{equation*}
 		\incto{\ConsP(\X) \tensor \ConsQ(\Y)}{\X \tensor \Y}
 	\end{equation*}
 	factors through $ \Cons_{P \cross Q}(\X \tensor \Y) $.
\end{lemma}

\begin{proof}
	Let $ (p,q) \in P \cross Q $.
 	Note that by the definition of $ \ConsP(\X) \tensor \ConsQ(\Y) $, the composite 
 	\begin{equation}\label{eq:i_p_tensor_i_q}
 		\begin{tikzcd}[sep=3em]
 			\ConsP(\X) \tensor \ConsQ(\Y) \arrow[r, hooked] & \X \tensor \Y \arrow[r, "\iupperstar_p \tensor \iupperstar_q"] & \X_p \tensor \Y_q
 		\end{tikzcd}
 	\end{equation}
 	factors through $ \LC(\X_p) \tensor \LC(\Y_q) $.
 	By \Cref{prop:Kunneth_formula_for_monodromic_topoi}, we have
 	\begin{equation*}
 		\LC(\X_p) \tensor \LC(\Y_q) = \LC(\X_p \tensor \Y_q)
 	\end{equation*}
 	as full subcategories of $ \X_p \tensor \Y_q $.
 	Hence the functor \eqref{eq:i_p_tensor_i_q} factors through $ \LC(\X_p \tensor \Y_q) $, as desired.
\end{proof}

\begin{proposition}[{(Künneth formula for exodromic stratified \topoi)}]\label{prop:Kunneth_formula_for_exodromic_topoi}
	Let $ \slowerstar \colon \fromto{\X}{\Fun(P,\Spc)} $ and $ \tlowerstar \colon \fromto{\Y}{\Fun(Q,\Spc)} $ be exodromic stratified \topoi.
	If $ P $ and $ Q $ are noetherian, then:
	\begin{enumerate}
		\item The natural fully faithful functor
		\begin{equation*}
			\incto{\ConsP(\X) \tensor \ConsQ(\Y)}{\Cons_{P \cross Q}(\X \tensor \Y)}
		\end{equation*}
		is an equivalence.

		\item The stratified \topos $ (\X \tensor \Y, P \cross Q) $ is exodromic and the natural functor
		\begin{equation*}
			\fromto{\Piinfty(\X \tensor \Y, P \cross Q)}{\Piinfty(\X,P) \cross \Piinfty(\Y,Q)}
		\end{equation*}
		is an equivalence of \categories.
	\end{enumerate}
\end{proposition}

\begin{proof}
 	We now proceed by noetherian induction.
 	First, let us prove that when $Q = \ast$, the functor we just constructed
 	\[ 
 		\boxtensor \colon \ConsP(\X) \tensor \LC(\Y) \to \ConsP(\X \tensor \Y) 
 	\]
 	is an equivalence.
 	When $P = \ast$, the conclusion follows from \enumref{prop:Kunneth_formula_for_monodromic_topoi}{3}.
 	Otherwise, notice that \cref{lem:commuting_limits_past_tensors} implies that the question is local on $P$.
 	We can therefore reduce ourselves to prove that $\boxtensor$ is an equivalence for posets of the form $P_{\geq p}$.
 	In this case, consider the following diagram:
 	\begin{equation*}
 		\begin{tikzcd}[sep=3em]
			\LC(\X_p) \tensor \LC(\Y) \arrow[d, "\boxtensor"'] & \ConsP(\X) \tensor \LC(\Y) \arrow[d, "\boxtensor"] \arrow[r] \arrow[l] & \Cons_{P_{>p}}(\X_{>p}) \tensor \LC(\Y) \arrow[d, "\boxtensor"] \\
			\LC(\X_p \tensor \Y) & \ConsP(\X \tensor \Y) \arrow[l] \arrow[r] & \Cons_{P_{>p}}(\X_{>p} \tensor \Y) \period
		\end{tikzcd}
 	\end{equation*}
 	Since $\Y$ is monodromic, $\LC(\Y)$ is compactly generated and therefore the top row is a recollement.
 	By \enumref{lem:recollements}{4}, the bottom line is also a recollement.
 	The inductive hypothesis guarantees that the outer vertical functors are equivalences.
 	Therefore, \enumref{lem:comparison_of_recollements_conservativity}{4} implies that the same goes for the middle one.
 	We now repeat the same argument proceeding by noetherian induction on the length of $Q$ and for arbitrary $P$.
 	Reasoning as above, we reduce ourselves to consider the following diagram:
 	\begin{equation*}
	 	\begin{tikzcd}[sep=3em]
	 		\ConsP(\X) \tensor \LC(\Y) \arrow[d, "\boxtensor"'] & \ConsP(\X) \tensor \ConsQ(\Y) \arrow[d, "\boxtensor"] \arrow[r] \arrow[l] & \ConsP(\X) \tensor \Cons_{Q_{>q}}(\Y_{>q}) \arrow[d, "\boxtensor"] \\
	 		\Cons_{P \times \{q\}}(\X \tensor \Y_q) & \Cons_{P \cross Q}(\X \tensor \Y) \arrow[l] \arrow[r] & \Cons_{P \cross Q_{>q}}(\X \tensor \Y_{>q}) \period
	 	\end{tikzcd} 
 	\end{equation*}
	Once again, since $(\X,P)$ is exodromic, $\ConsP(\X)$ is compactly generated and therefore the top row is a recollement.
	The same goes for the bottom row.
	Thus, the conclusion follows from the previous step, the inductive hypothesis and \enumref{lem:comparison_of_recollements_conservativity}{4}.

 	For (2), note that by \Cref{obs:stratification_of_the_product_of_exodromic_topoi}, the pullback functor $ \supperstar \tensor \tupperstar $ preserves limits and colimits.
 	Moreover, by (1), $ \Cons_{P \cross Q}(\X \tensor \Y) $ is atomically generated and closed under limits and colimits in $ \X \tensor \Y $.
 	Hence, $ (\X \tensor \Y, P \cross Q) $ is exodromic.
 	Finally, the equivalence
 	\begin{equation*}
 		\equivto{\ConsP(\X) \tensor \ConsQ(\Y)}{\Cons_{P \cross Q}(\X \tensor \Y)}
 	\end{equation*}
 	shows that 
 	\begin{equation*}
 		\equivto{\Piinfty(\X \tensor \Y, P \cross Q)}{\Piinfty(\X,P) \cross \Piinfty(\Y,Q)} \period \qedhere
 	\end{equation*}
 \end{proof}


\subsection{Stability properties of categorical finiteness \& compactness}\label{subsec:stability_properties_of_categorical_finiteness_and_compactness}

As we'll explain in \cite{HPT-derived_moduli_of_perverse_sheaves}, the compactness of exit-path \categories can be used to prove that moduli stacks of constructible and perverse sheaves are locally geometric.
Hence knowing when a stratified \topos has compact exit-path \category is of great utility.
To complete this section, we explain why the classes of exodromic stratified \topoi with finite or compact exit-path \category are stable under coarsening.
In \cref{sec:applications_and_examples}, we use the results of this subsection to extend the representability results from \cite[\S7]{arXiv:2211.05004} beyond the conical situation.

Recall from \cite[Definition 2.2.1]{PortaTeyssier} the following:

\begin{definition}\label{def:categorical_compactness_and_finiteness}
	Let $ (\X,P) $ be an exodromic stratified \topos.
	We say that $ (\X,P) $ is:
	\begin{enumerate}
		\item \defn{Categorically finite} if $ \Piinfty(\X,P) $ is a finite object of $ \Catinfty $. 
		(See \Cref{recollection_finite_category}.)

		\item \defn{Categorically compact} if $ \Piinfty(\X,P) $ is a compact object of $ \Catinfty $.
	\end{enumerate}
\end{definition}

\begin{lemma}\label{lem:pulling_back_to_a_locally_closed_subposet_preserves_categorical_finiteness_and_compactness}
	Let $ (\X,P) $ be an exodromic stratified \topos and $ S \subset P $ a locally closed subposet.
	If $ (\X,P) $ is categorically finite (resp., compact), then $ (\X_S,S) $ is categorically finite (resp., compact).
\end{lemma}

\begin{proof}
	This is a special case of \Cref{prop:base_change_lc_preserves_finite_and_compact_categories}.
\end{proof}

\begin{lemma}\label{lem:categorical_finiteness_and_compactness_can_be_detected_on_finite_covers}
	Let $ (\X,P) $ be a stratified \topos and let $ U_1, \ldots, U_n \in \X $ be a finite set of objects such that the induced map $ \fromto{U_1 \coproduct \cdots \coproduct U_n}{1_{\X}} $ is an effective epimorphism.
	Assume that for all $ 1 \leq i_1 < \cdots < i_k \leq n $, the stratified \topos \smash{$ (\X_{/U_{i_1} \cross \cdots \cross U_{i_k}},P) $} is exodromic and is categorically finite (resp., compact).
	Then $ (\X,P) $ is exodromic and is categorically finite (resp., compact).
\end{lemma}

\begin{proof}
	Immediate from \Cref{cor:van_Kampen_internally} and the fact that both finite and compact \categories are closed under finite colimits in $ \Catinfty $.
\end{proof}

\begin{proposition}\label{prop:categorical_finiteness_and_compactness_are_stable_under_coarsening}
	Let $ (\X,R) $ be an exodromic stratified \topos and let $ \phi \colon \fromto{R}{P} $ be a map of posets.
	If $ (\X,R) $ is categorically finite (resp., compact), then $ (\X,P) $ is categorically finite (resp., compact).
\end{proposition}

\begin{proof}
	The fact that $ (\X,P) $ is exodromic follows from the stability of the class of exodromic stratified \topoi under coarsening (\enumref{thm:stability_under_coarsening}{1}).
	By \enumref{thm:stability_under_coarsening}{2}, there is an equivalence
	\begin{equation*}
		\Piinfty(\X,P) \equivalent \Piinfty(\X,R)[W_P\inv] \period
	\end{equation*}
	Since $ \Piinfty(\X,R) $ is a finite (resp., compact), the claim now follows from \Cref{prop:inverting_morphisms_fiberwise_preserves_compactness}.
\end{proof}


\section{Exodromy with coefficients}\label{sec:exodromy_with_coefficients}

This section concerns exodromy with coefficients in \categories other than the \category of spaces.
In \cref{subsec:exodromy_with_coefficients_in_a_presentable_category}, we explain when the exodromy equivalence holds for sheaves with coefficients in more general presentable \categories.
In particular, exodromy with coefficients in $ \Spc $ implies exodromy with coefficients in any compactly assembled \category; see \Cref{lem:exit_path_category_for_compactly_assembled_coefficients_is_automatic}.
\Cref{subsec:exodromy_with_coefficients_in_PrL} treats exodromy with coefficients in the \category \smash{$ \PrL $} of presentable \categories; these results are needed in forthcoming work of the second- and third-named authors \cites{arXiv:2401.12335}{arXiv:2504.05360}.


\subsection{Exodromy with coefficients in a presentable \texorpdfstring{$\infty$}{∞}-category}\label{subsec:exodromy_with_coefficients_in_a_presentable_category}

We are also interested in when the exit-path \category corepresents constructible objects with coefficients in a presentable \category $ \Ecal $.
The following slight generalization of the discussion in \cite[\S6.1]{arXiv:2211.05004} captures this more general situation.

\begin{observation}\label{obs:closure_of_E-exodromic_under_retracts}
	Let $ (\X,P) $ be an exodromic stratified \topos and let $ \Ecal $ be a presentable \category.
	Since the \category \smash{$ \ConsP(\X) $} is presentable and the inclusion
	\begin{equation*}
		\incto{\ConsP(\X)}{\X}
	\end{equation*}
	is both a left and a right adjoint, tensoring with $ \Ecal $ gives a fully faithful functor
	\begin{equation*}\label{eq:inclusion_of_constructible_sheaves_tensor_E}
		\boxtensor \colon \incto{\ConsP(\X) \tensor \Ecal}{\Sh(\X;\Ecal)}
	\end{equation*}
	that is both a left and a right adjoint.
\end{observation}

\begin{lemma}\label{lem:boxtensor_for_constructible_sheaves}
	Let $ \Ecal $ be a presentable \category, and let $ (\X,P) $ be an exodromic stratified \topos.
	Then the functor
	\begin{equation*}
		\begin{tikzcd}
			\boxtensor \colon \incto{\ConsP(\X) \tensor \Ecal}{\Sh(\X;\Ecal)}
		\end{tikzcd}
	\end{equation*}
	factors through $ \ConsP(\X;\Ecal) \subset \Sh(\X;\Ecal) $.
\end{lemma}

\begin{proof}	
	The functoriality of the tensor product in \smash{$ \PrL $} implies that for each $ p \in P $, there is a commutative square
	\begin{equation*}
		\begin{tikzcd}[sep=3.5em]
			\ConsP(\X) \tensor \Ecal \arrow[r, hooked] \arrow[d, "\iupperstar_p \tensor \id{\Ecal}"'] & \Sh(\X) \tensor \Ecal \arrow[d, "\iupperstar_p \tensor \id{\Ecal}"] \\ 
			\LC(\X_p) \tensor \Ecal \arrow[r, hooked] & \Sh(\X_p) \tensor \Ecal \period
		\end{tikzcd}
	\end{equation*}
	Since the strata of $ (\X,P) $ are monodromic (\enumref{cor:stability_under_pulling_back_to_locally_closed_subposets}{2}), the natural functor
	\begin{equation*}
		\fromto{\LC(\X_p) \tensor \Ecal}{\LC(\X_p;\Ecal)}
	\end{equation*}
	is an equivalence (\Cref{rec:monodromy}).
	The claim is now immediate.
\end{proof}

\begin{nul}
	In the setting of \Cref{lem:boxtensor_for_constructible_sheaves}, we have a commutative triangle
	\begin{equation}\label{eq:box_product}
		\begin{tikzcd}[column sep=3.5em, row sep=2em]
			\ConsP(\X) \tensor \Ecal \arrow[r, hooked] \arrow[dr, hooked] & \ConsP(\X;\Ecal) \arrow[d, hooked] \\ 
			 & \Sh(\X;\Ecal) \period
		\end{tikzcd}
	\end{equation}
\end{nul}

\begin{definition}\label{def:exit-paths_for_coefficients}
	Let $ \Ecal $ be a presentable \category and let $ (\X,P) $ be a stratified \topos.
	We say that $ (\X,P) $ \defn{is $ \Ecal $-exodromic} if the following conditions are satisfied:
	\begin{enumerate}
		\item\label{def:exit-paths_for_coefficients.1} The stratified \topos $ (\X,P) $ is exodromic.

		\item\label{def:exit-paths_for_coefficients.2} The functor $ \boxtensor \colon \incto{\ConsP(\X) \tensor \Ecal}{\ConsP(\X;\Ecal)} $ is an equivalence.
	\end{enumerate}
\end{definition}

We collect some basic properties of $ \Ecal $-exodromic stratified \topoi.

\begin{observation}\label{obs:stability_of_E-exodromic_under_retracts}
	Let $ (\X,P) $ be an exodromic stratified \topos.
	Since equivalences of \categories are stable under retracts, the class of presentable \categories $ \Ecal $ for which $ (\X,P) $ is $ \Ecal $-exodromic is also stable under retracts.
\end{observation}

\begin{lemma}\label{lem:exodromy_with_coefficients}
	Let $ \Ecal $ be a presentable \category and let $ (\X,P) $ be a $ \Ecal $-exodromic stratified \topos.
	Then the equivalence
	\begin{align*}
		\boxtensor \colon \ConsP(\X) \tensor \Ecal &\equivalence \ConsP(\X;\Ecal) \\ 
		\intertext{induces a canonical equivalence}
		\Fun(\PiSigma(\X,P), \Ecal) &\equivalent \ConsP(\X;\Ecal) \period
	\end{align*}
\end{lemma}

\begin{proof}
	Indeed, we have the following canonical equivalences:
	\begin{align*}
		\ConsP(\X) \tensor \Ecal & \equivalent \Fun(\PiSigma(\X,P), \Spc) \tensor \Ecal && \\
		& \equivalent \Fun(\PiSigma(\X,P), \Ecal) \period && \text{\HA{Proposition}{4.8.1.17}} 
	\end{align*}
	The conclusion follows.
\end{proof}

We now prove an analogue of \cref{cor:stability_under_pulling_back_to_locally_closed_subposets}.
We first need the following lemma:

\begin{lemma}\label{lem:comparison_of_recollements_conservativity}
	Let $ \Xcal_1 $ and $ \Xcal_2 $ be \categories with finite limits and an inital object.
	Let 
	\begin{equation*}\label{diag:comparison_of_recollements}
		\begin{tikzcd}[column sep=4.5em, row sep=2.5em]
			\Zcal_1 \arrow[d, "F_{\Zcal}"'] & \Xcal_1 \arrow[d, "F"] \arrow[l, "\iupperstar_1"'] \arrow[r, "\jupperstar_1"] &  \Ucal_1 \arrow[d, "F_{\Ucal}"] \\ 
			\Zcal_2 & \Xcal_2 \arrow[l, "\iupperstar_2"] \arrow[r, "\jupperstar_2"'] & \Ucal_2
		\end{tikzcd}
	\end{equation*}
	be a commutative diagram where each of the horizontal rows exhibits $ \Xcal_i $ as the recollement of $ \Zcal_i $ and $ \Ucal_i $.
	\begin{enumerate}
		\item\label{lem:comparison_of_recollements_conservativity.1} If $ F $ is essentially surjective, then $ F_{\Zcal} $ and $ F_{\Ucal} $ are essentially surjective.
		
		\item\label{lem:comparison_of_recollements_conservativity.2} If $F_{\Zcal}$ preserves the initial object, then the right-hand square is horizontally left adjointable.
		In this case, if $ F $ is fully faithful (resp., an equivalence), then the same is true of $F_{\Ucal}$.
		
		\item\label{lem:comparison_of_recollements_conservativity.3} If $F_{\Ucal}$ preserves the terminal object, then the left-hand square is horizontally right adjointable.
		In this case, if $ F $ is fully faithful (resp., an equivalence), then the same is true of $F_{\Zcal}$.
		
		\item\label{lem:comparison_of_recollements_conservativity.4} Assume that $F$ is left exact. 
		If $F_{\Zcal}$ and $F_{\Ucal}$ are equivalences, then $ F $ is also an equivalence
	\end{enumerate}
\end{lemma}

\begin{proof}
	For (1), we prove that $ F_{\Ucal} $ is essentially surjective; the proof of the essential surjectivity of $ F_{\Zcal} $ is identical.
	Since $ F $ is essentially surjective, given $ u \in \Ucal_2 $ there exists $ x \in \Xcal_1 $ and an equivalence $ j_{2,*}(u) \equivalent F(x) $.
	Hence the full faithfulness of $ j_{2,*} $ and the commutativity of the right-hand square show that
	\begin{equation*} 
		u \equivalent j_2^\ast j_{2,*}(u) \equivalent j_2^\ast( F(x) ) \equivalent F_{\Ucal} (j_1^\ast(x)) \period
	\end{equation*}
	
	We now prove (2); item (3) follows by a dual argument.
	Consider the exchange transformation
	\begin{equation*} 
		\alpha \colon \fromto{j_{2,!} F_{\Ucal}}{F j_{1,!}} \period  
	\end{equation*}
	Since the bottom line is a recollement, to prove that $\alpha$ is an equivalence it suffices to check that $j_2^\ast(\alpha)$ and $i_2^\ast(\alpha)$ are equivalences.
	We first deal with the former.
	Since the right-hand square commutes, we have $ j_2^\ast F j_{1,!} \equivalent F_{\Ucal} j_1^\ast j_{1,!} $, so the conclusion follows from the full faithfulness of both $ j_{1,!} $ and $ j_{2,!} $.
	As for $ i_2^\ast(\alpha) $, recall that the theory of recollements shows that both $ i_2^\ast j_{2,!} $ and $ i_1^\ast j_{1,!} $ are constant with value the initial object.
	Also, since the left-hand square commutes, we have $ i_2^\ast F j_{1,!} \equivalent F_{\Zcal} i_1^\ast j_{1,!} $.
	Since $F_{\Zcal}$ preserves the initial object, it follows that both the source and target of $ i_2^\ast(\alpha) $ are constant with value the initial object; hence $ i_2^\ast(\alpha) $ is an equivalence.
	
	From the horizontal left adjointability of the right-hand square and the full faithfulness of $j_{1,!}$ and $j_{2,!}$, it immediately follows that if $ F $ is fully faithful, then $ F_{\Ucal} $ is also fully faithful.
	Finally, if $ F $ is an equivalence, then we have just seen that $ F_{\Ucal} $ is fully faithful and (1) shows that $ F_{\Ucal} $ is also essentially surjective.
	
	We are left to prove (4). 
	Since $F_{\Zcal}$ and $F_{\Ucal}$ are equivalences, they preserve both the initial and the terminal object.
	Then (4) follows from the above adjointability statements and \HAa{Proposition}{A.8.14}.
\end{proof}

\begin{proposition}\label{prop:E-exodromic_is_stable_under_pulling_back_to_locally_closed_subposets}
	Let $ (\X,P) $ be a stratified \topos and let $ \Ecal $ be a presentable \category.
	Let $S \subset P$ be a locally closed subposet.
	If $ (\X,P) $ is $ \Ecal $-exodromic and $ \Ecal $ is compatible with recollements (\Cref{def:compatibility_with_recollements}), then $(\X_S,S)$ is also $ \Ecal $-exodromic.
\end{proposition}

\begin{proof}
	It is enough to prove that if $U \subset P$ is an open subposet with closed complement $ Z $, then both $(\XU,U)$ and $(\XZ,Z)$ are $ \Ecal $-exodromic.
	First of all, we already know from \cref{cor:stability_under_pulling_back_to_locally_closed_subposets} that these stratified \topoi are exodromic.
	Consider now the following commutative diagram:
	\begin{equation*}
		\begin{tikzcd}[sep=3em]
			\ConsU(\XU) \tensor \Ecal \arrow[d, hooked, "\boxtensor_U"'] & \ConsP(\X) \tensor \Ecal \arrow[l, "\iupperstar_U \tensor \Ecal"'] \arrow[r, "\iupperstar_Z \tensor \Ecal"] \arrow[d, "\boxtensor"] & \ConsZ(\XZ) \tensor \Ecal \arrow[d, hooked, "\boxtensor_Z"] \\
			\ConsU(\XU;\Ecal) & \ConsP(\X;\Ecal) \arrow[l, "\iupperstar_U"] \arrow[r, "\iupperstar_Z"'] & \ConsZ(\XZ;\Ecal) \period
		\end{tikzcd}
	\end{equation*}
	Since $ (\X,P) $ is $ \Ecal $-exodromic, the middle vertical functor is an equivalence.
	Morever, because $ (\X,P) $ is exodromic, the functor
	\begin{equation*}
		\ConsP(\X) \tensor \Ecal \to \Sh(\X) \tensor \Ecal \equivalent \Sh(\X;\Ecal) 
	\end{equation*}
	preserves both limits and colimits.
	Combining \cref{lem:terminal_object_constructible} and \enumref{lem:recollements}{4}, we see that the bottom row exhibits $\ConsP(\X;\Ecal)$ as a recollement of $\ConsU(\XU;\Ecal)$ and $\ConsZ(\XZ;\Ecal)$.
	On the other hand, since $ \Ecal $ is compatible with recollements, the top row is a recollement as well.
	Clearly, $\boxtensor_U$ preserves the initial object.
	On the other hand, since $\boxtensor_Z$ is compatible with the inclusion into
	\begin{equation*}
		\Sh(\XZ) \tensor \Ecal \equivalent \Sh(\XZ;\Ecal)
	\end{equation*}
	and since the terminal object in $\Sh(\XZ;\Ecal)$ is $ Z $-constructible thanks to \cref{lem:terminal_object_constructible}, we conclude that $\boxtensor_Z$ preserves the terminal object as well.
	Thus, \cref{lem:comparison_of_recollements_conservativity} implies that $\boxtensor_U$ and $\boxtensor_Z$ are equivalences.
\end{proof}

To explain why $ \Ecal $-exodromicity can be checked locally, we need descent for the tensor decomposition 
\begin{equation*}
	\ConsP(\X) \tensor \Ecal \equivalent \ConsP(\X;\Ecal) \period
\end{equation*}
For this, we make use of the following lemma.

\begin{lemma}\label{lem:commuting_limits_past_tensors}
	Let $ A $ be a small \category and let $ \Ccal_{\bullet} \colon \fromto{A}{\CATinfty} $ be a diagram of \categories.
	Assume that for each $ \alpha \in A $, the \category $ \Ccal_{\alpha} $ is presentable and that for each morphism $ \fromto{\alpha}{\beta} $ in $ A $, the transition functor $ \fromto{\Ccal_{\alpha}}{\Ccal_{\beta}} $ is both a left and a right adjoint.
	Then:
	\begin{enumerate}
		\item\label{lem:commuting_limits_past_tensors.1} The limits of $ \Ccal_{\bullet} $ when computed in $ \PrR $, $ \PrL $, or $ \CATinfty $ all agree.

		\item\label{lem:commuting_limits_past_tensors.2} For any presentable \category $ \Ecal $, the natural morphism
		\begin{equation*}
			\fromto{\lim_{\alpha \in A} \Ecal \tensor \Ccal_{\alpha}}{ \Ecal \tensor \lim_{\alpha \in A} \Ccal_{\alpha}}
		\end{equation*}
		in $ \PrL $ is an equivalence.
		(Here, both limits are computed in $ \PrL $.)
	\end{enumerate}
\end{lemma}

\begin{proof}
	Item (1) follows from the fact that both of the forgetful functors $ \fromto{\PrL}{\CATinfty} $ and $ \fromto{\PrR}{\CATinfty} $ preserve limits \cite[\HTTthm{Proposition}{5.5.3.13} \& \HTTthm{Theorem}{5.5.3.18}]{HTT}.
	Item (2) follows from (1), the equivalence \smash{$ \PrR \equivalent (\PrL)^{\op} $}, and the fact that the functor
	\begin{equation*}
		\Ecal \tensor (-) \colon \fromto{\PrR}{\PrR}
	\end{equation*}
	preserves limits \HA{Remark}{4.8.1.24}.
\end{proof}

\begin{proposition}\label{prop:E_exodromic_Van_Kampen}
	Let $ \Ecal $ be a presentable \category, let $ A $ be \acategory, and let
	\begin{equation*}
		(\X_{\bullet},P_{\bullet}) \colon \fromto{A}{\StrTop}
	\end{equation*}
	be a diagram of stratified \topoi.
	Let $ (\X_{\infty},P_{\infty}) $ be a cone under $ (\X_{\bullet},P_{\bullet}) $.
	Assume that the following conditions are satisfied:
	\begin{enumerate}
		\item For each $ \alpha \in A $, the stratified \topos $ (\X_{\alpha},P_{\alpha}) $ is $ \Ecal $-exodromic.
		
		\item The natural pullback functors
		\begin{equation*}
		 \X_\infty \to \lim_{\alpha \in A^{\op}} \X_\alpha \andeq \Cons_{P_\infty}(\X_\infty) \to \lim_{\alpha \in A^{\op}} \Cons_{P_\alpha}(\X_\alpha) 
		\end{equation*}
		as well as
		\begin{equation*}
			\Cons_{P_{\infty}}(\X_{\infty};\Ecal) \to \lim_{\alpha \in A^{\op}} \Cons_{P_{\alpha}}(\X_{\alpha};\Ecal)
		\end{equation*}
		are equivalences.
	\end{enumerate}
	Then the stratified \topos $ (\X_{\infty},P_{\infty}) $ is $ \Ecal $-exodromic.
\end{proposition}

\begin{proof}
	\Cref{prop:van_Kampen} implies that $ (\X,P) $ is exodromic.
	Consider the following commutative square
	\begin{equation*}
	 \begin{tikzcd}
		\Cons_{P_\infty}(X_\infty) \tensor \Ecal \arrow[r] \arrow[d] & \lim_{\alpha \in A^{\op}} \Cons_{P_\alpha}(X_\alpha) \tensor \Ecal \arrow[d] \\
		\Cons_{P_\infty}(X_\infty; \Ecal) \arrow[r] & \lim_{\alpha \in A^{\op}} \Cons_{P_\alpha}(X_\alpha;\Ecal)
	\end{tikzcd} 
\end{equation*}
	Since each $(\X_\alpha, P_\alpha)$ is $ \Ecal $-exodromic, the left vertical functor is an equivalence.
	Also, by assumption, the bottom horizontal functor is an equivalence.
	Thus it suffices to show that the top horizontal functor is an equivalence.
	By \cref{lem:commuting_limits_past_tensors}, it suffices to show that for every morphism $\alpha \to \beta$ in $A^{\op}$, the pullback functor
	\begin{equation*}
		\Cons_{P_\alpha}(\X_\alpha) \to \Cons_{P_\beta}(\X_\beta) 
	\end{equation*}
	is both a left and a right adjoint.
	By assumption $(\X_\beta, P_\beta) $ and $ (\X_\alpha, P_\alpha)$ are exodromic, so this is an immediate consequence of \Cref{thm:all_morphisms_are_exodromic}.
\end{proof}

\begin{corollary}\label{cor:E_exodromic_Van_Kampen}
	Let $ (\X,P) $ be a stratified \topos and let $ \Ecal $ be a presentable \category.
	Let $ U_{\bullet} \colon \fromto{A}{\X} $ be a diagram with $ \colim_{\alpha \in A} U_{\alpha} \equivalent 1_{\X} $.
	If for each $ \alpha \in A $, the stratified \topos $ (\X_{/U_{\alpha}},P) $ is $ \Ecal $-exodromic, then the stratified \topos $ (\X,P) $ is also $ \Ecal $-exodromic.
\end{corollary}

\begin{proof}
	By \Cref{rec:limits_of_topoi,prop:E_exodromic_Van_Kampen}, it suffices to show that the natural pullback functor
	\begin{equation}\label{eq:E_exodromic_Van_Kampen}
		\Cons_{P}(\X;\Ecal) \to \lim_{\alpha \in A^{\op}} \Cons_{P}(\X_{/U_{\alpha}};\Ecal)
	\end{equation}
	is an equivalence.
	Notice that for every map $\alpha \to \beta$ in $A^{\op}$, the induced pullback functor
	\begin{equation*}
		\X_{/U_\alpha} \to \X_{/U_\beta} 
	\end{equation*}
	is both a left and a right adjoint.
	Therefore, \cref{lem:commuting_limits_past_tensors} implies that the pullback functor
	\begin{equation*}
		\Sh(\X;\Ecal) \to \lim_{\alpha \in A^{\op}} \Sh(\X_{/U_{\alpha}}; \Ecal) 
	\end{equation*}
	is an equivalence.
	This immediately implies that \eqref{eq:E_exodromic_Van_Kampen} is fully faithful.
	To conclude, it is enough to observe that $F \in \Sh(\X;\Ecal)$ is $P$-constructible if and only if for every $\alpha \in A$, the restriction of $ F $ to $\X_{/U_\alpha}$ is $P$-constructible.
\end{proof}

\begin{recollection}[(compactly assembled \categories)]\label{rec:compactly_assembled}
	A presentable \category $ \Ecal $ is \textit{compactly assembled} if $ \Ecal $ is a retract in \smash{$ \PrL $} of a compactly generated \category \cite[\SAGthm{Definition}{21.1.2.1} \& \SAGthm{Theorem}{21.1.2.18}]{SAG}.
	If $ \Ecal $ is a presentable stable \category, then $ \Ecal $ is compactly assembled if and only if $ \Ecal $ is dualizable in the symmetric monoidal \category of presentable stable \categories and left adjoints equipped with the Lurie tensor product \SAG{Proposition}{D.7.3.1}.
\end{recollection}

\begin{corollary}\label{lem:exit_path_category_for_compactly_assembled_coefficients_is_automatic}
	Let $ (\X,P) $ be a exodromic stratified \topos and let $ \Ecal $ be a presentable \category.
	Then:
	\begin{enumerate}
		\item If $ \Ecal $ is compactly assembled, then $ (\X,P) $ is $ \Ecal $-exodromic.
		
		\item If $ \Ecal $ is stable and $P$ is noetherian, then $ (\X,P) $ is $ \Ecal $-exodromic.
	\end{enumerate}
\end{corollary}

\begin{proof}
	For (1), note that by \Cref{obs:closure_of_E-exodromic_under_retracts}, it suffices to prove the claim in the case that $ \Ecal $ is compactly generated.
	In this case, the proof of \cite[Theorem B.9]{arXiv:2012.10777} works \textit{verbatim}.
	
	We now prove (2).
	For $p \in P$, we write $\X_{\geq p}$ for $\X_{P_{\geq p}}$.
	Since the sets $\{P_{\geq p}\}_{p\in P}$ form an open cover of $P$, by \cref{cor:E_exodromic_Van_Kampen} it suffices to show that for every $p \in P$ the stratified \topos $(\X_{\geq p}, P_{\geq p})$ is $ \Ecal $-exodromic.
	We prove this statement by noetherian induction.
	When $P$ is  a single element, the conclusion follows from \cref{rec:monodromy}.
	We are then reduced to showing that if for every $ q > p $ the stratified \topos $(\X_{\geq q},P_{\geq q})$ is $ \Ecal $-exodromic, then $(X_{\geq p}, P_{\geq p})$ is also $ \Ecal $-exodromic.
	Note that
	\begin{equation*}
		P_{\geq p} \sminus \{p\} = P_{> p} = \bigcup_{q > p} P_{\geq q} \period 
	\end{equation*}
	Thus, \cref{cor:E_exodromic_Van_Kampen} implies that $(\X_{>p}, P_{>p})$ is $ \Ecal $-exodromic.
	
	Now consider the following diagram:
	\begin{equation*}
		\begin{tikzcd}[sep=3em]
			\LC(\X_p) \tensor \Ecal \arrow[d, "\boxtensor"'] & \ConsP(\X) \tensor \Ecal \arrow[l] \arrow[r] \arrow[d, "\boxtensor"] & \Cons_{P_{>p}}(\X_{>p}) \tensor \Ecal \arrow[d, "\boxtensor"] \\
			\LC(\X_p;\Ecal) & \ConsP(\X;\Ecal) \arrow[r] \arrow[l] & \Cons_{P_{>p}}(\X_{>p};\Ecal) \period
		\end{tikzcd} 
	\end{equation*}
	The inductive hypothesis implies that the exterior vertical functors are equivalences.
	Since $ \Ecal $ is stable, $\ConsP(\X;\Ecal)$ is closed under finite limits in $ \Sh(\X;\Ecal) $.
	Thus, \cref{cor:exodromic_implies_universally_presentable} implies that the assumptions of \enumref{lem:recollements}{4} are satisfied.
	It follows that the bottom line is a recollement.
	Since $ \Ecal $ is stable, it is compatible with recollements; therefore, the top line is also a recollement.
	Thus, \enumref{lem:comparison_of_recollements_conservativity}{4} implies that the middle functor is an equivalence as well.
\end{proof}


\subsection{Exodromy with coefficients in \texorpdfstring{$\PrL$}{Prᴸ}}\label{subsec:exodromy_with_coefficients_in_PrL} 

Let $(\X, P)$ be an exodromic stratified \topos.
Recall that we write $\CATinfty$ for the (very large) \category of large \categories.
Working in a sufficiently large Grothendieck universe, $\CATinfty$ is compactly generated.
Therefore, combining \cref{lem:exodromy_with_coefficients} with \cref{lem:exit_path_category_for_compactly_assembled_coefficients_is_automatic}, we obtain an equivalence
\begin{equation}\label{eq:exodromy_large_cats}
	\ConsP(X ; \CATinfty ) \equivalent \Fun(\Piinfty(\X,P), \CATinfty) \period
\end{equation}
In many situations it is convenient to replace $\CATinfty$ by $\PrL$; however, since $\PrL$ is not itself presentable, one needs some extra care.

\begin{definition}
	Let $ (\X,P) $ be a stratified \topos.
	The \category of $\PrL$-valued sheaves on $\X$ is
	\begin{equation*}
		 \Sh(\X; \PrL) \colonequals \Funlim(\X^{\op}, \PrL) \period 
	\end{equation*}
\end{definition}

\begin{observation}
	Recall from \cite[Proposition 5.5.3.13]{HTT} that the forgetful functor $\PrL \to \CATinfty$ preserves limits.
	Since $\Sh(\X;\CATinfty) \colonequals \X \tensor \CATinfty$, \cite[Proposition 4.8.1.17]{HA} supplies a canonical functor
	\begin{equation*}
		\Sh(\X; \PrL) \to \Sh(\X; \CATinfty) \period 
	\end{equation*}
\end{observation}

\begin{definition}
	Let $(\X, P)$ be a stratified \topos.
	The \category of \defn{$\PrL$-valued $P$-constructible sheaves on $\X$} is the fiber product
	\begin{equation*}
		\ConsP(\X;\PrL) \colonequals \Sh(\X; \PrL) \crosslimits_{\Sh(\X;\CATinfty)} \ConsP(\X;\CATinfty) \period 
	\end{equation*}
\end{definition}

Although the above definition might seem ad hoc (because the restriction to strata are computed in $\CATinfty$ rather than in $\PrL$), it is justified by the following result:

\begin{proposition} \label{prop:exodromy_PrL}
	Let $ (\X,P) $ be an exodromic stratified \topos.
	Then the equivalence \eqref{eq:exodromy_large_cats} induces an adjoint equivalence
	\begin{equation*}
		\Phi \colon \ConsP(\X;\PrL) \leftrightarrows \Fun(\Piinfty(\X,P), \PrL) \colon \Psi \period 
	\end{equation*}
\end{proposition}

\begin{proof}
	Under the identification
	\begin{equation*}
		\ConsP(\X;\CATinfty) \equivalent \ConsP(\X) \tensor \CATinfty \equivalent \Funlim(\ConsP(\X)^{\op}, \CATinfty) \comma 
	\end{equation*}
	the equivalence \eqref{eq:exodromy_large_cats} is realized by the functor
	\begin{equation*}
		\Phi \colon \Funlim(\ConsP(\X)^{\op}, \CATinfty) \to \Fun(\Piinfty(\X,P), \CATinfty)  
	\end{equation*}
	given by restriction along the inclusion $\Piinfty(\X,P) \hookrightarrow \ConsP(\X)^{\op}$.
	The inverse of $ \Phi $ is the functor
	\begin{equation*}
		\Psi \colon \Fun(\Piinfty(\X,P), \CATinfty) \to \Funlim(\ConsP(\X)^{\op}, \CATinfty) 
	\end{equation*}
	given by right Kan extension along the same inclusion.
	Consider the composite
	\begin{equation*}
		\begin{tikzcd}
			\ConsP(\X;\PrL) \arrow[r] & \ConsP(\X;\CATinfty) \arrow[r, "\Phi"] & \Fun(\Piinfty(\X,P), \CATinfty) \period
		\end{tikzcd} 
	\end{equation*}
	Unraveling the definitions, we see that this functor takes $F \in \ConsP(\X;\PrL)$ seen as a limit-preserving functor
	\begin{equation*}
		F \colon \ConsP(\X)^{\op} \to \PrL 
	\end{equation*}
	to the restriction of $ F $ to $\Piinfty(\X,P)$.
	In particular, this composite factors through $\Fun(\Piinfty(\X,P), \PrL)$.
	Committing a slight abuse of notation, we still denote the resulting functor as
	\begin{equation*}
		\Phi \colon \ConsP(\X;\PrL) \to \Fun(\Piinfty(\X,P), \PrL) \period 
	\end{equation*}
	Similarly, since the forgetful functor $\PrL \to \CATinfty$ preserves limits by \cite[Proposition 5.5.3.13]{HTT} we see that $\Psi$ induces a well defined functor
	\begin{equation*}
		\Psi \colon \Fun(\Piinfty(\X,P), \PrL) \to \ConsP(\X; \PrL) \period 
	\end{equation*}
	Since the pair $(\Phi, \Psi)$ is an adjoint equivalence and the forgetful functor $\PrL \to \CATinfty$ is faithful and full on equivalences, we deduce that unit and counits at the level of $\CATinfty$ induce a unit and a counit transformation at the level of $\PrL$, and therefore that they form an adjoint equivalence.
\end{proof}

\begin{corollary}\label{cor:exodromy_PrL_functoriality}
	Let $\flowerstar \colon (\X,P) \to (\Y,Q)$ be a morphism of exodromic stratified \topoi.
	Then the functor
	\begin{equation*}
		\fupperstar \colon \ConsQ(\Y; \CATinfty) \to \ConsP(\X;\CATinfty)
	\end{equation*}
	induces a well defined functor
	\begin{equation*}
		 \fupperstar \colon \ConsQ(\Y;\PrL) \to \ConsP(\X;\PrL) 
	\end{equation*}
	making the square
	\begin{equation*}
		\begin{tikzcd}[sep=3em]
			\ConsQ(\Y;\PrL) \arrow[r] \arrow[d, "\fupperstar"'] & \Fun(\Piinfty(\Y,Q), \PrL) \arrow[d, "- \of \Piinfty(f)"] \\
			\ConsP(\X;\PrL) \arrow[r] & \Fun(\Piinfty(\X,P), \PrL)
		\end{tikzcd} 
	\end{equation*}
	commutative.
\end{corollary}

\begin{proof}
	Recall from \cref{thm:all_morphisms_are_exodromic} that $\flowerstar$ is exodromic.
	Since $\CATinfty$ is compactly generated, it follows from \cref{lem:exit_path_category_for_compactly_assembled_coefficients_is_automatic} that the diagram
	\begin{equation*}
		\begin{tikzcd}[sep=3em]
			\ConsQ(\Y;\CATinfty) \arrow[r] \arrow[d, "\fupperstar"'] & \Fun(\Piinfty(\Y,Q), \CATinfty) \arrow[d, "- \of \Piinfty(f)"] \\
			\ConsP(\X;\CATinfty) \arrow[r] & \Fun(\Piinfty(\X,P), \CATinfty)
		\end{tikzcd} 
	\end{equation*}
	commutes.
	Since the functor $ - \of \Piinfty(f) $ clearly lifts to a functor
	\begin{equation*}
		- \of \Piinfty(f) \colon \Fun(\Piinfty(\Y,Q), \PrL) \to \Fun(\Piinfty(\X,P),\PrL) \comma 
	\end{equation*}
	it follows from \cref{prop:exodromy_PrL} that the same is true of $\fupperstar$.
\end{proof}

\begin{warning}
	The use of constructible sheaves in \Cref{cor:exodromy_PrL_functoriality} is fundamental.
	For instance, the functor
	\begin{equation*}
		\fupperstar \colon \Sh(\Y;\CATinfty) \to \Sh(\X;\CATinfty)
	\end{equation*}
	generally does not carry $\Sh(\Y;\PrL)$ to $\Sh(\X;\PrL)$.
\end{warning}

\begin{notation}
	Let $ \PrLomega \subset \PrL $ for the non-full subcategory with objects compactly generated presentable \categories and morphisms left adjoints that preserve compact objects.
\end{notation}

\begin{nul}
	Recall from \cite[Proposition 2.8.4]{Ayoub_Vezzani_Six_functors_rigid_analytic} that $\PrLomega$ is compactly generated.
	In particular for an exodromic stratified \topos $ (\X,P) $, \cref{lem:exodromy_with_coefficients} with \cref{lem:exit_path_category_for_compactly_assembled_coefficients_is_automatic} provide an adjoint equivalence
	\begin{equation*}
		\Phi^{(\upomega)} \colon \ConsP(\X;\PrLomega) \leftrightarrows \Fun(\Piinfty(X,P), \PrLomega) \colon \Psi^{(\upomega)} \period 
	\end{equation*}
	The natural functor $\PrLomega \to \PrL$ induces by composition a map
	\begin{equation*}
		j \colon \Fun(\Piinfty(\X,P), \PrLomega) \to \Fun(\Piinfty(\X,P), \PrL) \period 
	\end{equation*}
	However, since the functor $ \PrLomega \to \PrL $ does \emph{not} preserve limits, we do not get an induced functor
	\begin{equation*}
		\Sh(\X;\PrLomega) \to \Sh(\X;\PrL) \period 
	\end{equation*}
	On the other hand, we have:
\end{nul}

\begin{corollary}
	There exists a canonical functor
	\begin{equation*}
		\ConsP(\X;\PrLomega) \to \ConsP(\X;\PrL) 
	\end{equation*}
	which makes the square
	\begin{equation*}
		\begin{tikzcd}[sep=3em]
			\ConsP(\X;\PrLomega) \arrow[r, "\Phi_{X,P}^{(\upomega)}"] \arrow[d] & \Fun(\Piinfty(\X,P), \PrLomega) \arrow[d, "j"] \\
			\ConsP(\X;\PrL) \arrow[r, "\Phi_{X,P}"'] & \Fun(\Piinfty(\X,P), \PrL)
		\end{tikzcd} 
	\end{equation*}
	commute.
\end{corollary}

\begin{proof}
	Thanks to \cref{prop:exodromy_PrL}, it is enough to define the left vertical map as $\Psi_{X,P} \circ j \circ \Phi_{X,P}^{(\upomega)}$.
\end{proof}


\section{Applications \& examples}\label{sec:applications_and_examples}

In this section, we apply the stability properties of \cref{sec:stability_properties_of_exodromic_stratified_topoi} to stratified \topoi arising from topology.
In \cref{subsec:consequences_for_stratified_topological_spaces}, we introduce the topological context for our results and state the stability theorem in this context (\Cref{thm:stability_properties_of_exodomic_stratified_spaces}).
Importantly, as a consequence of \Cref{thm:stability_properties_of_stratified_topoi} and the exodromy theorem for conically stratified spaces \cite{arXiv:2211.05004}, we deduce that for any stratified space $ (X,P) $ that \textit{locally} admits a conical refinement, the stratified \topos \smash{$ (\Shhyp(X),P) $} is exodromic (see \Cref{prop:properties_of_locally_conically_refineable_stratified_spaces}).
Many examples fall into this framework; see \cref{subsec:locally_conically_refineable_stratifications-examples}.
Of particular interest are stratified spaces coming from subanalytic geometry and real algebraic geometry. 
Under mild assumptions, we prove that in these geometric settings, the exit-path \categories are finite (\Cref{thm:subanalytic_stratified_spaces_are_conically_refineable,thm:algebraic_stratified_spaces_are_conically_refineable_and_categorically_finite}).
For use in a future paper, in \cref{subsec:criterion_for constructibility_with_respect_to_a_coarsening}, given an exodromic stratified \topos $ (\X,R) $ and map of posets $ \phi \colon \fromto{R}{P} $, we provide a recognition criterion for when $ R $-constructible objects are $ P $-constructible.
In \cref{subsec:relationship_to_Luries_simplicial_model_for_exit-paths}, we conclude by posing some questions about the relationship between our work and Lurie's simplicial model for exit-path \categories in the setting of conically refineable stratifications.


\subsection{Consequences for stratified topological spaces}\label{subsec:consequences_for_stratified_topological_spaces}

To fix a topological context to apply \Cref{thm:stability_properties_of_stratified_topoi}, we make the following definition.

\begin{definition}\label{def:exodromic_stratified_space}
	Let $ \Ecal $ be a presentable \category.
	We say that a stratified topological space $ s \colon \fromto{X}{P} $ is \defn{$ \Ecal $-exodromic} if the stratified \topos
	\begin{equation*}
		\slowerstarhyp \colon \fromto{\Shhyp(X)}{\Fun(P,\Spc)}
	\end{equation*}
	is $ \Ecal $-exodromic.
	In this case, we write
	\begin{equation*}
		\Piinfty(X,P) \colonequals \Piinfty(\Shhyp(X),P) \period
	\end{equation*}
\end{definition}

\noindent We also have the topological version of \Cref{def:categorical_compactness_and_finiteness}:

\begin{definition}\label{def:categorical_compactness_and_finiteness_in_topology}
	Let $ (X,P) $ be an exodromic stratified space.
	We say that $ (X,P) $ is:
	\begin{enumerate}
		\item \defn{Categorically finite} if $ \Piinfty(X,P) $ is a finite object of $ \Catinfty $. 
		(See \Cref{recollection_finite_category}.)

		\item \defn{Categorically compact} if $ \Piinfty(X,P) $ is a compact object of $ \Catinfty $.
	\end{enumerate}
\end{definition}

The following class of presentable \categories is well-behaved from the perspective of exodromy in topology:

\begin{definition}\label{def:admissible_category}
	Let $ P $ be a poset.
	We say that a presentable \category $ \Ecal $ is \defn{$ P $-admissible} if for every conically $ P $-stratified space $ (X,P) $ the hyperrestriction functors
	\begin{equation*}
		\left\{ \iupperstarhyp_p \colon \Shhyp(X;\Ecal) \to \Shhyp(X_p;\Ecal) \right\}_{p \in P}
	\end{equation*}
	are jointly conservative.
	We say that a presentable \category $ \Ecal $ is \defn{admissible} if for every poset $ P $, the \category $ \Ecal $ is $ P $-admissible.
\end{definition}

\begin{example}[{\cites[Lemma 5.21]{arXiv:2010.06473}[Lemma 2.12]{arXiv:2108.03545}}]
	Let $ \Ecal $ be a presentable \category.
	\begin{enumerate}
		\item If $ \Ecal $ is compactly assembled, then $ \Ecal $ is admissible.

		\item If $ \Ecal $ is stable or \atopos, then for every noetherian poset $ P $, the \category $ \Ecal $ is $ P $-admissible.
	\end{enumerate}
\end{example}

\begin{example}[{\cite[Theorem 5.17 \& Remark 5.18]{arXiv:2211.05004}}]\label{ex:conically_stratified_spaces_admit_exit-path_categories_for_admissible_coefficients}
	Let $ (X,P) $ be a conically stratified space with locally weakly contractible strata and let $ \Ecal $ be a $ P $-admissible \category.
	Then $ (X,P) $ is $ \Ecal $-exodromic.
\end{example}

When the strata of $ (X,P) $ are locally weakly contractible, we get a particularly nice description of the objects of the exit-path \category:

\begin{observation}[{(the objects of $ \Piinfty(X,P)$)}]\label{obs:objects_of_Exit_with_locally_weakly_contractible_strata}
	Let $ (X,P) $ be an exodromic stratified space with locally weakly contractible strata.
	Combining \enumref{ex:monodrom_in_topology}{1} with \Cref{obs:objects_of_Exit}, we see that there is a natural identification 
	\begin{equation*}
		\Piinfty(X,P)^{\equivalent} \equivalent \coprod_{p \in P} \Piinfty(X_p)
	\end{equation*}
	between the maximal \subgroupoid of $ \Piinfty(\X,P) $ and the coproduct of the underlying homotopy types of the strata of $ (X,P) $.

	Hence each point $ x \in X $ gives rise to an object \smash{$ [x] \in \ConsPhyp(X) $}, and every object of $ \Piinfty(X,P) $ is of this form.
	Moreover, it follows from the functoriality of the monodromy equivalence that the functor
	\begin{equation*}
		\fromto{\ConsPhyp(X)}{\Spc}
	\end{equation*}
	corepresented by $ [x] $ is equivalent to the stalk functor $ \xupperstar \colon \fromto{\ConsPhyp(X)}{\Spc} $.
	As a consequence, given a $ P $-hyperconstructible hypersheaf $ F $, every morphism $ \fromto{[x]}{[y]} $ gives rise to a \textit{specialization map} $ \fromto{\xupperstar F}{\yupperstar F} $ on stalks.
\end{observation}

The stability theorem for exodromic stratified \topoi has the following topological consequence:

\begin{theorem}[(stability properties of exodromic stratified spaces)]\label{thm:stability_properties_of_exodomic_stratified_spaces}
	\noindent 
	\begin{enumerate}
		\item\label{thm:stability_properties_of_exodomic_stratified_spaces.2} \emph{Stability under pulling back to locally closed subposets:} If $ (X,P) $ is an exodromic stratified space, then for each locally closed subposet $ S \subset P $, the stratified space $ (X_S, S) $ is exodromic and the induced functor
		\begin{equation*}
			\fromto{\Piinfty(X_S,S)}{\Piinfty(X,P)_S}
		\end{equation*}
		is an equivalence.
		In particular, the induced functor $ \fromto{\Piinfty(X,P)}{P} $ is conservative.

		\item\label{thm:stability_properties_of_exodomic_stratified_spaces.3} \emph{Stability under coarsening and localization formula:} 
		Let $ (X,R) $ be an exodromic stratified space and let $ \phi \colon \fromto{R}{P} $ be a map of posets. 
		Then $ (X,P) $ is exodromic and there is a natural equivalence
		\begin{equation*}
			\equivto{\Piinfty(X,R)[W_P\inv]}{\Piinfty(X,P)} \period 
		\end{equation*}

		\item\label{thm:stability_properties_of_exodomic_stratified_spaces.5} \emph{Functoriality:} The exodromy equivalence is functorial in all stratified maps between exodromic stratified spaces.

		\item\label{thm:stability_properties_of_exodomic_stratified_spaces.4} \emph{van Kampen:}
		Let $ (X,P) $ be a stratified space and let
		\begin{equation*}
			U_\bullet \colon \Deltainjop \to \Top_{/X} 
		\end{equation*}
		be an semi-simplicial étale hypercovering of $ X $.
		If for each $ n \geq 0 $, the stratified space $ (U_n,P) $ is exodromic, then the stratified space $ (X,P) $ is exodromic.
		Moreover, the natural functor
		\begin{equation*}
			\colim_{[n] \in \Deltainjop} \Piinfty(U_n,P) \to \Piinfty(X,P)
		\end{equation*}
		is an equivalence of \categories.

		\item\label{thm:stability_properties_of_exodomic_stratified_spaces.6} \emph{Stability of finiteness/compactness:} Let $ (X,P) $ be a stratified space.
		\begin{enumerate}[{\upshape (a)}]
			\item If $ (X,P) $ is exodromic and categorically finite (resp., compact), then for any locally closed subposet $ S \subset P $, the stratified space $ (X_S,S) $ is exodromic and categorically finite (resp., compact).

			\item Let $ U_1, \ldots, U_n $ be a \emph{finite} open cover of $ X $.
			Assume that each intersection $ (U_{i_1} \intersect \cdots \intersect U_{i_k},P) $ admits an refinement which is exodromic and categorically finite (resp., compact).
			Then $ (X,P) $ is exodromic and categorically finite (resp., compact).
		\end{enumerate}
	\end{enumerate}
\end{theorem}

\begin{proof}
	Item (1) is a special case of \Cref{cor:stability_under_pulling_back_to_locally_closed_subposets}, item (2) is a special case of \Cref{thm:stability_under_coarsening}, item (3) is a special case of \Cref{thm:all_morphisms_are_exodromic}, item (4) is a special case of \Cref{cor:van_Kampen_internally}, and item (5) follows from \Cref{lem:pulling_back_to_a_locally_closed_subposet_preserves_categorical_finiteness_and_compactness,lem:categorical_finiteness_and_compactness_can_be_detected_on_finite_covers,prop:categorical_finiteness_and_compactness_are_stable_under_coarsening}.
\end{proof}

\noindent Provided $ X $ is also locally weakly contractible, the classifying space of the exit-path \category of $ (X,P) $ coincides with the underlying homotopy type of $ X $:

\begin{corollary}\label{cor:Env_of_Exit_recovers_the_underlying_homotopy_type}
	Let $ (X,P) $ be an exodromic stratified space.
	If $ X $ locally weakly contractible, then the space $ \Env(\Piinfty(X,P)) $ is naturally equivalent to the underlying homotopy type of $ X $. 
\end{corollary}

\begin{proof}
	Note that \enumref{thm:stability_properties_of_exodomic_stratified_spaces}{3} shows that there is a natural equivalence
	\begin{equation*}
		\equivto{\Env(\Piinfty(X,P))}{\Piinfty(X,\pt)}
	\end{equation*}
	between the space obtained by inverting all morphisms in $ \Piinfty(X,P) $ and the shape of the \topos \smash{$ \Shhyp(X) $}.
	To conclude, recall that since $ X $ is locally weakly contractible, by \enumref{ex:monodrom_in_topology}{1}, the shape of \smash{$ \Shhyp(X) $} is naturally equivalent to the underlying homotopy type of $ X $.
\end{proof}

We conclude this subsection with some remarks about the stability theorem.

\begin{remark}
	\Cref{thm:stability_properties_of_stratified_topoi} also applies to other topological contexts.
	For example, given a topological space or stack $ X $ stratified by a noetherian poset $ P $, Jansen \cites{arXiv:2308.09551}{arXiv:2308.09550}{arXiv:2012.10777} and Clausen--Jansen \cite{arXiv:2108.01924} consider the stratified \topos $ (\Sh(X),P) $.
	\Cref{thm:stability_properties_of_stratified_topoi} applies in that setting as well, giving a variant of \Cref{thm:stability_properties_of_exodomic_stratified_spaces} for sheaves rather than hypersheaves.
	In that context, many of these results were already proven by Clausen--Jansen and Jansen; see \cites[Proposition 3.6]{arXiv:2108.01924}[Propositions 3.13 \& 3.20]{arXiv:2308.09550}.
\end{remark}

\begin{remark}[(the Künneth formula)]\label{rem:possible_failure_of_the_Kunneth_formula_in_topology}
	Let $ (X,P) $ and $ (Y,Q) $ be exodromic stratified spaces.
	The astute reader may have noticed that, unlike in \Cref{thm:stability_properties_of_stratified_topoi}, in \Cref{thm:stability_properties_of_exodomic_stratified_spaces} we have not stated that $ (X \cross Y,P \cross Q) $ is exodromic. 
	Neither have we stated that there is a Künneth formula
	\begin{equation*}
		\Piinfty(X \cross Y,P \cross Q) \equivalent \Piinfty(X,P) \cross \Piinfty(Y,Q) \period
	\end{equation*}
	This is because, in complete generality, we do not know if this is true.
	
	The issue is the following: there are natural colimit-preserving functors
	\begin{equation}\label{eq:Kunneth_comparison_for_sheaves}
		\Sh(X) \tensor \Sh(Y) \to \Sh(X \cross Y) \andeq \Shhyp(X) \tensor \Shhyp(Y) \to \Shhyp(X \cross Y) \comma
	\end{equation}
	however, in general neither of these functors need be an equivalence.
	In particular, in the topological setting, we do not immediately deduce a Künneth formula from \Cref{prop:Kunneth_formula_for_exodromic_topoi}.
	Nonetheless, Künneth formulas still hold in many contexts. 
	For example, if $ X $ is locally compact Hausdorff, then the left-hand functor in \eqref{eq:Kunneth_comparison_for_sheaves} is an equivalence \HTT{Proposition}{7.3.1.11}.
	So if $ X $ is locally compact Hausdorff and both $ \Sh(X) $ and $ \Sh(Y) $ are hypercomplete, then \Cref{thm:stability_properties_of_stratified_topoi} implies the Künneth formula for the exit-path \category of $ (X \cross Y, P \cross Q) $.
	For another important example, in \cref{subsec:locally_conically_refineable_stratifications-formal_properties} we show that if $ (X,P) $ and $ (Y,Q) $ locally admit refinements by conical stratifications, then we have a Künneth formula.
	See \Cref{prop:Kunneth_formula_for_locally_conically_refineable_stratifications}.
\end{remark}


\subsection{Locally conically refineable stratifications: formal properties}\label{subsec:locally_conically_refineable_stratifications-formal_properties}

Recall that if $ (X,P) $ is a conically stratified space, then for any open subset $ U \subset X $, the stratified space $ (U,P) $ is also conically stratified.
It is not clear if our definition of an exodromic stratified space is stable under passage to open subsets (cf. \Cref{quest:are_topoi_etale_over_an_exodromic_topos_exodromic}).
So we introduce the following strengthening of exodromicity that applies to many examples from geometry.  

\begin{definition}
	Let $ \Ecal $ be a presentable \category.
	A stratified space $ (X,P) $ is \defn{locally $ \Ecal $-exodromic} if there exists a basis $ \Bcal \subset \Open(X) $ such that for each $ U \in \Bcal $, the stratified space $ (U,P) $ is $ \Ecal $-exodromic.
\end{definition}

\begin{example}\label{ex:conically_stratified_spaces_are_locally_exodromic}
	Let $ (X,P) $ be a conically stratified space with locally weakly contractible strata and let $ \Ecal $ be a $ P $-admissible presentable \category in the sense of \Cref{def:admissible_category}.
	Then $ (X,P) $ is locally $ \Ecal $-exodromic.
\end{example}

\noindent In light of \Cref{thm:stability_properties_of_exodomic_stratified_spaces}, we have the following stability properties of locally exodromic stratifications:

\begin{proposition}\label{prop:basic_properties_of_locally_exodromic_stratifications}
	Let $ \Ecal $ be a presentable \category and $ (X,P) $ a stratified space.
	\begin{enumerate}
		\item\label{prop:basic_properties_of_locally_exodromic_stratifications.1} If $ (X,P) $ is locally $ \Ecal $-exodromic, then $ (X,P) $ is $ \Ecal $-exodromic.

		\item\label{prop:basic_properties_of_locally_exodromic_stratifications.2} If there exists an open cover $ \U $ of $ X $ such that for each $ U \in \U $, the stratified space $ (U,P) $ is locally $ \Ecal $-exodromic, then $ (X,P) $ is locally $ \Ecal $-exodromic. 

		\item\label{prop:basic_properties_of_locally_exodromic_stratifications.3} If $ (X,P) $ is locally $ \Ecal $-exodromic, then for any open subset $ U \subset X $, the stratified space $ (U,P) $ is locally $ \Ecal $-exodromic.

		\item\label{prop:basic_properties_of_locally_exodromic_stratifications.4} Assume that $ \Ecal $ is compatible with recollements.
		If $ (X,P) $ is locally $ \Ecal $-exodromic, then for any locally closed subposet $ S \subset P $, the stratified space $ (X_S,S) $ is locally $ \Ecal $-exodromic. 

		\item\label{prop:basic_properties_of_locally_exodromic_stratifications.5} If $ (X,P) $ is locally $ \Ecal $-exodromic, then for any map of posets $ \phi \colon \fromto{P}{P'} $, the stratified space $ (X,P') $ is locally $ \Ecal $-exodromic. 
	\end{enumerate}
\end{proposition}

\begin{proof}
	Item (1) is immediate from the fact that $ \Ecal $-exodromicity can be checked locally (\Cref{cor:E_exodromic_Van_Kampen}).
	Items (2) and (3) are immediate from the definitions.
	Item (4) follows from the definitions and the stability of $ \Ecal $-exodromicity under pulling back to locally closed subposets (\Cref{prop:E-exodromic_is_stable_under_pulling_back_to_locally_closed_subposets}).
	Item (5) follows from the definitions and the stability of $ \Ecal $-exodromicity under coarsenings (\enumref{thm:stability_properties_of_exodomic_stratified_spaces}{3}).
\end{proof}

For the examples in the rest of this subsection, it is convenient to introduce the following definition.

\begin{definition}
	Let $ s \colon \fromto{X}{P} $ be a stratified space.
	\begin{enumerate}
		\item A \defn{conical refinement} of $ (X,P) $ is the data of a conical stratification $ t \colon \fromto{X}{R} $ of $ X $ with locally weakly contractible strata and a map of posets $ \phi \colon \fromto{R}{P} $ such that $ s = \phi t $.
		We say that $ (X,P) $ is \defn{conically refineable} if there exists a conical refinement of $ (X,P) $.

		\item We say that $ (X,P) $ is \defn{locally conically refineable} if there exists an open cover $ \U $ of $ X $ such that for each $ U \in \U $, the stratified space $ (U,P) $ is conically refineable.
	\end{enumerate}
\end{definition}

First observe that locally conically refineable stratified spaces have locally weakly contractible strata (hence \Cref{obs:objects_of_Exit_with_locally_weakly_contractible_strata} applies).
In fact, even more is true; we introduce the following definition to axiomatize the categorical features of the exit-path \category of a locally conically refineable stratified space.

\begin{definition}\label{def:locally_cone-like_stratified_space}
	We say that a stratified space $ (X,P) $ is \defn{locally cone-like} if the following conditions are satisfied:
	\begin{enumerate}
		\item\label{def:locally_cone-like_stratified_space.1} The stratified space $ (X,P) $ is locally exodromic.

		\item\label{def:locally_cone-like_stratified_space.2} The strata of $ X $ are locally weakly contractible.

		\item\label{def:locally_cone-like_stratified_space.3} Every point $ x \in X $ admits a fundamental system of open neighborhoods $ \Ucal_x $ such that for each $ U \in \Ucal_x $, the object $ x \in \Piinfty(U,P) $ is initial.
	\end{enumerate}
\end{definition}

\begin{lemma}\label{lem:conically_stratified_spaces_are_locally_weakly_contractible_and_locally_cone-like}
	Let $ (X,P) $ be a conically stratified space with locally weakly contractible strata.
	Then:
	\begin{enumerate}
		\item\label{lem:conically_stratified_spaces_are_locally_weakly_contractible_and_locally_cone-like.1} The topological space $ X $ is locally weakly contractible.

		\item\label{lem:conically_stratified_spaces_are_locally_weakly_contractible_and_locally_cone-like.2} The stratified space $ (X,P) $ is locally cone-like.
	\end{enumerate}
\end{lemma}

\begin{proof}
	First recall that conically stratified spaces with locally weakly contractible strata are locally exodromic.
	We prove both items simultaneously.
	By \cite[Proposition 2.1.18]{PortaTeyssier}, every point $ x \in X $ admits a fundamental system of open neighborhoods $ \Ucal_x $ such that for each $ U \in \Ucal_x $, the object $ x $ is initial in \smash{$ \Piinfty(U,P) $}.
	For any such $ U $, \cite[Corollary 6.2.7]{PortaTeyssier} provides a canonical equivalence
	\begin{equation*} 
		\Piinfty(U) \equivalent \Env(\Piinfty(U,P)) \equivalent \ast \comma 
	\end{equation*}
	where $ \Piinfty(U) $ denotes the underlying homotopy type of $ U $.
	Therefore, each $ U $ is weakly contractible, i.e., $ X $ is locally weakly contractible.
\end{proof}

We now analyze the stability properties of the class of locally cone-like stratified spaces.
To start, we need a lemma.

\begin{lemma}\label{lem:localizations_preserve_initial_objects}
	Let $ L \colon \fromto{\Ccal}{\Dcal} $ be a functor of \categories that exhibits $ \Dcal $ as the localization of $ \Ccal $ at a collection of morphisms.
	If $ c \in \Ccal $ is initial, then $ L(c) \in \Dcal $ is initial.
\end{lemma}

\begin{proof}
	Recall that for \acategory $ \Ecal $, an object $ e \in \Ecal $ is initial if and only if the functor $ e \colon \fromto{\pt}{\Ecal} $ that picks out $ e $ is a limit-cofinal functor.
	Since $ L $ is a localization, $ L \colon \fromto{\Ccal}{\Dcal} $ is limit-cofinial \cite[Proposition 5.13]{MR4074276}.
	Hence the composite
	\begin{equation*}
		\begin{tikzcd}
			\pt \arrow[r, "c"] & \Ccal \arrow[r, "L"] & \Dcal
		\end{tikzcd}
	\end{equation*}
	is limit-cofinal.
\end{proof}

\begin{lemma}\label{lem:stability_properties_of_locally_cone-like_stratifed_spaces}
	\hfill
	\begin{enumerate}
		\item\label{lem:stability_properties_of_locally_cone-like_stratifed_spaces.1} Let $ (X,P) $ be a locally cone-like stratified space.
		Then for each locally closed subposet $ S \subset P $, the stratified space $ (X_S,S) $ is locally cone-like.

		\item\label{lem:stability_properties_of_locally_cone-like_stratifed_spaces.2} Let $ (X,R) $ be a locally cone-like stratified space and $ \phi \colon \fromto{R}{P} $ is a map of posets.
		Then the stratified space $ (X,P) $ is locally cone-like.

		\item\label{lem:stability_properties_of_locally_cone-like_stratifed_spaces.3} If $ (X,P) $ is a stratified space and $ \{U_{\alpha}\}_{\alpha \in A} $ is an open cover of $ X $ such that each stratified space $ (U_{\alpha},P) $ is locally cone-like, then $ (X,P) $ is locally cone-like.
	\end{enumerate}
\end{lemma}

\begin{proof}
	For (1), the only nontrivial condition to check is \enumref{def:locally_cone-like_stratified_space}{3}.
	Let $ x \in X_S $ and let $ \Ucal_x $ be a fundamental system of open neighborhoods of $ x $ in $ X $ such that for each $ U \in \Ucal_x $, the object $ x \in \Piinfty(U,P) $ is initial.
	Write
	\begin{equation*}
		\Ucal_{x,S} \colonequals \setbar{U_S}{U \in \Ucal_x} \period
	\end{equation*}
	Notice that $ U_S = U \intersect X_S $ and $ \Ucal_{x,S} $ is a fundamental system of open neighborhoods of $ x $ in $ X_S $.
	By \enumref{thm:stability_properties_of_exodomic_stratified_spaces}{2}, for each $ U \in \Ucal_x $, the natural functor 
	\begin{equation*}
		\fromto{\Piinfty(U_S,S)}{\Piinfty(U,P)}
	\end{equation*}
	is fully faithful.
	Since $ x \in \Piinfty(U_S,S) $ and $ x $ is initial in the larger \category $ \Piinfty(U,P) $, we deduce that $ x $ is also initial in $ \Piinfty(U_S,S) $.

	For (2), again the only nontrivial condition to check is \enumref{def:locally_cone-like_stratified_space}{3}.
	Let $ x \in X $ and let $ \Ucal_x $ be a fundamental system of open neighborhoods of $ x $ in $ X $ such that for each $ U \in \Ucal_x $, the object $ x \in \Piinfty(U,R) $ is initial.
	Then \Cref{lem:localizations_preserve_initial_objects} shows that $ x \in \Piinfty(U,P) $ is also initial.

	Item (3) is immediate from the definitions.
\end{proof}

Now we record the fundamental properties of the class of locally conically refineable stratified spaces.

\begin{proposition}[(properties of locally conically refineable stratified spaces)]\label{prop:properties_of_locally_conically_refineable_stratified_spaces}
	\hfill
	\begin{enumerate}
		\item\label{prop:properties_of_locally_conically_refineable_stratified_spaces.1} Let $ (X,P) $ be a stratified space and let $ \Ecal $ be an admissible presentable \category.
		If $ (X,P) $ is locally conically refineable, then $ (X,P) $ is locally $ \Ecal $-exodromic. 

		\item\label{prop:properties_of_locally_conically_refineable_stratified_spaces.2} Let $ (X,P) $ be a locally conically refineable stratified space.
		Then for each open subspace $ U \subset X $, the stratified space $ (U,P) $ is locally conically refineable.

		\item\label{prop:properties_of_locally_conically_refineable_stratified_spaces.3} Let $ (X,P) $ be a locally conically refineable stratified space.
		Then for each locally closed subposet $ S \subset P $, the stratified space $ (X_S,S) $ is locally conically refineable.

		\item\label{prop:properties_of_locally_conically_refineable_stratified_spaces.4} Let $ (X,R) $ be a locally conically refineable stratified space and $ \phi \colon \fromto{R}{P} $ is a map of posets.
		Then the stratified space $ (X,P) $ is locally conically refineable.

		\item\label{prop:properties_of_locally_conically_refineable_stratified_spaces.5} If $ (X,P) $ is a stratified space and $ \{U_{\alpha}\}_{\alpha \in A} $ is an open cover of $ X $ such that each stratified space $ (U_{\alpha},P) $ is locally conically refineable, then $ (X,P) $ is locally conically refineable.

		\item\label{prop:properties_of_locally_conically_refineable_stratified_spaces.6} If $ (X,P) $ is locally conically refineable, then $ X $ is locally weakly contractible.
		Moreover, the space
		\begin{equation*}
			\Env(\Piinfty(X,P))
		\end{equation*}
		is naturally equivalent to the underlying homotopy type of $ X $.

		\item\label{prop:properties_of_locally_conically_refineable_stratified_spaces.7} If $ (X,P) $ is a locally conically refineable stratified space, then $ (X,P) $ is locally cone-like.
	\end{enumerate}
\end{proposition}

\begin{proof}
	Item (1) follows from \Cref{prop:basic_properties_of_locally_exodromic_stratifications} and the fact that conically stratified spaces with locally weakly contractible strata are $ \Ecal $-exodromic.

	For (2), note that since the statement is local, it suffices to prove the claim when $ (X,P) $ admits a global conical refinement $ (X,R) $.
	Now note that since $ (X,R) $ is conically stratified, for any open subset $ U \subset X $, the stratified space $ (U,R) $ is also conical. 

	For (3), note that since the statement is local, it suffices to prove the claim when $ (X,P) $ admits a global conical refinement $ (X,R) $.
	In this case, \cite[Lemma 2.1.11]{arXiv:2211.05004} shows that the stratified space $ (X_S,R_S) $ is conical with locally weakly contractible strata.
	To conclude, note that $ (X_S,R_S) $ is a refinement of $ (X_S,S) $.

	Items (4) and (5) are immediate from the definitions.
	For (6), note that by \enumref{lem:conically_stratified_spaces_are_locally_weakly_contractible_and_locally_cone-like}{1}, $ X $ admits an open cover by locally weakly contractible topological spaces.
	Hence the claim is a special case of \Cref{cor:Env_of_Exit_recovers_the_underlying_homotopy_type}.
	Item (7) follows from the fact that conically stratified spaces are locally cone-like (\enumref{lem:conically_stratified_spaces_are_locally_weakly_contractible_and_locally_cone-like}{2}) and the stability properties of locally cone-like stratified spaces (\Cref{lem:stability_properties_of_locally_cone-like_stratifed_spaces}).
\end{proof}

We conclude this subsection with a Künneth formula for the exit-path \category of a product of locally conically refineable stratified spaces.
Due the issues mentioned in \Cref{rem:possible_failure_of_the_Kunneth_formula_in_topology}, our proof does not rely on the Künneth formula for exodromic stratified \topoi (\Cref{prop:Kunneth_formula_for_exodromic_topoi}).
Instead, we make use of the localization formula for the exit-path \category of a coarsening and the following lemma.

\begin{lemma}\label{lem:localization_product}
	Let $ \Ccal_1 $ and $ \Ccal_2 $ be \categories and let $ W_i \subset \Mor(\Ccal_i) $ be collections of morphisms.
	Then the natural functor
	\begin{equation*} 
		(\Ccal_1\cross \Ccal_2)[(W_1 \cross W_2)\inv] \to \Ccal_1[W_1\inv] \cross \Ccal_2[W_2\inv]
	\end{equation*}
	is an equivalence.
\end{lemma}

\begin{proof}
	This is an immediate consequence of \kerodon{02LV}.
\end{proof}

\begin{proposition}[(Künneth formula for locally conically refineable stratifications)]\label{prop:Kunneth_formula_for_locally_conically_refineable_stratifications}
	Let $ (X,P) $ and $ (Y,Q) $ be locally conically refineable stratified spaces.
	Then:
	\begin{enumerate}
		\item The product stratified space $ (X \cross Y, P \cross Q) $ is locally conically refineable.

		\item The natural functor
		\begin{equation*} 
			\Piinfty(X \cross Y, P \cross Q) \to \Piinfty(X,P) \cross \Piinfty(Y,Q) 
		\end{equation*}
		is an equivalence of \categories.

		\item The natural functor
		\begin{equation*} 
			\boxtensor \colon \ConsPhyp(X) \tensor \ConsQhyp(Y) \to \Conshyp_{P \cross Q}(X \cross Y)
		\end{equation*}
		is an equivalence of \categories.
	\end{enumerate}
\end{proposition}

\begin{proof}
	Item (1) is immediate from the definitions and the fact that a product of conically stratified spaces is still conically stratified.

	For (2), let
	\begin{equation*} 
		U_\bullet \colon \Deltainjop \to \Top_{/X}  \andeq V_\bullet \colon \Deltainjop \to \Top_{/Y} 
	\end{equation*}
	be open semi-simplicial hypercoverings of $ X $ and $ Y $ respectively, such that for each $ n \geq 0 $ the stratified spaces $(U_n, P)$ and $(V_n,Q)$ are conically refineable.
	Since $ \Deltainj $ is sifted, $ \Deltainj $-indexed colimits commute with finite products in $ \Catinfty $; hence \enumref{thm:stability_properties_of_exodomic_stratified_spaces}{4} shows that the natural functor
	\begin{equation*} 
		\colim_{[n] \in \Deltainj} \Piinfty(U_n, P) \cross \Piinfty(V_n, Q) \to \Piinfty(X \cross Y, P \cross Q) 
	\end{equation*}
	is an equivalence.
	We can therefore assume that $ (X,P) $ and $ (Y,Q) $ are (globally) conically refineable.

	Let $ (X,P') $ and $ (Y,Q') $ be conical refinements of $ (X,P) $ and $ (Y,Q) $, respectively.
	Then $ (X \cross Y, P' \cross Q') $ is conical and thus it is a conical refinement of $ (X \cross Y, P \cross Q) $.
	It follows from \cite[Theorem 5.4.1]{PortaTeyssier} and the explicit geometrical definition of the exit-path \category that the natural functor
	\begin{equation*} 
		\Piinfty(X \cross Y, P' \cross Q') \to \Piinfty(X,P') \cross \Piinfty(Y,Q') 
	\end{equation*}
	is an equivalence.
	Unraveling the definitions, we see that $W_{P \cross Q} = W_{P} \cross W_Q$ as collection of morphisms in $ \Piinfty(X,P') \cross \Piinfty(Y,Q') $.
	The conclusion now follows from \Cref{lem:localization_product}.

	Item (3) is immediate from (2) and the fact that the functor $\Fun(-,\Spc)$ carries products of \categories to tensor products in $ \PrL $.
\end{proof}


\subsection{Locally conically refineable stratifications: examples}\label{subsec:locally_conically_refineable_stratifications-examples}

We give some examples of locally conically refineable (hence locally exodromic) stratifications.

\begin{notation}[(simplicial complexes)]
	Let $ (V,S) $ be an simplicial complex, and regard $ S $ as a poset ordered by inclusion.
	Write $ \Delta^{(V,S)} $ for the geometric realization of $ (V,S) $.
	There is a natural stratification $ \fromto{\Delta^{(V,S)}}{S} $ with locally contractible strata; see \HAa{Definition}{A.6.7}.
\end{notation}

\begin{example}\label{ex:stratifications_of_locally_finite_simplicial_complexes_are_conically_refineable}
	Let $ (V,S) $ be a \textit{locally finite} simplicial complex and let $ \Ecal $ be an admissible presentable \category.
	Then the natural stratification $ \fromto{\Delta^{(V,S)}}{S} $ is conical \HAa{Proposition}{A.6.8}.
	Moreover, \HAa{Theorem}{A.6.10} shows that
	\begin{equation}\label{Pi_infty_simplicial_complex}
		\Piinfty(\Delta^{(V,S)},S) \equivalent S \period
	\end{equation}
	By \Cref{prop:properties_of_locally_conically_refineable_stratified_spaces}, we see that for any map of posets $ \fromto{S}{P} $, the stratified space $ (\Delta^{(V,S)},P) $ is locally $ \Ecal $-exodromic.
	That is, any stratified space admitting a refinement by a locally finite triangulation is locally $ \Ecal $-exodromic.
\end{example}

\begin{observation}\label{finite_geometric_realization_simplicial_complex}
	In light of \eqref{Pi_infty_simplicial_complex}, given a locally finite simplicial complex $ (V,S) $, the stratified space $ (\Delta^{(V,S)}, S) $ is categorically finite if and only if the set $ S $ is finite.
\end{observation}

\begin{example}\label{ex:Favero-Huang_tree_stratification_is_conically_refineable}
	The \textit{tree stratification} of a finite simplicial complex considered by Favero--Huang \cite[\S4.4]{arXiv:2205.03730} is conically refineable, hence locally exodromic.
	Moreover, \enumref{thm:stability_properties_of_exodomic_stratified_spaces}{6} and \Cref{finite_geometric_realization_simplicial_complex} show that the tree stratification is categorically finite. 
\end{example}

One source of locally exodromic stratifications comes from subanalytic stratifications of real analytic spaces.
Recall that subanalytic stratifications need not be conical; see \Cref{fig:circle}.

\begin{definition}
	Let $ X $ be a topological space.
	We say that a stratification $ \fromto{X}{P} $ is \defn{locally finite} if for every point $ x \in X $, there is an open neighborhood $ U $ of $ x $ such that $ U $ intersects only finitely many strata of $ (X,P) $.
\end{definition}


\begin{definition}\label{def:subanalytic_stratified_space}
	A \textit{subanalytic stratified space} is the data of a triple $ (M,X,P) $ where $ M $ is a smooth real analytic space, $ X \subset M $ is a locally closed subanalytic subset, and $ X \to P $ is a locally finite stratification by subanalytic subsets of $ M $.
\end{definition}

Subanalytic stratified spaces provide many examples of (locally) categorically finite stratified spaces:

\begin{definition}\label{def:local_categorical_compactness}
	Let $ (X,P) $ be a locally exodromic stratified space.
	We say that $ (X,P) $ is \defn{locally categorically finite (resp., compact)} if there exists an open cover $ \Ucal $ such that for each $ U \in \Ucal $, the exodromic stratified space $ (U,P) $ is categorically finite (resp., compact).
\end{definition}

\begin{theorem}\label{thm:subanalytic_stratified_spaces_are_conically_refineable}
	Let $ (M,X,P) $ be a subanalytic stratified space.
	Then:
	\begin{enumerate}
		\item\label{thm:subanalytic_stratified_spaces_are_conically_refineable.1} The stratified space $ (X,P) $ admits a refinement by a locally finite triangulation. 

		\item\label{thm:subanalytic_stratified_spaces_are_conically_refineable.2} For any admissible \category $ \Ecal $, the stratified space $ (X,P) $ is locally $ \Ecal $-exodromic.

		\item\label{thm:subanalytic_stratified_spaces_are_conically_refineable.3} If $ X $ is compact, then $ (X,P) $ admits a refinement by a \emph{finite} triangulation.
		Hence $ (X,P) $ is categorically finite.

		\item\label{thm:subanalytic_stratified_spaces_are_conically_refineable.4} The stratified space $ (X,P) $ is locally categorically finite.
		
		\item If $U \Subset X$ is a relatively compact subanalytic open subset, then $(U,P)$ is categorically finite.
	\end{enumerate}
\end{theorem}

\begin{proof}
	Item (1) follows from \cite[\S 1.7]{GM} combined with \cite{Goreski_triangulation}.
	Item (2) follows from (1) and \Cref{prop:properties_of_locally_conically_refineable_stratified_spaces}. 
	For (3), note that by (1), the stratified space $ (X,P) $ admits a  triangulation by a locally finite simplicial complex $ (\Delta^{(V,S)}, S) $. 
	Since $ X $ is compact, the poset $ S $ is finite.
	The final statement in (3) follows from \enumref{thm:stability_properties_of_exodomic_stratified_spaces}{6} and \Cref{finite_geometric_realization_simplicial_complex}.

	Now we prove (4).
	At the cost of shrinking $ M $, we can assume that $ X $ is closed in $ M $.
	Let $ x\in X $ and let $ B\subset M $ be a small ball centered at $ x $ such that $ X \intersect B $ intersects only finitely many strata.
	We claim that $ (X \intersect B,P) $ is categorically finite.
	Note that since $ X \intersect B $ intersects only finitely many strata, we may assume that $ P $ is finite.
	Extend $ X \intersect B\to P $ to a finite stratification $ B\to P^{\rhd} $ sending $ B\setminus (X \intersect B)$ to the terminal object of $P^{\rhd}$.
	Since $ P $ is closed in $P^{\rhd}$, \enumref{thm:stability_properties_of_exodomic_stratified_spaces}{6} reduces the claim to the case where $ X = B $.
	We thus need to show that $ (B,P^{\rhd}) $ is categorically finite.
	Write $ Q \colonequals P^{\rhd} $ and extend $ B\to Q $ to a finite stratification $ \Bbar\to Q^{\lhd} $ by sending $ \partial \Bbar$ to the initial object of $ Q^{\lhd} $.
	Recall that subanalytic subsets are closed under finite intersections, closures, and complements; as a consequence, the stratification $ (\Bbar, Q^{\lhd}) $ is also subanalytic.
	Since $ Q $ is open in $ Q^{\lhd} $, \enumref{thm:stability_properties_of_exodomic_stratified_spaces}{6} reduces the claim to the case where $ X = \Bbar $.
	An application of (3) now shows that $ (\Bbar,Q^{\lhd}) $ is categorically finite.
	
	Finally, we prove (5).
	The closure $\Ubar$ is again a subanalytic (see e.g., the discussion following \cite[Definition 3.1]{Bierstone_Milman}), and it is compact by assumption.
	In particular, it intersects only finitely many strata.
	As before, we can thus assume that $ P $ is finite.
	Extend $ U \to P$ to a finite stratification \smash{$ \Ubar \to P^{\lhd}$} sending the boundary \smash{$\partial \Ubar \colonequals \Ubar \sminus U$} to the initial object of $P^{\lhd}$.
	Then $P$ is open in $P^{\lhd}$, so \enumref{thm:stability_properties_of_exodomic_stratified_spaces}{6} reduces us to verify that $ (\Ubar,P)$ is categorically finite, and this follows directly from (3).
\end{proof}

\begin{example}\label{ex:Favero-Huang-Bondal-Ruan_stratification_is_locally_conically_refineable}
	The \textit{Bondal--Ruan stratification} of the $ n $-torus considered by Favero--Huang \cites{MR2278891}[\S5.2]{arXiv:2205.03730} is subanalytic, hence locally exodromic, categorically finite, and locally categorically finite.
\end{example}

Stratifications of real algebraic varieties are especially well-behaved:

\begin{definition}\label{def:algebraic_stratified_space}
	An \textit{algebraic stratified space} is the data of a stratified space $ (X,P) $ where $ X $ is (the real points of) an algebraic variety over $ \RR $ and $ X \to P $ is a finite stratification by Zariski locally closed subsets.
\end{definition}

\begin{warning}
	Unlike a subanalytic stratified space, an algebraic stratified space $ (X,P) $ is not presented as a subspace of a smooth algebraic variety.
	Note that if $ X $ is singular, such a presentation may not exist.
\end{warning}

\begin{theorem}\label{thm:algebraic_stratified_spaces_are_conically_refineable_and_categorically_finite}
	Let $ (X,P) $ be an algebraic stratified space.
	Then:
	\begin{enumerate}
		\item\label{thm:algebraic_stratified_spaces_are_conically_refineable_and_categorically_finite.1} If $ X $ is affine, $ (X,P) $ admits a categorically finite conical refinement $ (X,R) $ with $ R $ finite.
		Hence $ (X,P) $ is categorically finite.

		\item\label{thm:algebraic_stratified_spaces_are_conically_refineable_and_categorically_finite.2} The stratified space $ (X,P) $ is locally conically refineable.

		\item\label{thm:algebraic_stratified_spaces_are_conically_refineable_and_categorically_finite.3} For any admissible \category $ \Ecal $, the stratified space $ (X,P) $ is locally $ \Ecal $-exodromic and locally categorically finite. 

		\item\label{thm:algebraic_stratified_spaces_are_conically_refineable_and_categorically_finite.4} The stratified space $ (X,P) $ is categorically finite.
	\end{enumerate}
\end{theorem}

\begin{proof}
	For (1), let us view $ X $ as a closed subset of $ \AA^n $.
	Let $ \Xbar $ be the closure of $ X $ in $ \PP^n $.
	Define $ Q \colonequals (P^{\rhd})^{\lhd} $ and let us extend $ X\to P $ as a stratification $ \PP^n \to Q $ by sending \smash{$ \Xbar \setminus X $}  to the initial object of $ Q $ and \smash{$ \PP^n\setminus \Xbar $} to the terminal object of $ Q $.
	Then, $ (\PP^n,Q) $ is a compact subanalytic stratified space.
	By \enumref{thm:subanalytic_stratified_spaces_are_conically_refineable}{3}, $ (\PP^n,Q) $ admits a refinement $ Q' \to Q $ by a finite triangulation.
	Thus, $ (\PP^n,Q') $ is conically stratified with locally weakly contractible strata.
	Moreover, \Cref{finite_geometric_realization_simplicial_complex} shows that $ (\PP^n,Q') $ is categorically finite.
	Since $ P \subset Q $ is locally closed, $ (X,Q'_P) $ is also conically stratified with locally weakly contractible strata.
	Moreover, \Cref{prop:base_change_lc_preserves_finite_and_compact_categories} shows that $ (X,Q'_P) $ is categorically finite.
	Finally, since $ Q $ is finite, so is $ Q'_P $.

	Item (2) is an immediate consequence of (1).
	Item (3) follows from (1) and \Cref{prop:properties_of_locally_conically_refineable_stratified_spaces}.
	Since $ X $ admits a finite cover by affine subsets whose iterated intersections are again affine, (4) follows from (1) and \enumref{thm:stability_properties_of_exodomic_stratified_spaces}{6}.
\end{proof}


\subsection{A criterion for constructibility with respect to a coarsening}\label{subsec:criterion_for constructibility_with_respect_to_a_coarsening}

Let $ (X,R) $ be an exodromic stratified space with locally weakly contractible strata and let $ \phi \colon \fromto{R}{P} $ be a map of posets.
It is often useful to have a geometric recognition criterion for when an $ R $-hyperconstructible hypersheaf is $ P $-hyperconstructible.
The goal of this subsection is to explain such a criterion: an $ R $-hyperconstructible hypersheaf $ F $ on $ X $ is $ P $-hyperconstructible if and only if for each morphism $ \gamma \colon \fromto{x}{y} $ in the exit-path \category $ \Piinfty(X,R) $ that lies in a single stratum of the coarser stratification $ (X,P) $, the induced specialization map on stalks
\begin{equation*}
	\fromto{\yupperstar F}{\xupperstar F}
\end{equation*}
is an equivalence.%
\footnote{We do not make use of this result in the present paper, but need it in future work.}
This criterion is an easy consequence of the exodromy equivalence and localization formula for the exit-path \category of a coarsening.

\begin{notation}[(cospecialization maps)]\label{ntn:cospecialization_maps}
	Let $ \Ecal $ be a presentable \category and let $ (\X,R) $ be an $ \Ecal $-exodromic stratified \topos.
	\begin{enumerate}
		\item Write
		\begin{equation*}
			[-] \colon \incto{\Piinfty(\X,R)^{\op}}{\ConsR(\X)} \comma \quad \goesto{x}{[x]}
		\end{equation*}
		for the inclusion of the subcategory of atomic objects.
		For each $ E \in \Ecal $ and $ x \in \Piinfty(\X,R) $, we write $[x] \tensor E$ for the canonical object in
		\begin{equation*} 
			\ConsR(\X) \tensor \Ecal \equivalence \ConsR(\X;\Ecal) \period 
		\end{equation*}
		
		\item Given a morphism $ \gamma \colon \fromto{x}{y} $ in $ \Piinfty(\X,R) $, we write
		\begin{equation*} 
			\cosp_{R}^{\gamma} \colonequals [\gamma] \colon [y] \to [x] 
		\end{equation*}
		for the corresponding morphism in $ \ConsR(\X) $.
		We refer to $ \cosp_{R}^{\gamma} $ as the \defn{cospecialization map} associated to $ \gamma $.
		Again, for general $ \Ecal $ and for each $E \in \Ecal$, we write \smash{$ \cosp_R^{\gamma} \tensor \id{E} $} for the corresponding morphism in \smash{$\ConsR(\X;\Ecal)$}.
	\end{enumerate}
\end{notation}

\begin{observation}[(specialization maps)]\label{obs:specalization_maps}
	Let $ (X,R) $ be an exodromic stratified space with locally weakly contractible strata.
	In light of \Cref{obs:objects_of_Exit_with_locally_weakly_contractible_strata}, given a $ R $-hyper\-constructible hypersheaf $ F $ and a morphism $ \gamma \colon \fromto{x}{y} $ in $ \Piinfty(X,R) $, applying $\Map(-,F)$ to the cospecialization map
	\begin{equation*}
		\cosp_{R}^{ \gamma} \colon [y] \to [x] 
	\end{equation*}
	yields a \textit{specialization map} $ \fromto{\xupperstar F}{\yupperstar F} $ on stalks.
\end{observation}

\begin{recollection}\label{rec:presheaves_on_a_localization}
	Let $ \Dcal_0 $ be a small \category and let $ W \subset \Mor(\Dcal_0) $ be a class of morphisms.
	Write $ L \colon \fromto{\Dcal_0}{\Dcal_0[W\inv]} $ for the localization functor.
	Then, by the definition of localization, the induced pullback functor
	\begin{equation*}
		\Lupperstar \colon \fromto{\PSh(\Dcal_0[W\inv])}{\PSh(\Dcal_0)}
	\end{equation*}
	is fully faithful with image those $ F \colon \fromto{\Dcal_0^{\op}}{\Spc} $ that carry morphisms in $ W $ to equivalences.
\end{recollection}

\begin{proposition}\label{prop:presheaves_on_a_localization_via_locality}
	Let $ \Dcal_0 $ be a small \category, $ W \subset \Mor(\Dcal_0) $ a class of morphisms, and $ \Ecal $ a presentable \category.
	Write $ L \colon \fromto{\Dcal_0}{\Dcal_0[W\inv]} $ for the localization functor.
	Then:
	\begin{enumerate}
		\item Let $ F \in \PSh(\Dcal_0;\Ecal) $ and let $ f $ be a morphism in $ \Dcal_0 $.
		Then the full subcategory of $ \Ecal $ spanned by those objects $ E \in \Ecal $ such that $ F $ is $ \yo(f) \tensor \id{E} $-local is closed under colimits and retracts.

		\item An object $ F \in \PSh(\Dcal_0;\Ecal) $ is in the image of the fully faithful pullback functor
		\begin{equation*}
			\Lupperstar \colon \incto{\PSh(\Dcal_0[W\inv];\Ecal)}{\PSh(\Dcal_0;\Ecal)}
		\end{equation*}
		if and only if for each $ w \in W $ and $ E \in \Ecal $, the object $ F $ is $ \yo(w) \tensor \id{E} $-local.
	\end{enumerate}
\end{proposition}

\begin{proof}
	Immediate from \Cref{rec:presheaves_on_a_localization} and the definitions.
\end{proof}

\begin{corollary}\label{cor:P-constructibility_via_locality}
	Let $ \Ecal $ be a presentable \category, let $ (\X,R) $ be an $ \Ecal $-exodromic stratified \topos, let $ \phi \colon \fromto{R}{P} $ be a map of posets, let $ F \in \ConsR(\X;\Ecal) $, and let $ \gamma \colon \fromto{x}{y} $ be a morphism in $ \Piinfty(\X,R) $.
	Then:
	\begin{enumerate}
		\item\label{cor:P-constructibility_via_locality.1} The full subcategory of $ \Ecal $ spanned by those objects $ E \in \Ecal $ such that $ F $ is $ (\cosp_R^\gamma \tensor \id{E}) $-local is closed under colimits and retracts.

		\item\label{cor:P-constructibility_via_locality.2} The $ R $-constructible object $ F $ is $ P $-constructible if and only if for each $ \gamma \in W_P $ and $ E \in \Ecal $, the object $ F $ is $\cosp_{R}^{\gamma} \tensor \id{E} $-local.
	\end{enumerate}
\end{corollary}

\begin{proof}
	In light of the exodromy equivalence and the localization formula for the exit-path \category of a coarsening (\Cref{thm:stability_under_coarsening}), this result is a special case of \Cref{prop:presheaves_on_a_localization_via_locality}.
\end{proof}


\subsection{Relationship to Lurie's simplicial model for exit-paths}\label{subsec:relationship_to_Luries_simplicial_model_for_exit-paths}

We conclude with some remarks and questions regarding the relationship between the exit-path \category in the conically refineable setting and Lurie's simplicial model for exit-paths $ \Sing(X,R) $.
See \cites[\HAappthm{Definition}{A.6.2}]{HA}[\S2]{arXiv:2211.05004} for background on the simplicial model.

\begin{recollection}
	Let $ (X,R) $ be a conically stratified space with locally weakly contractible strata.
	Then Lurie's exit-path simplicial set $ \Sing(X,R) $ is \acategory \HAa{Theorem}{A.6.4}.
	Moreover, $ (X,R) $ is exodromic in the sense of \cref{def:exodromic_stratified_space} and \cite[Theorem 5.4.1]{PortaTeyssier} implies that there is an equivalence of \categories 
	\begin{equation*}
		\Piinfty(X,R) \equivalent \Sing(X,R) \period 
	\end{equation*}
	That is, \cite[Theorem 5.4.1]{PortaTeyssier} provides an \textit{explicit} simplicial model for the exit-path \category.
\end{recollection}

\begin{observation}\label{obs:simplicial_model_for_exit-paths}
	Let $ (X,R) $ be a conically stratified space with locally weakly contractible strata and let $ \phi \colon \fromto{R}{P} $ be a map of posets.
	In general, the exit-path simplicial set $ \Sing(X,P) $ need not be \acategory.
	Write \smash{$ \widetilde{\Sing}(X,P) $} for the fibrant replacement of $ \Sing(X,P) $ in the Joyal model structure on simplicial sets over (the nerve of) $ P $.
	By construction, the composite
	\begin{equation*}
		\begin{tikzcd}
			\Piinfty(X,R) \equivalent \Sing(X,R) \arrow[r] & \Sing(X,P) \arrow[r] & \widetilde{\Sing}(X,P)
		\end{tikzcd}
	\end{equation*}
	carries all morphisms in $ W_P $ to equivalences.
	By \Cref{thm:stability_properties_of_exodomic_stratified_spaces} and the universal property of the localization, this induces a functor
	\begin{equation*}
		\begin{tikzcd}
			\Piinfty(X,P) \equivalent \Piinfty(X,R)[W_P\inv] \arrow[r] & \widetilde{\Sing}(X,P) \period
		\end{tikzcd}
	\end{equation*}
	Moreover, \cite[Lemma 2.5.2]{arXiv:1811.01119} and \enumref{thm:stability_properties_of_exodomic_stratified_spaces}{2} imply that for each $ p \in P $, the induced map on strata
	\begin{equation*}
		\Piinfty(X,P) \cross_{P} \{p\} \to \widetilde{\Sing}(X,P) \cross_{P} \{p\} 
	\end{equation*}
	is an equivalence of \groupoids.
\end{observation}

\begin{nul}
	Note that \textit{if} the functor $ \fromto{\Piinfty(X,P)}{\widetilde{\Sing}(X,P)} $ is an equivalence of \categories, then \Cref{prop:properties_of_locally_conically_refineable_stratified_spaces} implies that there is an equivalence of \categories
	\begin{equation*}
		\ConsPhyp(X) \equivalent \Fun(\Sing(X,P),\Spc) \period
	\end{equation*}
	That is, even though Lurie's exit-path simplicial set $ \Sing(X,P) $ may not be \acategory, $ \Sing(X,P) $ still corepresents hyperconstructible hypersheaves. 
\end{nul}

\begin{question}
	In the setting of \Cref{obs:simplicial_model_for_exit-paths}, is the functor
	\begin{equation*}
		\fromto{\Piinfty(X,P)}{\widetilde{\Sing}(X,P)}
	\end{equation*}
	an equivalence of \categories?
	If not, what are some mild conditions on the stratified space $ (X,P) $ that guarantee that this functor is an equivalence?
\end{question}

\appendix


\section{Inverting arrows over a poset}\label{appendix:inverting_arrows_over_a_poset}

Let $ P $ be a poset.
In \Cref{thm:stability_properties_of_stratified_topoi}, we are interested in the following situation: we have \acategory $ \Ccal $ and functor $ F \colon \fromto{\Ccal}{P} $, and we want to form the localization of $ \Ccal $ at the set $ W_P $ of morphisms that $ F $ carries to identities in $ P $.
There are two goals of this appendix.
First, we show that for each $ p \in P $, the fiber of $ \Ccal[W_P\inv] $ over $ p $ coincides with the classifying space of the fiber $ \Ccal \cross_{P} \{p\} $; see \Cref{prop:ff_and_localization}. 
From this we deduce that the natural functor \smash{$ \fromto{\Ccal[W_P\inv]}{P} $} is conservative and that \smash{$ \Ccal[W_P\inv] $} is idempotent complete.
Second, we show that if $ \Ccal $ is finite (resp., compact), then the localization $ \Ccal[W_P\inv] $ is also finite (resp., compact).
See \Cref{prop:inverting_morphisms_fiberwise_preserves_compactness}.

In \cref{subsec:layered_categories}, we review some basic facts about \categories with a conservative functor to a poset.
\Cref{subsec:strata_of_localizations} proves structural results about the localization $ \Ccal[W_P\inv] $.
In \cref{subsec:compactness_of_categories_with_a_conservative_functor_to_a_poset}, we explain various characterizations of finiteness and compactness in the \category of \categories with a conservative functor to the poset $ P $.
We use these characterizations to prove stability properties of finite and compact \categories with over $ P $.


\subsection{Layered \texorpdfstring{$\infty$}{∞}-categories}\label{subsec:layered_categories}

We start by collecting background material about the types of \categories that arise as exit-path \categories of stratified spaces.

\begin{recollection}\label{rec:characterization_of_conservative_functors_to_posets}
	Let $ F \colon \fromto{\Ccal}{P} $ be a functor from \acategory to a poset.
	The following are equivalent:
	\begin{enumerate}
		\item\label{rec:characterization_of_conservative_functors_to_posets.1} The functor $ F \colon \fromto{\Ccal}{P} $ is conservative.

		\item\label{rec:characterization_of_conservative_functors_to_posets.2} For each $ p \in P $, the fiber $ \Ccal \cross_{P} \{p\} $ is \agroupoid. 
	\end{enumerate}
\end{recollection}

\begin{recollection}\label{rec:layered_categories}
	Let $ \Ccal $ be \acategory.
	The following are equivalent:
	\begin{enumerate}
		\item\label{rec:layered_categories.1} There exists a poset $ P $ and a conservative functor $ \fromto{\Ccal}{P} $.
		
		\item\label{rec:layered_categories.2} For each $ x \in \Ccal $, every endomorphism $ \fromto{x}{x} $ is an equivalence. 
	\end{enumerate}
	If these equivalent conditions are satisfied, we say that $ \Ccal $ is a \defn{layered \category}.
	By the stratified homotopy hypothesis, \acategory $ \Ccal $ is layered if and only if $ \Ccal $ is equivalent to the exit-path \category of a stratified space; see \cite[Theorem 0.1.1]{arXiv:1811.01119} for a precise formulation of this result.
\end{recollection}

\noindent An important fact is that layered \categories are idempotent complete.
For this, recall \Cref{ntn:enveloping_groupoids}.

\begin{lemma}\label{lem:layered_implies_idempotent_complete}
	Let $ \Ccal $ be layered \category.
	Then:
	\begin{enumerate}
		\item\label{lem:layered_implies_idempotent_complete.1} If $ e \colon \fromto{x}{x} $ is a morphism in $ \Ccal $ such that there exists an equivalence $ e^2 \equivalent e $, then $ e \equivalent \id{x} $.

		\item\label{lem:layered_implies_idempotent_complete.2} The \category $ \Ccal $ is idempotent complete.
	\end{enumerate}
\end{lemma}

\begin{proof}
	For (1), note that since $ \Ccal $ is layered, the morphism $ e $ is an equivalence.
	Since $ e^2 \equivalent e $, the fact that $ e $ is invertible implies that $ e \equivalent \id{x} $.
	For (2), observe that since $ \Ccal $ is layered, every idempotent $ e \colon \Idem \to \Ccal $ factors through the maximal \subgroupoid $ \Ccal^{\equivalent} $ of $ \Ccal $.
	Hence $ e $ descends to a functor $ \Env(\Idem) \to \Ccal^{\equivalent} $.
	Since $ \Env(\Idem) $ is contractible \HTT{Lemma}{4.4.5.10}, we conclude that $ e $ splits.
\end{proof}

We conclude this subsection by recalling an alternative description of \categories with a conservative functor to a fixed poset in terms of links.

\begin{notation}
	Let $ P $ be a poset and write \smash{$ \CatPcons \subset \CatP $} for the full subcategory spanned by those objects such that the specified functor $ \fromto{\Ccal}{P} $ is conservative.
\end{notation}

\begin{recollection}\label{rec:categories_with_a_conservative_functor_to_a_poset}
    Let $ P $ be a poset.
    Write $ \sd(P) $ for the poset of nonempty linearly ordered finite subsets of $ P $, ordered by inclusion. 
    The poset $ \sd(P) $ is referred to as the \defn{subdivision} of $ P $.
    Consider the \defn{nerve} functor
     \begin{align*}
        \Nup_P \colon \CatP  &\longrightarrow \Fun(\sd(P)^{\op},\Spc) \\
            [\Ccal \to P] &\longmapsto \brackets{\{p_0 < \cdots < p_n\} \mapsto \Map_{\CatP }(\{p_0 < \cdots < p_n\}, \Ccal)} \period
    \end{align*}
    The spaces $ \Nup_P(\Ccal)(\{p_0 < \cdots < p_n\}) $ are called the \defn{$ n $-links} of $ \Ccal $.
    
    Barwick--Glasman--Haine proved that the restriction \smash{$ \Nup_P \colon \CatPcons \to \Fun(\sd(P)^{\op},\Spc) $} to \categories with a conservative functor to $ P $     is a fully faithful right adjoint.
    Moreover, they identified the image as the full subcategory spanned by those $ X \colon \sd(P)^{\op} \to \Spc $ such that for every nonempty linearly ordered finite subset $ \{p_0 < \cdots < p_n \} \subset P $, the natural map
    \begin{equation*}
        X\{p_0 < \cdots < p_n \} \to X\{p_0 < p_1\} \crosslimits_{X\{p_1\}} X\{p_1 < p_2\} \crosslimits_{X\{p_2\}} \cdots \crosslimits_{X\{p_{n-1}\}} X\{p_{n-1} < p_n\}
    \end{equation*}
    is an equivalence.
    See \cite[Theorem 2.7.4]{arXiv:1807.03281}.
    Note that this characterization implies the functors $ \CatPcons \to \Spc $ given taking $ 0 $-links and $ 1 $-links are jointly conservative.
\end{recollection}


\subsection{Strata of localizations}\label{subsec:strata_of_localizations}

\noindent The purpose of this subsection is to prove a fundamental proposition about the types of localizations that appear in \enumref{thm:stability_properties_of_stratified_topoi}{3}.
To state it, we need to fix some notation.

\begin{notation}\label{ntn:restriction_subposet}
	Let $ F \colon \fromto{\Ccal}{P} $ be a functor from \acategory to a poset.
    \begin{enumerate}
    	\item\label{ntn:restriction_subposet.1} Given a subposet $ S\subset P $, we write $ F_S \colon \Ccal_S\to S $ for the basechange of $ F \colon \Ccal \to P $ to $ S $.

    	\item\label{ntn:restriction_subposet.2} We write $ W_P \subset \Mor(\Ccal) $ for the set of morphisms in $ \Ccal $ that $ F $ sends to equivalences (i.e., identities) in $ P $.
    \end{enumerate}
    By construction, functor $ F $ uniquely extends to a functor \smash{$ \fromto{\Ccal[W_P\inv]}{P} $}. 
\end{notation}

\begin{proposition}\label{prop:ff_and_localization}
	Let $ F \colon \Ccal \to P $ be a functor between \categories where $ P $ is a poset.
	Then: 
	\begin{enumerate}
		\item\label{prop:ff_and_localization.1} For each locally closed subposet $ S \subset P $, the induced functor $ \Ccal_S[W_S^{-1}] \to \Ccal[W_P^{-1}]_S $ is an equivalence.

		\item\label{prop:ff_and_localization.2} The induced functor \smash{$ \fromto{\Ccal[W_P\inv]}{P} $} is conservative.
		In particular, the \category \smash{$ \Ccal[W_P\inv] $} is idempotent complete.
	\end{enumerate}
\end{proposition}

Since localizations do not generally commute with pullbacks, \Cref{prop:ff_and_localization} is not completely formal.
To prove \Cref{prop:ff_and_localization}, we recall the following description of localizations.

\begin{recollection}[(localizations as pushouts)]\label{rec:localizations_as_pushouts}
	Let $ \Ccal $ be \acategory and let $ W \subset \Mor(\Ccal) $ be a class of morphisms. 
	The localization $ \Ccal[W\inv] $ can be defined as the pushout
	\begin{equation*}
        \begin{tikzcd}[column sep=2.5em]
            \coprod_{w \in W} [1] \arrow[r] \arrow[d] \arrow[dr, phantom, very near end, "\ulcorner", xshift=0.25em, yshift=-0.25em] & \Ccal \arrow[d] \\
            \coprod_{w \in W} \pt \arrow[r] & \Ccal[W\inv] \period
        \end{tikzcd}
    \end{equation*} 
    Here, the top horizontal functor is the induced by the functors $ \fromto{[1]}{\Ccal} $ that pick out each morphism $ w \in W $.
\end{recollection}

\noindent Hence \Cref{prop:ff_and_localization} amounts to commuting the pullback $ S \cross_P (-) $ past the pushout defining the localization $ \Ccal[W_P\inv] $.
To explain why we can do this, we recall some categorical notions.

\begin{recollection}
	A functor $ F \colon \fromto{\Ccal}{\Dcal} $ is an \defn{exponentiable fibration} if the right adjoint pullback functor
	\begin{equation*}
		\Ccal \cross_{\Dcal} (-) \colon \fromto{\Cat_{\infty,/\Dcal}}{\Cat_{\infty,/\Ccal}}
	\end{equation*}
	is also a left adjoint.
	Note that the class of exponentiable fibrations is closed under basechange.
\end{recollection}

\begin{example}[{\cite[Lemma 2.15]{MR4074276}}]\label{ex:cocartesian_fibrations_are_exponentiable}
	Cartesian and cocartesian fibrations are exponentiable fibrations.
	In particular, right and left fibrations are exponentiable fibrations.

	Recall that for any \category $ \Ccal $, the unique functor $ \fromto{\Ccal}{\pt} $ is both a cartesian and a cocartesian fibration.
	In this case, the right adjoint to $ \Ccal \cross (-) \colon \fromto{\Catinfty}{\Cat_{\infty,/\Ccal}} $ is given by sending $ \fromto{\Bcal}{\Ccal} $ to the \category of sections $ \Fun_{/\Ccal}(\Ccal,\Bcal) $.
\end{example}

\begin{lemma}\label{lem:locally_closed_subposets_are_exponentiable}
	Let $ P $ be a poset.
	\begin{enumerate}
		\item\label{lem:locally_closed_subposets_are_exponentiable.1} If $ U \subset P $ is an open subposet, then the inclusion $ \incto{U}{P} $ is a left fibration.

		\item\label{lem:locally_closed_subposets_are_exponentiable.2} If $ Z \subset P $ is a closed subposet, then the inclusion $ \incto{Z}{P} $ is a right fibration.

		\item\label{lem:locally_closed_subposets_are_exponentiable.3} If $ S \subset P $ is a locally closed subposet, then the inclusion $ \incto{S}{P} $ is an exponentiable fibration.
	\end{enumerate}
\end{lemma}

\begin{proof}
	For (1), first observe that the inclusion $ \incto{\set{1}}{\set{0 < 1}} $ is a left fibration.
	Let $ \chi_U \colon \fromto{P}{\set{0 < 1}} $ be the map sending $ U $ to $ 1 $ and $ P \sminus U $ to $ 0 $.
	Then we have a pullback square
	\begin{equation*}
	    \begin{tikzcd}[sep=2.25em]
	       U \arrow[dr, phantom, very near start, "\lrcorner", xshift=-0.25em, yshift=0.12em] \arrow[d] \arrow[r, hooked] & P  \arrow[d, "\chi_U"] \\ 
	       \set{1} \arrow[r, hooked] & \set{0 < 1} \period
	    \end{tikzcd}
	\end{equation*}
	The claim now follows from the fact that the class of left fibrations is closed under basechange.

	Item (2) follows from (1) by passing to opposite posets.
	Item (3) follows from (1), (2), \Cref{ex:cocartesian_fibrations_are_exponentiable}, and the fact that exponentiable fibrations are closed under composition.
\end{proof}

\begin{proof}[Proof of \Cref{prop:ff_and_localization}]
	For (1), consider the commutative diagram
	\begin{equation*}
        \begin{tikzcd}[column sep={12ex,between origins}, row sep={8ex,between origins}]
            \coprod_{w \in W_S} [1] \arrow[rr] \arrow[dd, hooked]  \arrow[dr] & & \Ccal_S \arrow[dd, hooked]  \arrow[dr] \\
            & \coprod_{w \in W_S} \pt \arrow[rr, crossing over] & & \Ccal[W_P\inv]_S \arrow[dd, hooked] \arrow[rr] & & S \arrow[dd, hooked]  \\
            \coprod_{w \in W_P} [1] \arrow[rr] \arrow[dr] & & \Ccal \arrow[dr] \\
            & \coprod_{w \in W_P} \pt \arrow[rr] \arrow[from=uu, crossing over, hooked] & & \Ccal[W_P\inv] \arrow[rr] & & P \period
        \end{tikzcd}
    \end{equation*}
    Notice that by \Cref{rec:localizations_as_pushouts}, the bottom face is a pushout.
    Moreover, all of the vertical faces are pullbacks.
    Since the inclusion $ \incto{S}{P} $ is an exponentiable fibration (\Cref{lem:locally_closed_subposets_are_exponentiable}), the top face is also a pushout; again applying \Cref{rec:localizations_as_pushouts} completes the proof.

    For (2), note that by \Cref{rec:characterization_of_conservative_functors_to_posets}, to show that \smash{$ \fromto{\Ccal[W_P\inv]}{P} $} is conservative, we need to show that each fiber $ \Ccal[W_{P}\inv]_{p} $ is \agroupoid.
	To see this, note that for each $ p \in P $, part (1) provides an identification 
	\begin{equation*}
		\Ccal[W_P^{-1}]_{p} \equivalent \Ccal_p[W_p\inv] \period
	\end{equation*}
	To complete the proof, observe that $ W_p $ is the set of \textit{all} morphisms in $ \Ccal_p $.
\end{proof}


\subsection{Compactness}\label{subsec:compactness_of_categories_with_a_conservative_functor_to_a_poset}

The goal of this subsection is to characterize the compact objects of $ \CatP $ as well as the compact objects of the full subcategory spanned by the conservative functors $ \fromto{\Ccal}{P} $ (\Cref{lem:compact_objects_of_CatD,cor:compactness_for_categories_with_a_conservative_functor_to_a_poset}). 
We then use this to explain why the assingment $ \goesto{\Ccal}{\Ccal[W_P\inv]} $ and pulling back to a locally closed subposet $ S \subset P $ both preserve compactness; see \Cref{prop:inverting_morphisms_fiberwise_preserves_compactness,prop:base_change_lc_preserves_finite_and_compact_categories}.
We begin by introducing some notation.

\begin{recollection}[{(finite \& compact \categories)}]\label{recollection_finite_category}
	Write $ \Catfin  \subset  \Catinfty $ for the smallest full subcategory closed under pushouts and containing the \categories $ \emptyset $, $ \pt $, and $ [1] $.
	\Acategory $ \Ccal $ is \defn{finite} if \smash{$ \Ccal \in \Catfin $}.
	In particular, \smash{$ \Catfin $} is closed under finite colimits in $ \Catinfty $.
	Equivalently, \acategory $ \Ccal $ is finite if and only if $ \Ccal $ is categorically equivalent to a simplicial set with only finitely many nondegenerate simplicies \cite[Corollary 2.3]{Volpe_Verdier_duality}.

	Importantly, the full subcategory \smash{$ \Catomega \subset \Catinfty $} of compact \categories is the smallest full subcategory containing \smash{$ \Catfin $} and closed under retracts.
\end{recollection}

We now establish some pleasant features of the inclusion \smash{$ \CatPcons \subset \CatP $}.
See \cite[\S2.2]{arXiv:1807.03281} for a related discussion.

\begin{observation}
	Let $ P $ be a poset.
	Then \Cref{prop:ff_and_localization} implies that the functor
	\begin{equation*}
		\fromto{\CatP}{\CatPcons}
	\end{equation*}
	given by the assignment $ \goesto{\Ccal}{\Ccal[W_P\inv]} $ is left adjoint to the inclusion.
\end{observation}

\noindent We introduce a more convenient notation for this left adjoint.

\begin{notation}
	Given a poset $ P $, write $ \Env_P \colon \fromto{\CatP}{\CatPcons} $ for the left adjoint to the inclusion.
\end{notation}
	
\begin{observation}\label{obs:iota_P}
	The inclusion \smash{$ \CatPcons \subset \CatP $} also admits a right adjoint
	\begin{equation*}
		\iotaup_P \colon \fromto{\CatP}{\CatPcons}
	\end{equation*}
	defined as follows.
	Given a functor $ F \colon \fromto{\Ccal}{P} $, let $ \iotaup_P(\Ccal) \subset \Ccal $ be the largest subcategory containing all objects such that the composite
	\begin{equation*}
		\begin{tikzcd}
			\iotaup_P(\Ccal) \arrow[r] & \Ccal \arrow[r, "F"] & P
		\end{tikzcd}
	\end{equation*}
	is conservative.
	Equivalently, $ \iotaup_P(\Ccal) \subset \Ccal $ is the subcategory containing all objects such that a morphism $ f \colon \fromto{x}{y} $ in $ \Ccal $ lies in $ \iotaup_P(\Ccal) $ if and only if one of the following disjoint conditions is satisfied:
	\begin{enumerate}
		\item The morphism $ f $ is an equivalence in $ \Ccal $.

		\item The elements $ F(x) $ and $ F(y) $ of the poset $ P $ are not equal.
	\end{enumerate}
\end{observation}

\begin{observation}\label{obs:iotaPC_to_C_have_same_maximal_subgroupoids}
	By definition, that the inclusion $ \fromto{\iotaup_P(\Ccal)}{\Ccal} $ restricts to an equivalence on maximal \subgroupoids.
	Note, moreover, that the inclusion $ \iotaup_P(\Ccal) \to \Ccal $ induces an equivalence on $ 0 $-links and $ 1 $-links, in the sense of \Cref{rec:categories_with_a_conservative_functor_to_a_poset}.
\end{observation}



In order to understand when $ \Env_P(\Ccal) $ is compact, we make use of the following two general fact:

\begin{recollection}[{\HTT{Proposition}{5.5.7.2}}]\label{rec:adjoint_triple_preservation_of_compactness}
	Let $ \adjto{\fupperstar}{\Dcal}{\Ccal}{\flowerstar} $ be an adjunction between \categories that admit filtered colimits.
	If $ \flowerstar $ preserves filtered colimits, then $ \fupperstar $ preserves compact objects.
	As a consequence, if $ \fupperstar $ admits a further left adjoint $ \flowersharp $, then $ \flowersharp $ preserves compact objects.
\end{recollection}


\begin{lemma}\label{lem:compact_objects_of_CatD}
	Let $ \Dcal $ be \acategory.
	An object $ F \colon \fromto{\Ccal}{\Dcal} $ of $ \Cat_{\infty,/\Dcal} $ is compact if and only if the \category $ \Ccal $ is compact in $ \Catinfty $. 
\end{lemma}

\begin{proof}
	Since the unique functor $ \fromto{\Dcal}{\pt} $ is an exponentiable fibration (\Cref{ex:cocartesian_fibrations_are_exponentiable}), \Cref{rec:adjoint_triple_preservation_of_compactness} shows that the forgetful functor $ \fromto{\Cat_{\infty,/\Dcal}}{\Catinfty} $ preserves compact objects.
	Hence all that remains to be proven is that if $ \Ccal \in \Catinfty $ is compact, then $ F \colon \fromto{\Ccal}{\Dcal} $ is compact in $ \Cat_{\infty,/\Dcal} $.
	For this, consider a filtered diagram $ \Dcal_{\bullet} \colon \fromto{A}{\Cat_{\infty,/\Dcal}} $.
	Note that we have a pullback square
	\begin{equation*}
	    \begin{tikzcd}[sep=2.25em]
	       \Map_{\Cat_{\infty,/\Dcal}}(\Ccal,\textstyle\colim_{\alpha \in A} \Dcal_{\alpha}) \arrow[dr, phantom, very near start, "\lrcorner", xshift=-3em, yshift=0.12em] \arrow[d] \arrow[r] & \Map_{\Catinfty}(\Ccal,\textstyle\colim_{\alpha \in A} \Dcal_{\alpha})  \arrow[d] \\ 
	       \set{F} \arrow[r] & \Map_{\Catinfty}(\Ccal,\Dcal) \period
	    \end{tikzcd}
	\end{equation*}
	Since $ \Ccal $ is compact in $ \Catinfty $, the natural map 
	\begin{equation*}
		\colim_{\alpha \in A} \Map_{\Catinfty}(\Ccal, \Dcal_{\alpha}) \to \Map_{\Catinfty}(\Ccal,\textstyle\colim_{\alpha \in A} \Dcal_{\alpha}) 
	\end{equation*}
	is an equivalence.
	The fact that colimits are universal in $ \Spc $ completes the proof.
\end{proof}

\begin{lemma}\label{lem:iotatP_preserves_filtered_colimits}
	Let $ P $ be a poset.
	Then:
	\begin{enumerate}
		\item\label{lem:iotatP_preserves_filtered_colimits.1} The functor $ \iotaup_P \colon \fromto{\CatP}{\CatPcons} $ preserves filtered colimits.

		\item\label{lem:iotatP_preserves_filtered_colimits.2} The inclusion $ \incto{\CatPcons}{\CatP} $ preserves compact objects.

		\item\label{lem:iotatP_preserves_filtered_colimits.3} The functor $ \Env_P \colon \fromto{\CatP}{\CatPcons} $ preserves compact objects.
	\end{enumerate}
\end{lemma}

\begin{proof}
	For (1), let $ \Ccal_{\bullet} \colon \fromto{A}{\CatP} $ be a filtered diagram.
	Note that by \Cref{rec:categories_with_a_conservative_functor_to_a_poset}, to prove that 
	\begin{equation*}
		\textstyle \colim_{\alpha \in A} \iotaup_P(\Ccal_{\alpha}) \longrightarrow \iotaup_P(\textstyle\colim_{\alpha \in A} \Ccal_{\alpha})
	\end{equation*}
	is an equivalence, it suffices to check that this map induces an equivalence on $ 0 $-links and $ 1 $-links.
	By \Cref{obs:iotaPC_to_C_have_same_maximal_subgroupoids}, it suffices to check that the natural map
	\begin{equation*}
		\colim_{\alpha \in A} \iotaup_P(\Ccal_{\alpha}) \longrightarrow \colim_{\alpha \in A} \Ccal_{\alpha}
	\end{equation*}
	induces an equivalence on $ 0 $-links and $ 1 $-links.
	For this, let $ \Sigma \subset P $ be a nonempty linearly ordered subset of cardinality $ \leq 2 $.
	Then $ \Sigma $ is a finite \category, hence \Cref{lem:compact_objects_of_CatD} shows that $ \Sigma $ is a compact object of $ \CatP $.
	Again applying \Cref{obs:iotaPC_to_C_have_same_maximal_subgroupoids}, we see that
	\begin{align*}
		\Map_{\CatP}(\Sigma,\textstyle \colim_{\alpha \in A} \iotaup_P(\Ccal_{\alpha})) &\equivalent \colim_{\alpha \in A} \Map_{\CatP}(\Sigma,\iotaup_P(\Ccal_{\alpha})) \\
		&\equivalence \colim_{\alpha \in A} \Map_{\CatP}(\Sigma,\Ccal_{\alpha}) \\
		&\equivalent \Map_{\CatP}(\Sigma,\textstyle \colim_{\alpha \in A} \Ccal_{\alpha}) \comma
	\end{align*}
	as desired.

	Finally, observe that \Cref{rec:adjoint_triple_preservation_of_compactness} shows that (1) implies (2) and (3).
\end{proof}

We can now give a characterization of the compact objects of $ \CatPcons $.

\begin{corollary}\label{cor:compactness_for_categories_with_a_conservative_functor_to_a_poset}
	Let $ P $ be a poset and let $ F \colon \fromto{\Ccal}{P} $ be a conservative functor from \acategory.
	Then the following are equivalent:
	\begin{enumerate}
		\item\label{cor:compactness_for_categories_with_a_conservative_functor_to_a_poset.1} The object $ F \colon \fromto{\Ccal}{P} $ of $ \CatPcons $ is compact.

		\item\label{cor:compactness_for_categories_with_a_conservative_functor_to_a_poset.2} The object $ F \colon \fromto{\Ccal}{P} $ of $ \CatP $ is compact.

		\item\label{cor:compactness_for_categories_with_a_conservative_functor_to_a_poset.3} The \category $ \Ccal $ is a compact object of $ \Catinfty $.
	\end{enumerate}
\end{corollary}

\begin{proof}
	The fact that both the inclusion \smash{$ \incto{\CatPcons}{\CatP} $} and its left adjoint $ \Env_P $ preserve compact objects (\Cref{lem:iotatP_preserves_filtered_colimits}) shows that (1) $ \Leftrightarrow $ (2).
	\Cref{lem:compact_objects_of_CatD} shows that (2) $ \Leftrightarrow $ (3).
\end{proof}

\begin{remark}
	\Cref{cor:compactness_for_categories_with_a_conservative_functor_to_a_poset} was mentioned in \cite[Remark 2.14]{arXiv:2206.02728}.
\end{remark}

Finiteness is also a well-behaved notion in $ \CatP $:

\begin{definition}
	Given \acategory $ \Dcal $, we say that an object $ F \colon \fromto{\Ccal}{\Dcal} $ of $ \Cat_{\infty,/\Dcal} $ is \defn{finite} if the \category $ \Ccal $ is finite.

	Given a poset $ P $, we say that an object $ F \colon \fromto{\Ccal}{P} $ of $ \CatPcons $ is finite if the \category $ \Ccal $ is finite.
\end{definition}

\begin{notation}
	For the sake of convenience, let us write $ [-1] \colonequals \emptyset $ for the empty poset.
\end{notation}

\begin{observation}\label{obs:finite_and_compact_objects_of_CatD}
	Let $ \Dcal $ be \acategory.
	Then the full subcategory
	\begin{equation*}
		\CatfinD \subset \Cat_{\infty,/\Dcal}
	\end{equation*}
	spanned by the finite objects is the smallest subcategory closed under pushouts and containing all objects of the form $ \sigma \colon \fromto{[n]}{\Dcal} $ where $ -1 \leq n \leq 1 $.
	Similarly,
	\begin{equation*}
		\CatomegaD \subset \Cat_{\infty,/\Dcal}
	\end{equation*}
	is the smallest full subcategory containing $ \CatfinD $ and closed under retracts.
\end{observation}

We conclude by recording some important operations that preserve finiteness and compactness.

\begin{proposition}\label{prop:inverting_morphisms_fiberwise_preserves_compactness}
	Let $ F \colon \fromto{\Ccal}{P} $ be a functor from \acategory to a poset.
	If $ \Ccal $ is a finite (resp., compact) object of $ \Catinfty $, then the \category
	\begin{equation*}
		\Env_P(\Ccal) = \Ccal[W_P\inv]
	\end{equation*}
	is a finite (resp., compact) object of $ \Catinfty $. 
\end{proposition}

\begin{proof}
	In light of \Cref{obs:finite_and_compact_objects_of_CatD}, it suffices to show that $ \Env_{P} $ preserves finite objects.
	Moreover, to prove this, it suffices to show that for $ -1 \leq n \leq 1 $ and each map of posets $ \sigma \colon \fromto{[n]}{P} $, the localization $ \End_P([n]) $ is finite.
	If $ n = -1 $ or $ n = 0 $, then $ \Env_P([n]) = [n] $, so the claim is clear.

	If $ n = 1 $, then there are two cases.
	First, if the map $ \sigma \colon \fromto{[1]}{P} $ is constant, then the class $ W_{P} $ consists of all morphisms in $ P $, hence $ \Env_{P}([1]) \equivalent \pt $ is finite.
	Second, if the map $ \sigma \colon \fromto{[1]}{P} $ is \textit{not} constant, then the class $ W_{P} $ consists of only the identity morphisms in $ P $, hence $ \Env_{P}([1]) \equivalent [1] $ is finite.
\end{proof}

\begin{proposition}\label{prop:base_change_lc_preserves_finite_and_compact_categories}
	Let $ P $ be a poset and let $ S \subset P $ be a locally closed subposet.
	Then the basechange functor 
	\begin{equation*}
		S \cross_P (-) \colon \CatP \to  \Cat_{\infty,/S}
	\end{equation*} 
	preserves finite and compact objects.
\end{proposition}

\begin{proof}
	Since the inclusion $ \incto{S}{P} $ is an exponentiable fibration (\Cref{lem:locally_closed_subposets_are_exponentiable}), the functor $ S \cross_P (-) $ preserves colimits.
	Hence by \Cref{obs:finite_and_compact_objects_of_CatD}, it suffices to prove that $ S \cross_P (-) $ preserves finite objects.
	Moreover, to prove this, it suffices to show that for $ -1 \leq n \leq 1 $ and each map of posets $ \sigma \colon \fromto{[n]}{P} $, the basechange $ S \cross_P [n] $ is finite.
	To conclude, observe that since $ S \subset P $ is locally closed, $ S \cross_P [n] \subset [n] $ is also locally closed; hence, there exists $ -1 \leq m \leq n $ such that $ S \cross_P [n] \equivalent [m] $.
\end{proof}

The following application of \Cref{prop:base_change_lc_preserves_finite_and_compact_categories} is not needed in the present paper, but is quite useful:

\begin{lemma}
	Let $ \Ccal $ and $ \Dcal $ be \categories.
	Then the join $ \Ccal \join \Dcal $ is finite (resp., compact) if and only if both $ \Ccal $ and $ \Dcal $ are finite (resp., compact).
\end{lemma}

\begin{proof}
	By definition, the join $ \Ccal \join \Dcal $ is the colimit in $ \Catinfty $ of the diagram
	\begin{equation*}
		\begin{tikzcd}[column sep={8ex,between origins}, row sep={8ex,between origins}]
			 & \Ccal \cross \Dcal \cross \{0\} \arrow[dl] \arrow[dr, hooked] & & \Ccal \cross \Dcal \cross \{1\} \arrow[dl, hooked'] \arrow[dr] & \\
			\Ccal & & \Ccal \cross \Dcal \cross {[1]} & & \Dcal \comma
		\end{tikzcd}
	\end{equation*}
	where the outermost functors are the projections.
	Furthermore, the unique functors $ \fromto{\Ccal}{\{0\}} $ and $ \fromto{\Dcal}{\{1\}} $ induce a functor
	\begin{equation*}
		\begin{tikzcd}
			\Ccal \join \Dcal \arrow[r] & \{0\} \join \{1\} \equivalent [1] 
		\end{tikzcd}
	\end{equation*}
	with fibers $ (\Ccal \join \Dcal)_0 \equivalent \Ccal $ and $ (\Ccal \join \Dcal)_1 \equivalent \Dcal $.
	In particular, the forward implication follows from the fact that finite (resp., compact) \categories are stable under finite products and finite colimits.
	The reverse implcation follows from \Cref{prop:base_change_lc_preserves_finite_and_compact_categories} applied to the induced functor $ \fromto{\Ccal \join \Dcal}{[1]} $.
\end{proof}

\noindent Of particular interest are cones:
	
\begin{corollary}\label{compactness_and_cone}
	Let $ \Ccal $ be \acategory.
	Then $ \Ccal $ is finite (resp., compact) if and only if the cone $ \Ccal^{\lhd} $ is finite (resp., compact). 
\end{corollary}


\section{Complements on \texorpdfstring{$\infty$}{∞}-topoi}\label{appendix:complements_on_topoi}

The purpose of this appendix is to prove some fundamental results about \topoi that are used in the main body of the paper.
In \cref{subsec:open_and_closed_subtopoi}, we recall the basics of étale geometric morphisms as well as open and closed immersions of \topoi.
In \cref{subsec:hypercompleteness_and_etale_geometric_morphisms}, we explain how hypercompletion interacts with étale geometric morphisms.
In \cref{subsec:the_hypercompletion_of_a_recollement}, we prove that the hypercompletion of a recollement of \topoi is still a recollement (\Cref{prop:hypercompletions_of_recollements}).
We then use this to explain how hypercompletion interacts with locally closed immersions of \topoi (\Cref{cor:pulling_back_a_locally_closed_immersion_along_hypercompletion,lem:hypercompletion_commutes_with_taking_strata}).


\subsection{Open and closed subtopoi}\label{subsec:open_and_closed_subtopoi}

In this subsection, we recall the notions of open and closed immersions of \topoi and how they give rise to recollements.
In order to discuss open immersions, we start with the more general notion of \textit{étale} geometric morphisms.
For more background on étale geometric morphisms, the reader should consult \cite[\HTTsubsec{6.3.5}]{HTT}.

\begin{recollection}[(étale geometric morphisms)]
	Let $ \X $ be \atopos and $ U \in \X $.
	Then the overcategory $ \X_{/U} $ is \atopos.
	Moreover, the forgetful functor $ \plowersharp \colon \fromto{\X_{/U}}{\X} $ admits a right adjoint $ \pupperstar \colon \fromto{\X}{\X_{/U}} $ given by the assignment $ \goesto{X}{X \cross U} $.
	Since colimits are universal in $ \X $, the functor $ \pupperstar $ admits a further right adjoint $ \plowerstar \colon \fromto{\X_{/U}}{\X} $.
	See \HTT{Proposition}{6.3.5.1}.
	We always regard the \topos $ \X_{/U} $ as \atopos over $ \X $ via the natural geometric morphism $ \plowerstar \colon \fromto{\X_{/U}}{\X} $.

	Let $ \elowerstar \colon \fromto{\W}{\X} $ be a geometric morphism of \topoi.
	Then the following conditions are equivalent:
	\begin{enumerate}
		\item There exists an object $ U \in \X $ and an equivalence $ \equivto{\W}{\X_{/U}} $ of \topoi over $ \X $.

		\item The functor $ \eupperstar $ admits a left adjoint $ \elowersharp \colon \fromto{\W}{\X} $ and the induced functor
		\begin{equation*}
			\elowersharp \colon \fromto{\W}{\X_{/\elowersharp(1_{\W})}}
		\end{equation*}
		is an equivalence of \categories.

		\item The functor $ \eupperstar $ admits a conservative left adjoint $ \elowersharp \colon \fromto{\W}{\X} $ and for all maps $ \fromto{X}{Z} $ in $ \X $, objects $ Y \in \W $, and maps $ \fromto{\elowersharp(Y)}{Z} $, the natural map
		\begin{equation}\label{eq:etale_projection_formula}
			\elowersharp\paren{\eupperstar(X) \crosslimits_{\eupperstar(Z)} Y} \to X \cross_{Z} \elowersharp(Y) 
		\end{equation}
		is an equivalence.
		This equivalence is referred to as the \textit{étale projection formula}.
	\end{enumerate}
	See \HTT{Proposition}{6.3.5.11}.
	We call a geometric morphism satisfying these equivalent conditions an \defn{étale} geometric morphism.
\end{recollection}

\textit{Open immersions} are a special case of étale geometric morphisms:

\begin{lemma}\label{lem:equivalent_characterization_of_open_immersions}
	Let $ \jlowerstar \colon \fromto{\U}{\X} $ be a geometric morphism of \topoi.
	Then the following conditions are equivalent:
	\begin{enumerate}
		\item There exists a $ (-1) $-truncated object $ U \in \X $ and an equivalence $ \equivto{\U}{\X_{/U}} $ of \topoi over $ \X $.

		\item The geometric morphism $ \jlowerstar \colon \fromto{\U}{\X} $ is étale and $ \jlowersharp(1_{\U}) \in \X $ is $ (-1) $-truncated.

		\item The geometric morphism $ \jlowerstar \colon \fromto{\U}{\X} $ is étale and the functor $ \jlowersharp $ is fully faithful.

		\item The geometric morphism $ \jlowerstar \colon \fromto{\U}{\X} $ is étale and the functor $ \jlowerstar $ is fully faithful.
	\end{enumerate}
\end{lemma}

\begin{proof}
	Clearly (1) $ \Leftrightarrow $ (2).
	The equivalence (3) $ \Leftrightarrow $ (4) is a consequence of the fact that given a triple of adjoints $ \flowersharp \leftadjoint \fupperstar \leftadjoint \flowerstar $ between arbitrary \categories, the extreme left adjoint $ \flowersharp $ is fully faithful if and only if the extreme right adjoint $ \flowerstar $ is fully faithful.

	To see that (2) $ \Leftrightarrow $ (3), we first make some observations; assume that $ \jlowerstar $ is étale.
	Note that $ \jlowersharp $ is fully faithful if and only if the unit $ \id{\Ucal} \to \jupperstar\jlowersharp $ is an equivalence.
	Since $ \jlowersharp $ is conservative, this is the case if and only if the induced natural transformation $ \jlowersharp \to \jlowersharp \jupperstar \jlowersharp $ is an equivalence.
	Let $ V \in \Ucal $; by the étale projection formula \eqref{eq:etale_projection_formula} applied to $ X = \jlowersharp(V) $, $ Y = 1_{\Ucal} $, and $ Z = 1_{\Xcal} $, we see that
	\begin{equation*}
		\jlowersharp \jupperstar \jlowersharp(V) \equivalent \jlowersharp(V) \cross \jlowersharp(1_{\Ucal}) \period
	\end{equation*} 
	Moreover, unwinding definitions shows that, under this identification, the map $ \jlowersharp(V) \to \jlowersharp \jupperstar \jlowersharp(V) $ induced by the unit is identified with the graph
	\begin{equation*}
		\gamma_V \colon \jlowersharp(V) \to \jlowersharp(V) \cross \jlowersharp(1_{\Ucal}) 
	\end{equation*} 
	of the map $ \jlowersharp(V) \to \jlowersharp(1_{\Ucal}) $ induced by the unique map $ V \to 1_{\Ucal} $.
	Also note that the map $ \gamma_V $ fits into a pullback square
	\begin{equation*}
	    \begin{tikzcd}[sep=2.25em]
	       \jlowersharp(V) \arrow[d] \arrow[r, "\gamma_V"] \arrow[dr, phantom, very near start, "\lrcorner", xshift=-0.5em, yshift=0.25em] & \jlowersharp(V) \cross \jlowersharp(1_{\Ucal}) \arrow[d]  \\ 
	       \jlowersharp(1_{\Ucal}) \arrow[r, "\Delta"'] & \jlowersharp(1_{\Ucal}) \cross \jlowersharp(1_{\Ucal}) \period
	    \end{tikzcd}
	\end{equation*}

	To see that (2) $ \Rightarrow $ (3), note that if $ \jlowersharp(1_{\Ucal}) $ is $ (-1) $-truncated, then the diagonal $ \Delta \colon \jlowersharp(1_{\Ucal}) \to \jlowersharp(1_{\Ucal}) \cross \jlowersharp(1_{\Ucal}) $ is an equivalence; hence by pullback, for each $ V \in \Ucal $, the map $ \gamma_V $ is an equivalence.
	Conversely, note that $ \gamma_{1_{\Ucal}} $ is the diagonal, so if $ \jlowersharp $ is fully faithful, $ \Delta $ is an equivalence, i.e., $ \jlowersharp(1_{\Ucal}) $ is $ (-1) $-truncated.
\end{proof}

\begin{recollection}[(open immersions)]
	We call a geometric morphism satisfying the equivalent conditions of \Cref{lem:equivalent_characterization_of_open_immersions} an \defn{open immersion} of \topoi.
	For open immersions of \topoi, we write $ \jlowershriek \colonequals \jlowersharp $.
\end{recollection}

\begin{recollection}[(closed immersions)]
	Let $ \X $ be \atopos and let $ U \in \X $ be a $ (-1) $-truncated object.
	We write
	\begin{equation*}
		\X_{\sminus U} \subset \X
	\end{equation*}
	for the full subcategory spanned by those objects $ F $ such that the projection $ \pr_2 \colon \fromto{F \cross U}{U} $ is an equivalence. 
	The inclusion $ \X_{\sminus U} \subset \X $ is accessible and admits a left exact left adjoint \HTT{Proposition}{7.3.2.3}.
	In particular, $ \X_{\sminus U} $ is \atopos and the inclusion $ \incto{\X_{\sminus U}}{\X} $ is a geometric morphism.
	We call the \topos $ \X_{\sminus U} $ the \defn{closed complement} of the open subtopos $ \X_{/U} $.

	We say that a geometric morphism of \topoi $ \ilowerstar \colon \fromto{\Z}{\X} $ is a \defn{closed immersion} if there exists a $ (-1) $-truncated object $ U \in \X $ such that $ \ilowerstar $ factors through $ \X_{\sminus U} $ and restricts to an equivalence $ \ilowerstar \colon \equivto{\Z}{\X_{\sminus U}} $.
\end{recollection}

\begin{definition}
	Let $ \flowerstar \colon \fromto{\X}{\Y} $ be a geometric morphism of \topoi.
	We say that $ \flowerstar $ is a \defn{locally closed immersion} if there exists a factorization $ \flowerstar \equivalent \jlowerstar \ilowerstar $ where $ \ilowerstar $ is a closed immersion and $ \jlowerstar $ is an open immersion. 
\end{definition}

\begin{recollection}\label{rec:locally_closed_immersions_of_topological_spaces_induce_locally_closed_immersions_of_topoi}
	Let $ X $ be a topological space and let $ j \colon \incto{U}{X} $ be an open subspace with closed complement $ i \colon \incto{Z}{X} $.
	Also write $ U \in \Sh(X) $ for the sheaf represented by the open subset $ U \subset X $.
	Then:
	\begin{enumerate}
		\item The geometric morphism $ \jlowerstar \colon \incto{\Sh(U)}{\Sh(X)} $ is an open immersion that identifies $ \Sh(U) $ with $ \Sh(X)_{/U} $.

		\item The geometric morphism $ \ilowerstar \colon \incto{\Sh(Z)}{\Sh(X)} $ is a closed immersion that identifies $ \Sh(Z) $ with $ \Sh(X)_{\sminus U} $.
		See \HTT{Corollary}{7.3.2.10}.
	\end{enumerate}
	As a consequence, locally closed immersions of topological spaces induce locally closed immersions of \topoi of sheaves.
\end{recollection}

The key feature of open and closed immersions is that they give rise to recollements:

\begin{recollection}[{(open-closed recollement \HAa{Proposition}{A.8.15})}]\label{rec:open-closed_recollement}
	Let $ \X $ be \atopos and $ U \in \X $ a $ (-1) $-truncated object.
	Write $ \ilowerstar \colon \incto{\X_{\sminus U}}{\X} $ and $ \jlowerstar \colon \fromto{\X_{/U}}{\X} $ for the complementary closed and open geometric morphisms.
	Then the functors
	\begin{equation*}
		\iupperstar \colon \fromto{\X}{\X_{\sminus U}} \andeq \jupperstar \colon \fromto{\X}{\X_{/U}}
	\end{equation*}  
	exhibit $ \X $ as the recollement of $ \X_{\sminus U} $ and $ \X_{/U} $.
\end{recollection}

\noindent In light of \Cref{rec:locally_closed_immersions_of_topological_spaces_induce_locally_closed_immersions_of_topoi,rec:open-closed_recollement}, we see: 

\begin{example}\label{ex:open-closed_recollement_for_sheaves_on_topological_spaces}
	Let $ X $ be a topological space and let $ i \colon \incto{Z}{X} $ be a closed subspace with open complement $ j \colon \incto{U}{X} $.
	Then the functors
	\begin{equation*}
		\iupperstar \colon \fromto{\Sh(X)}{\Sh(Z)} \andeq \jupperstar \colon \fromto{\Sh(X)}{\Sh(U)}
	\end{equation*} 
	exhibit $ \Sh(X) $ as the recollement of $ \Sh(Z) $ and $ \Sh(U) $. 
\end{example}

Étale geometric morphisms and closed immersions also behave well under basechange.

\begin{proposition}\label{prop:formula_for_pullbacks_along_etale_morphisms_and_closed_immersions}
	Let $ \flowerstar \colon \fromto{\X}{\Y} $ be a geometric morphism of \topoi and let $ V \in \Y $.
	Then:
	\begin{enumerate}
		\item\label{prop:formula_for_pullbacks_along_etale_morphisms_and_closed_immersions.1} The induced square
		\begin{equation*}
			\begin{tikzcd}
				\X_{/\fupperstar(V)} \arrow[r] \arrow[d] & \X \arrow[d, "\flowerstar"] \\ 
				\Y_{/V} \arrow[r] & \Y
			\end{tikzcd}
		\end{equation*}
		is a pullback square in $ \RTop $.

		\item\label{prop:formula_for_pullbacks_along_etale_morphisms_and_closed_immersions.2} If $ V $ is $ (-1) $-truncated, then the induced square
		\begin{equation*}
			\begin{tikzcd}
				\X_{\sminus \fupperstar(V)} \arrow[r, hooked] \arrow[d] & \X \arrow[d, "\flowerstar"] \\ 
				\Y_{\sminus V} \arrow[r, hooked] & \Y
			\end{tikzcd}
		\end{equation*}
		is a pullback square in $ \RTop $.
	\end{enumerate}
\end{proposition}

\begin{proof}
	For (1), see \HTT{Remark}{6.3.5.8}.
	For (2), see \HTT{Proposition}{7.3.2.12}.
\end{proof}

\begin{nul}
	As a consequence of \Cref{prop:formula_for_pullbacks_along_etale_morphisms_and_closed_immersions} the properties being étale, an open immersion, a closed immersion, or a locally closed immersion are all stable under basechange in $ \RTop $.
\end{nul}

In general, the functor sending a topological space $ X $ to the \topos $ \Sh(X) $ does not preserve pullbacks. 
However, \Cref{prop:formula_for_pullbacks_along_etale_morphisms_and_closed_immersions} implies that the assignment $ \goesto{X}{\Sh(X)} $ \textit{does} preserve pullbacks along locally closed immersions: 

\begin{corollary}\label{cor:Sh_preserves_pullbacks_along_locally_closed_embeddings}
	Let
	\begin{equation*}
	    \begin{tikzcd}[sep=2.25em]
	       S \arrow[d] \arrow[r, hooked, "\ibar"] \arrow[dr, phantom, very near start, "\lrcorner", xshift=-0.25em, yshift=0.25em] & X \arrow[d, "f"]  \\ 
	       T \arrow[r, hooked, "i"'] & Y
	    \end{tikzcd}
	\end{equation*}
	be a pullback square of topological spaces where $ i $ is a locally closed immersion.
	Then the induced square of \topoi
	\begin{equation*}
	    \begin{tikzcd}[sep=2.25em]
	       \Sh(S) \arrow[d] \arrow[r, hooked, "\ibarlowerstar"] & \Sh(X) \arrow[d, "\flowerstar"]  \\ 
	       \Sh(T) \arrow[r, hooked, "\ilowerstar"'] & \Sh(Y)
	    \end{tikzcd}
	\end{equation*}
	is a pullback square in $ \RTop $.
\end{corollary}

\begin{proof}
	Note that by factoring $ i $ as a closed immersion followed by an open immersion, it suffices to treat the cases of closed immersions and open immersions separately.
	Since $ S = f^{-1}(T) $, the claim is immediate from \Cref{rec:locally_closed_immersions_of_topological_spaces_induce_locally_closed_immersions_of_topoi} and \Cref{prop:formula_for_pullbacks_along_etale_morphisms_and_closed_immersions}.
\end{proof}


\subsection{Hypercompleteness \& étale geometric morphisms}\label{subsec:hypercompleteness_and_etale_geometric_morphisms}

The purpose of this subsection is to prove the following characterization the hypercomplete objects of the slice \topos over a hypercomplete object.

\begin{proposition}\label{prop:hypercomplete_objects_of_a_slice}
	Let $ \X $ be \atopos and let $ U \in \X $ be a hypercomplete object.
	Write $ \elowerstar \colon \fromto{\X_{/U}}{\X} $ for the natural geometric morphism.
	For an object $ [p \colon \fromto{X}{U}] \in \X_{/U} $, the following are equivalent:
	\begin{enumerate}
		\item The object $ p \colon \fromto{X}{U} $ is a hypercomplete object of $ \X_{/U} $.

		\item The object $ X $ is a hypercomplete object of $ \X $.
	\end{enumerate}

	In particular, there is a natural identification
	\begin{equation*}
		(\Xhyp)_{/U} = (\X_{/U})^{\hyp}
	\end{equation*}
	as full subcategories of $ \X_{/U} $.
\end{proposition}

\begin{corollary}\label{cor:etale_topos_over_a_hypercomplete_topos_is_hypercomplete}
	Let $ \X $ be \atopos and let $ U \in \X $.
	If $ \X $ is hypercomplete, then the \topos $ \X_{/U} $ is hypercomplete.
\end{corollary}


To prove \Cref{prop:hypercomplete_objects_of_a_slice}, we need a few technical lemmas.
The first is a slight refinement of the statement of \HAa{Lemma}{A.2.6}:

\begin{lemma}\label{lem:lowersharp_preserves_n-connectedness}
	Let $ \elowerstar \colon \fromto{\W}{\X} $ be a geometric morphism of \topoi.
	Assume that $ \eupperstar $ admits a left adjoint $ \elowersharp \colon \fromto{\W}{\X} $.
	Then: 
	\begin{enumerate}
		\item For each $ -2 \leq n \leq \infty $, the functor $ \elowersharp $ preserves $ n $-connected maps.

		\item The functor $ \eupperstar \colon \fromto{\X}{\W} $ preserves hypercomplete objects.
	\end{enumerate}
\end{lemma}

\begin{lemma}\label{lem:conservative_functors_reflect_limits}
	Let $ F \colon \fromto{\Ccal}{\Dcal} $ be a functor between \categories.
	\begin{enumerate}
		\item\label{lem:conservative_functors_reflect_limits.1} Let $ \Ical $ be \acategory.
		Assume that $ \Ccal $ and $ \Dcal $ admit $ \Ical $-shaped colimits and that $ F $ preserves $ \Ical $-shaped colimits.
		If $ F $ is conservative, then $ F $ reflects $ \Ical $-shaped colimits.

		\item\label{lem:conservative_functors_reflect_limits.2} Assume that $ \Ccal $ and $ \Dcal $ admit pullbacks and geometric realizations of simplicial objects and that $ F $ preserves pullbacks and geometric realizations.
		If $ F $ is conservative, then $ F $ reflects effective epimorphisms.
	\end{enumerate}
\end{lemma}

\begin{proof}
	For (1), let $ X_{\bullet} \colon \fromto{\Ical^{\rhd}}{\Ccal} $ be a diagram, and assume that the composite diagram $ F \of X_{\bullet} \colon \fromto{\Ical^{\rhd}}{\Dcal} $ is a colimit diagram.
	Write $ X_{\infty} $ for the value of the cone point and let $ \lambda \colon \fromto{\colim_{i \in \Ical} X_i}{X_{\infty}} $ denote the natural map.
	Then $ F(\lambda) $ factors as a composite of natural maps
	\begin{equation*}
		\begin{tikzcd}
			F\paren{\colim_{i \in \Ical} X_i} \arrow[r] & \colim_{i \in \Ical} F(X_i) \arrow[r] & F(X_{\infty}) \period
		\end{tikzcd}
	\end{equation*}
	Since $ F $ preserves colimits, the left-hand map is an equivalence; since $ F \of X_{\bullet} $ is a colimit diagram, the right-hand map is also an equivalence.
	Since $ F $ is conservative, we deduce that $ \lambda $ is an equivalence, i.e., that $ X_{\bullet} $ is a colimit diagram, as desired.

	Item (2) is immediate from the definition of an effective epimorphism combined with item (1) and its dual.
\end{proof}

\begin{lemma}\label{lem:n-connectedness_can_be_checked_on_a_conservative_family_of_geometric_morphisms}
	Let $ \X $ be \atopos and let $ \{\fupperstar_{\alpha} \colon \fromto{\X}{\X_{\alpha}}\}_{\alpha \in A} $ be a jointly conservative family of functors between \topoi that each preserve pullbacks and geometric realizations of simplicial objects.
	Let $ -2 \leq n \leq \infty $ and let $ \phi \colon \fromto{U}{V} $ be a morphism in $ \X $.
	Then the following are equivalent:
	\begin{enumerate}
		\item The morphism $ \phi $ is $ n $-connected.

		\item For each $ \alpha \in A $, the morphism $ \fupperstar_{\alpha}(\phi) $ is $ n $-connected.
	\end{enumerate}
\end{lemma}

\begin{proof}
	Since functors that preserve pullbacks and geometric realizations of simplicial objects preserve $ n $-connectedness, (1) $ \Rightarrow $ (2).
	For the implication (2) $ \Rightarrow $ (1), first note a morphism $ \phi $ is $ \infty $-connected map if and only if for each $ n < \infty $, the morphism $ \phi $ is $ n $-connected.
	So it suffices to treat the case of finite $ n $.
	Write $ \Y $ for the product of \categories $ \prod_{\alpha \in A} \X_{\alpha} $ and $ \fupperstar \colon \fromto{\X}{\Y} $ for the functor induced by the functors $ \fupperstar_{\alpha} \colon \fromto{\X}{\X_{\alpha}} $ by the universal property of the product.
	Note that $ \Y $ is \atopos and since limits and colimits in $ \Y $ are computed levelwise, $ \fupperstar $ also preserves pullbacks and effective epimorphisms.
	Moreover, the statement (2) is equivalent to the statement:
	\begin{enumerate}
		\setcounter{enumi}{2}

		\item The morphism $ \fupperstar(\phi) $ is $ n $-connected.
	\end{enumerate}
	So we instead prove that (3) $ \Rightarrow $ (1).

	We prove the claim by induction on $ n $.
	The case $ n = -2 $ is clear; every morphism is $ (-2) $-connected.
	For the case $ n = -1 $, recall that a morphism $ \phi $ is $ (-1) $-connected if and only if $ \phi $ is an effective epimorphism.
	The claim now follows from \enumref{lem:conservative_functors_reflect_limits}{2}.

	For the inductive step, assume that $ n \geq 0 $, and that we know that for all $ k \leq n $, the functor $ \fupperstar \colon \fromto{\X}{\Y} $ reflects $ k $-connectedness.
	Let $ \phi \colon \fromto{U}{V} $ be a morphism of $ \X $ such that $ \fupperstar(\phi) $ is $ n $-connected.
	That is $ \fupperstar(\phi) $ is an effective epimorphism and the diagonal
	\begin{equation*}
		\Delta_{\fupperstar(\phi)} \colon \fromto{\fupperstar(U)}{\fupperstar(U) \crosslimits_{\fupperstar(V)} \fupperstar(U)}
	\end{equation*}
	is $ (n-1) $-connected.
	By the base case, $ \phi $ is an effective epimorphism.
	Moreover, since $ \fupperstar $ preserves pullbacks,
	\begin{equation*}
		\Delta_{\fupperstar(\phi)} \equivalent \fupperstar(\Delta_{\phi}) \period
	\end{equation*}
	The inductive hypothesis then show that $ \Delta_{\phi} $ is $ (n-1) $-connected.
	Thus $ \phi $ is $ n $-connected, as desired.
\end{proof}

\begin{corollary}\label{cor:forgetful_functor_reflects_n-connectedness}
	Let $ \elowerstar \colon \fromto{\W}{\X} $ be an étale geometric morphism of \topoi and let $ \phi $ be a morphism in $ \W $.
	Then for each $ -2 \leq n \leq \infty $, the morphism $ \phi $ in $ \W $ is $ n $-connected if and only if $ \elowersharp(\phi) $ is $ n $-connected.
\end{corollary}

\begin{proof}
	Since the forgetful functor $ \elowersharp \colon \fromto{\W}{\X} $ is a conservative left adjoint that preserves pullbacks, this is a special case of \Cref{lem:n-connectedness_can_be_checked_on_a_conservative_family_of_geometric_morphisms}.
\end{proof}

Now we are ready to prove \Cref{prop:hypercomplete_objects_of_a_slice}.

\begin{proof}[Proof of \Cref{prop:hypercomplete_objects_of_a_slice}]
	We start by proving that (1) $ \Rightarrow $ (2).
	Let $ \phi \colon \fromto{V}{V'} $ be an $ \infty $-connected map in $ \X $.
	We need to show that $ \Map_{\X}(-,X) $ inverts $ \phi $.
	Consider the commutative square
	\begin{equation*}\label{eq:mapping_spaces_in_slices}
		\begin{tikzcd}[sep=3em]
	      	\Map_{\X}(V',X) \arrow[d, "-\of \phi"'] \arrow[r, "p \of -"] & \Map_{\X}(V',U) \arrow[d, "-\of \phi"] \\ 
	    	\Map_{\X}(V,X) \arrow[r, "p \of -"'] & \Map_{\X}(V,U) \period
		\end{tikzcd}
	\end{equation*}
	Since $ \phi $ is $ \infty $-connected and $ U $ is hypercomplete, the right-hand vertical map is an equivalence.
	Thus to show that the left-hand vertical map is an equivalence, it suffices to show that for each map $ q \colon \fromto{V'}{U} $, the induced map on horizontal fibers is an equivalence.

	For this, regard $ V $ and $ V' $ as objects of $ \X_{/U} $ via the structure maps $ q\phi $ and $ q $, respectively; then $ \phi $ defines a map
	\begin{equation*}
		[q\phi \colon V \to U] \to [q \colon V' \to U]
	\end{equation*}
	in $ \X_{/U} $.
	By the definition of the mapping spaces in an overcategory, we have a commutative square
	\begin{equation}\label{eq:mapping_spaces_in_slices}
		\begin{tikzcd}[sep=3em]
	      	\Map_{\X_{/U}}(V',X) \arrow[d, "-\of \phi"'] \arrow[r, "\sim"{yshift=-0.25em}] &  \{q\} \crosslimits_{\Map_{\X}(V',U)} \Map_{\X}(V',X) \arrow[d] \\
	      	\Map_{\X_{/U}}(V,X) \arrow[r, "\sim"{yshift=-0.25em}]  & \{q\phi \} \crosslimits_{\Map_{\X}(V,U)} \Map_{\X}(V,X) \comma
		\end{tikzcd}
	\end{equation}
	where the horizontal maps are equivalences and the vertical maps are given by precomposition with $ \phi $.
	Since $ \phi $ is an $ \infty $-connected map in $ \X $, by \Cref{cor:forgetful_functor_reflects_n-connectedness}, $ \phi $ is also an $ \infty $-connected map when regarded as a map $ \fromto{V}{V'} $ in $ \X_{/U} $.
	Since $ X $ is a hypercomplete object of $ \X_{/U} $, we deduce that the left-hand vertical map in \eqref{eq:mapping_spaces_in_slices} is an equivalence.
	Thus the right-hand vertical map is also an equivalence, as desired.

	Now we prove that (2) $ \Rightarrow $ (1).
	Assume that $ X $ is hypercomplete when regarded as an object of $ \X $.
	Let $ \phi \colon \fromto{V}{V'} $ be an $ \infty $-connected map in $ \X_{/U} $, and write $ q \colon \fromto{V}{V'} $ for the structure map.
	We need to show that the functor \smash{$ \Map_{\X_{/U}}(-,X) $} inverts $ \phi $.
	Again consider the square \eqref{eq:mapping_spaces_in_slices}.
	Since $ \phi $ is an $ \infty $-connected map of $ \X_{/U} $, \Cref{cor:forgetful_functor_reflects_n-connectedness} shows that $ \phi $ is also $ \infty $-connected when regarded as a map of $ \X $.
	Since $ U $ and $ X $ are hypercomplete when regarded as objects of $ \X $, the right-hand vertical map in \eqref{eq:mapping_spaces_in_slices} is an equivalence; hence the left-hand vertical map is also an equivalence, as desired.
\end{proof}


\subsection{The hypercompletion of a recollement}\label{subsec:the_hypercompletion_of_a_recollement}

This subsection has two goals.
The first is to show that the hypercompletion of a recollement of \topoi is still a recollement (\Cref{prop:hypercompletions_of_recollements}).
The second is to show that hypercompletion preserves pullbacks along locally closed immersions of \topoi (\Cref{prop:hypercompletion_commutes_with_pullback_along_locally_closed_immersions}). 

We begin by using \Cref{prop:hypercomplete_objects_of_a_slice} to describe the hypercomplete objects of a locally closed subtopos.
To do this, we first observe that the pushforward along a closed immersion preserves $ \infty $-connectedness and detects hypercompleteness. 

\begin{lemma}\label{lem:pushforward_along_a_closed_immersion_preserves_n-connectedness}
	Let $ \ilowerstar \colon \fromto{\Z}{\X} $ be a closed immersion of \topoi and $ \phi $ a map in $ \Z $.
	For each $ -2 \leq n \leq \infty $, the following are equivalent:
	\begin{enumerate}
		\item The map $ \phi $ is an $ n $-connected map of $ \Z $.

		\item The map $ \ilowerstar(\phi) $ is an $ n $-connected map of $ \X $.
	\end{enumerate}
\end{lemma}

\begin{proof}
	First we show that (1) $ \Rightarrow $ (2).
	Let $ \jlowerstar \colon \incto{\U}{\X} $ denote the open complement of $ \Z $.
	Since $ \iupperstar $ and $ \jupperstar $ are jointly conservative, by \Cref{lem:n-connectedness_can_be_checked_on_a_conservative_family_of_geometric_morphisms} we need to show that if $ \phi $ is $ n $-connected, then $ \iupperstar\ilowerstar(\phi) $ and $ \jupperstar\ilowerstar(\phi) $ are $ n $-connected.
	Since $ \ilowerstar $ is fully faithful, $ \iupperstar\ilowerstar(\phi) \equivalent \phi $.
	Thus our assumption on $ \phi $ says that $ \iupperstar\ilowerstar(\phi) $ is $ n $-connected.
	Also, $ \jupperstar\ilowerstar $ is constant with value the terminal object, hence $ \jupperstar\ilowerstar(\phi) $ is an equivalence.

	To see that (2) $ \Rightarrow $ (1), note that since $ \phi \equivalent \iupperstar \ilowerstar(\phi) $, the claim immediately follows from the fact that $ \iupperstar $ preserves $ n $-connected maps.
\end{proof}

\begin{lemma}\label{lem:fully_faithful_geometric_morphisms_that_preserve_infty-connectedness_detect_hypercompleteness}
	Let $ \ilowerstar \colon \fromto{\Scal}{\X} $ be a fully faithful geometric morphism of \topoi.
	If $ \ilowerstar $ preserves $ \infty $-connected maps, then an object $ F \in \Scal $ is hypercomplete if and only if $ \ilowerstar(F) \in \X $ is hypercomplete.
\end{lemma}

\begin{proof}
	Since pushforwards preserve hypercompleteness, it suffices to show that if $ \ilowerstar(F) $ is hypercomplete, then $ F $ is hypercomplete. 
	Let $ \phi \colon \fromto{V}{V'} $ be an $ \infty $-connected map of $ \Scal $.
	By assumption, the morphism $ \ilowerstar(\phi) $ is also $ \infty $-connected.
	Since $ \ilowerstar(F) $ is hypercomplete, we deduce that the induced map
	\begin{equation*}
		-\of \ilowerstar(\phi) \colon \fromto{\Map_{\X}(\ilowerstar(V'),\ilowerstar(F))}{\Map_{\X}(\ilowerstar(V),\ilowerstar(F))}
	\end{equation*}
	is an equivalence.
	Since $ \ilowerstar $ is fully faithful, the map
	\begin{equation*}
		-\of \phi \colon \fromto{\Map_{\Scal}(V',F)}{\Map_{\Scal}(V,F)}
	\end{equation*}
	is also an equivalence.
\end{proof}

\begin{proposition}\label{prop:pushforward_along_an_open_immersion_detects_hypercomplteness}
	Let $ \ilowerstar \colon \incto{\Scal}{\X} $ be a locally closed immersion of \topoi.
	Then an object $ F \in \Scal $ is hypercomplete if and only if $ \ilowerstar(F) \in \X $ is hypercomplete.
\end{proposition}

\begin{proof}
	Since pushforwards preserve hypercompleteness, it suffices to show that if $ \ilowerstar(F) $ is hypercomplete, then $ F $ is hypercomplete. 
	By writing $ \ilowerstar $ as the composite of a closed immersion followed by an open immersion, we are reduced to treating the cases where $ \ilowerstar $ is a closed or an open immersion.

	If $ \ilowerstar $ is an open immersion, note that by \Cref{lem:lowersharp_preserves_n-connectedness}, the functor $ \iupperstar $ preserves hypercompletenss.
	Since $ \ilowerstar $ is fully faithful and $ \ilowerstar(F) $ is hypercomplete, we deduce that $ \iupperstar\ilowerstar(F) \equivalent F $ is hypercomplete.
	
	If $ \ilowerstar $ is a closed immersion, then \Cref{lem:pushforward_along_a_closed_immersion_preserves_n-connectedness} shows that $ \ilowerstar $ preserves $ \infty $-connected maps.
	The claim now follows from \Cref{lem:fully_faithful_geometric_morphisms_that_preserve_infty-connectedness_detect_hypercompleteness}.
\end{proof}

\begin{corollary}\label{cor:every_locally_closed_subtopos_of_a_hypercomplete_topos_is_hypercomplete}
	Let $ \ilowerstar \colon \incto{\Scal}{\X} $ be a locally closed immersion of \topoi.
	If $ \X $ is hypercomplete, then $ \Scal $ is hypercomplete.
\end{corollary}

	

We are now ready to show that the hypercompletion of a recollement remains a recollement:

\begin{proposition}\label{prop:hypercompletions_of_recollements}
	Let $ \X $ be \atopos and let $ U \in \X $ be a $ (-1) $-truncated object.
	Write $ \ilowerstar \colon \incto{\X_{\sminus U}}{\X} $ and $ \jlowerstar \colon \incto{\X_{/U}}{\X} $ for the natural geometric morphisms.
	Then:
	\begin{enumerate}
		\item\label{prop:hypercompletions_of_recollements.1} There are natural identifications
		\begin{equation*}
			(\X_{/U})^{\hyp} = (\Xhyp)_{/U} \andeq (\X_{\sminus U})^{\hyp} = (\Xhyp)_{\sminus U}
		\end{equation*}
		as full subcategories of $ \X $.

		\item\label{prop:hypercompletions_of_recollements.2} The functors
		\begin{equation*}
			\iupperstarhyp \colon \fromto{\Xhyp}{(\X_{\sminus U})^{\hyp}} \andeq \jupperstarhyp \colon \fromto{\Xhyp}{(\X_{/U})^{\hyp}}
		\end{equation*}
		exhibit $ \Xhyp $ as the recollement of $ (\X_{\sminus U})^{\hyp} $ and $ (\X_{/U})^{\hyp} $.
	\end{enumerate}
\end{proposition}

\begin{proof}
	For (1), note that the left-hand identification is a special case of \Cref{prop:hypercomplete_objects_of_a_slice}.
	For the right-hand identification, note that \Cref{prop:pushforward_along_an_open_immersion_detects_hypercomplteness} implies that
	\begin{equation*}
		(\X_{\sminus U})^{\hyp} = \Xhyp \intersect \X_{\sminus U}
	\end{equation*}
	as full subcategories of $ \X $.
	Since $ U $ is hypercomplete and $ \Xhyp \subset \X $ is closed under finite products, unpacking definitions we see that
	\begin{equation*}
		\Xhyp \intersect \X_{\sminus U} = (\Xhyp)_{\sminus U} \period
	\end{equation*}

	Finally, (2) is an immediate consequence of (1) and the open-closed recollement associated to a $ (-1) $-truncated object.
\end{proof}

\begin{example}\label{ex:open-closed_recollement_for_hypersheaves_on_topological_spaces}
	Let $ X $ be a topological space and let $ i \colon \incto{Z}{X} $ be a closed subspace with open complement $ j \colon \incto{U}{X} $.
	From \Cref{ex:open-closed_recollement_for_sheaves_on_topological_spaces,prop:hypercompletions_of_recollements}, we deduce that the functors
	\begin{equation*}
		\iupperstarhyp \colon \fromto{\Shhyp(X)}{\Shhyp(Z)} \andeq \jupperstarhyp \colon \fromto{\Shhyp(X)}{\Shhyp(U)}
	\end{equation*} 
	exhibit $ \Shhyp(X) $ as the recollement of $ \Shhyp(Z) $ and $ \Shhyp(U) $. 
\end{example}

In the remainder of this subsection, we use \Cref{prop:hypercompletions_of_recollements} to prove some compatibilities between hypercompletion and pulling back along locally closed immersions.
Note that since the inclusion of hypercomplete \topoi into all \topoi does not preserve limits, these results do not immediately follow from formal considerations.

\begin{corollary}\label{cor:pulling_back_a_locally_closed_immersion_along_hypercompletion}
	Let $ \ilowerstar \colon \incto{\Scal}{\X} $ be a locally closed immersion of \topoi.
	Then the natural square
	\begin{equation*}
		\begin{tikzcd}[sep=3em]
			\Scal^{\hyp} \arrow[r, hooked, "\ilowerstar^{\hyp}"] \arrow[d, hooked] & \Xhyp \arrow[d, hooked] \\ 
			\Scal \arrow[r, hooked, "\ilowerstar"'] & \X
		\end{tikzcd}
	\end{equation*}
	is a pullback square in $ \RTop $.
\end{corollary}

\begin{proof}
	By factoring $ \ilowerstar $ as the composite of a closed immersion followed by an open immersion, it suffices to treat the cases of closed and open immersions separately.
	These cases follow from \enumref{prop:hypercompletions_of_recollements}{1} and the explicit description of the pullbacks along open and closed immersions of \topoi (\Cref{prop:formula_for_pullbacks_along_etale_morphisms_and_closed_immersions}).
\end{proof}

\begin{proposition}\label{prop:hypercompletion_commutes_with_pullback_along_locally_closed_immersions}
	Let
	\begin{equation*}
		\begin{tikzcd}[sep=3em]
			\Scal \arrow[r, hooked, "\ibarlowerstar"] \arrow[d] \arrow[dr, phantom, very near start, "\lrcorner", xshift=-0.25em, yshift=0.25em] & \X \arrow[d] \\ 
			\Tcal \arrow[r, hooked, "\ilowerstar"'] & \Y
		\end{tikzcd}
	\end{equation*}
	be a pullback square of \topoi where $ \ilowerstar $ is a locally closed immersion.
	Then the induced square
	\begin{equation*}
		\begin{tikzcd}[sep=3em]
			\Scal^{\hyp} \arrow[r, hooked, "\ibarlowerstar^{\hyp}"] \arrow[d] & \Xhyp \arrow[d] \\ 
			\Tcal^{\hyp} \arrow[r, hooked, "\ilowerstarhyp"'] & \Yhyp
		\end{tikzcd}
	\end{equation*}
	is also a pullback square in $ \RTop $.
\end{proposition}

\begin{proof}
	Consider the commutative cube of \topoi
	\begin{equation*}
        \begin{tikzcd}[column sep={12ex,between origins}, row sep={8ex,between origins}]
            \Scal^{\hyp} \arrow[rr, "\ibarlowerstar^{\hyp}", hooked] \arrow[dd]  \arrow[dr, hooked] & & \Xhyp \arrow[dd]  \arrow[dr, hooked] \\
            & \Scal \arrow[rr, "\ibarlowerstar"{near start}, hooked, crossing over] & & \X \arrow[dd]  \\
            \Tcal^{\hyp} \arrow[rr, "\ilowerstarhyp"'{near end}, hooked] \arrow[dr, hooked] & & \Yhyp \arrow[dr, hooked] \\
            & \Tcal \arrow[rr, "\ilowerstar"', hooked] \arrow[from=uu, crossing over] & & \Y  \period
        \end{tikzcd}
    \end{equation*}
    By assumption, the front vertical face is a pullback square.
    Since $ \ilowerstar $ and $ \ibarlowerstar $ are locally closed immersions, \Cref{cor:pulling_back_a_locally_closed_immersion_along_hypercompletion} shows that the top and bottom horizontal faces are pullback squares.
    By the gluing lemma for pullbacks, the back vertical face is also a pullback square.
\end{proof}

In general, the functor sending a topological space $ X $ to the \topos \smash{$ \Shhyp(X) $} does not preserve pullbacks. 
However, the assignment \smash{$ \goesto{X}{\Shhyp(X)} $} \textit{does} preserve pullbacks along locally closed immersions: 

\begin{corollary}\label{cor:Shhyp_preserves_pullbacks_along_locally_closed_embeddings}
	Let
	\begin{equation*}
	    \begin{tikzcd}[sep=2.25em]
	       S \arrow[d] \arrow[r, hooked, "\ibar"] \arrow[dr, phantom, very near start, "\lrcorner", xshift=-0.25em, yshift=0.25em] & X \arrow[d, "f"]  \\ 
	       T \arrow[r, hooked, "i"'] & Y
	    \end{tikzcd}
	\end{equation*}
	be a pullback square of topological spaces where $ i $ is a locally closed immersion.
	Then the induced square of \topoi
	\begin{equation*}
	    \begin{tikzcd}[sep=2.25em]
	       \Shhyp(S) \arrow[d] \arrow[r, hooked, "\ibarlowerstar^{\hyp}"] & \Shhyp(X) \arrow[d, "\flowerstar^{\hyp}"]  \\ 
	       \Shhyp(T) \arrow[r, hooked, "\ilowerstarhyp"'] & \Shhyp(Y)
	    \end{tikzcd}
	\end{equation*}
	is a pullback square in $ \RTop $.
\end{corollary}

\begin{proof}
	By \Cref{cor:Sh_preserves_pullbacks_along_locally_closed_embeddings}, the claim is true \textit{before} hypercompletion.
	\Cref{prop:hypercompletion_commutes_with_pullback_along_locally_closed_immersions} shows that the claim remains true after hypercompletion.
\end{proof}


\DeclareFieldFormat{labelnumberwidth}{#1}
\printbibliography[keyword=alph]
\DeclareFieldFormat{labelnumberwidth}{{#1\adddot\midsentence}}
\printbibliography[heading=none, notkeyword=alph]

\end{document}